\documentclass[twoside]{article}

\usepackage{amsmath,amsthm,amsfonts,latexsym,amscd,amssymb,enumerate}
\usepackage[hypertex]{hyperref}
 \usepackage{color}
\definecolor{darkgreen}{cmyk}{1,0,1,.2}
\definecolor{m}{rgb}{1,0.1,1}

\newtheorem{Equation}{}[section]
\newtheorem{example}[Equation]{Example}
\newtheorem{theorem}[Equation]{Theorem}
\newtheorem{proposition}[Equation]{Proposition}
\newtheorem{lemma}[Equation]{Lemma}
\newtheorem{corollary}[Equation]{Corollary}

\newtheorem{definition}[Equation]{Definition}

\newtheorem{remark}[Equation]{Remark}
\newtheorem{question}[Equation]{Question}

\def\pa{\partial}
\def\Dom{\operatorname{Dom}}

\def\Hom{\operatorname{Hom}}
\def\End{\operatorname{End}}

\def\Ind{\operatorname{Ind}}

\def\Tr{\operatorname{Tr}}
\def\tr{\operatorname{tr}}

\def\CC{\mathbb C}

\def\RR{\mathbb R}
\def\S{\mathbb S}
\def\ZZ{\mathbb Z}

\def\ep{\epsilon}

\def\tg{\tilde g}
\def\tP{\tilde P}
\def\tY{\tilde Y}

\def\what{\widehat}

\def\maA{{\mathcal A}}
\def\maQ{{\mathcal Q}}
\def\maS{{\mathcal S}}
\def\maR{{\mathcal R}}
\def\maE{{\mathcal E}}
\def\maU{{\mathcal U}}

\def\maG{{\mathcal G}}
\def\cG{{\mathcal G}}

\def\D{{\mathcal D}}
\def\maS{{\mathcal S}}

\def\maA{{\mathcal A}}
\def\maB{{\mathcal B}}
\def\maE{{\mathcal E}}
\def\maF{{\mathcal F}}

\def\maG{{\mathcal G}}

\def\maP{{\mathcal P}}
\def\maH{{\mathcal H}}
\def\cG{{\mathcal G}}

\def\maS{{\mathcal S}}
\def\K{{\mathcal K}}

\def\tQ{{\tilde Q}}

\def\tR{{\tilde R}}
\def\tN{{\tilde N}}
\def\tK{{\tilde K}}

\def\tk{{\tilde k}}
\def\tK{{\tilde K}}
\def\tE{{\tilde E}}
\def\tB{{\tilde B}}

\def\tf{{\tilde f}}
\def\Nu{{\Xi}}
\def\tM{{\tilde M}}
\def\tm{{\tilde m}}

\newcommand{\KK}{\mathbb{K}}

\DeclareMathOperator{\Log}{Log}
\DeclareMathOperator{\GL}{GL}

\DeclareMathOperator{\ind}{ind}

\newcommand{\forget}[1]{}

\def  \nuint {\raise10pt\hbox{$\nu$}\kern-6pt\int}

\newcommand\tmu{\tilde{\mu}}

\def\N{\mathcal N}
\def \L{\mathcal L}

\def \P{\mathcal P}
\newcommand\E{\mathcal E}
\newcommand\G{\mathcal G}
\newcommand\Q{\mathcal Q}

\newcommand\C{\mathcal C}
\newcommand\A{\mathcal A}
\newcommand\F{\mathcal F}
\newcommand\maK{\mathcal K}
\def \L {{\cal L}}
\def \Sp {{\cal S}}
\newcommand\B{\mathcal B}

\def \H {{\cal H}}

\def\Id{{\rm Id}}

\def\maA{{\mathcal A}}
\def\maB{{\mathcal B}}
\def\maE{{\mathcal E}}
\def\maF{{\mathcal F}}

\def\maG{{\mathcal G}}

\def\maP{{\mathcal P}}
\def\maH{{\mathcal H}}
\def\cG{{\mathcal G}}

\def\maS{{\mathcal S}}

\def\tU{{\tilde U}}
\def\tE{{\tilde E}}

\def\tE{{\tilde E}}
\def\tf{{\tilde f}}
\def\Nu{{\Xi}}
\def\tM{{\tilde M}}
\def\tB{{\tilde B}}
\def\tm{{\tilde m}}

\newcommand\maM{\mathcal M}
\newcommand\tD{\tilde D}
\def \maL {{\cal L}}
\def \maS{{\cal S}}

\def \maH {{\mathcal H}}

\def\Id{{\rm Id}}

\newcommand\maD{\mathcal D}
\newcommand\Di{D\kern-6pt/}
\newcommand\cDi{{\mathcal D}\kern-6pt/}
\newcommand\spi{S\kern-6pt/}
\newcommand \cspi{\Sp\kern-6pt/}

\def \cal {\mathcal}
\def \maC {{\mathcal C}}
\def \maK {{\mathcal K}}

\def \B {\mathbb B}

\newcommand\NN{\mathbb N}

\newcommand\END{\operatorname{END}}
\def\cases{\left\{\begin{array}{l}}
\def\endcases{ \end{array}
\right.}

\begin{document}

\title{Index, eta and rho invariants on foliated bundles\\ \begin{small}Dedicated to Jean-Michel Bismut on the occasion of his sixtieth birthday.\end{small}}
\author{Moulay-Tahar Benameur\\ Universit\'{e} Paul Verlaine, Metz
\and Paolo Piazza\\Universit\`a di Roma  "La Sapienza"}


\maketitle

\begin{abstract}
We study primary and secondary invariants of leafwise Dirac
operators on foliated bundles.
Given such an operator, we begin by considering the
associated regular self-adjoint operator $\D_m$ on
the maximal Connes-Skandalis Hilbert module and  explain how the functional calculus of $\D_m$ encodes both the leafwise calculus and the monodromy calculus in the corresponding von Neumann algebras. When the foliation is endowed with a holonomy invariant transverse measure,
we  explain the compatibility  of various traces and determinants.
We  extend Atiyah's index theorem on Galois coverings
to these foliations.
We define a foliated rho-invariant and investigate
its stability properties for the signature operator.
Finally, we establish the foliated homotopy invariance of such a
signature rho-invariant under a  Baum-Connes assumption, thus
extending to the foliated context results proved by Neumann, Mathai, Weinberger and Keswani on Galois coverings.

\end{abstract}

\tableofcontents

\newpage

\section*{Introduction and main results}\label{sec:intro}

The Atiyah-Singer index theorem on closed compact manifolds is regarded
nowadays as a classic result in mathematics. The original result has branched
into several directions, producing new ideas and new results.
One of these directions consists in considering  elliptic differential  operators on the following
hierarchy of  geometric structures:
\begin{itemize}
\item fibrations and operators that are elliptic in the fiber directions; for example,
a product fibration $M\times T\to T$ and  a family $(D_\theta)_{\theta\in T} $
of elliptic operators on $M$
parametrized by $T$;
\item Galois $\Gamma$-coverings and $\Gamma$-equivariant elliptic operators;
\item measured foliations and operators that are elliptic along the leaves;
\item general foliations and, again, operators that are elliptic along the leaves.
\end{itemize}
One pivotal example, going through all these situations, is the one of foliated
bundles.
Let $\Gamma\to \tM\to M$ be a Galois $\Gamma$-cover of a smooth compact manifold $M$,
let $T$ be a compact manifold on which $\Gamma$ acts by diffeomorphism. 
We can  consider the diagonal action of $\Gamma$ on $\tM\times T$ and the 
quotient space $V:=\tM\times_\Gamma T$, which is a compact manifold,
a bundle over $M$ and carries 
a foliation $\maF$. This foliation is  obtained by considering the images
of the fibers of the trivial fibration $\tM\times T\to T$ under the quotient map
$\tM\times T\to \tM\times_\Gamma T$ and is known as a {\it foliated bundle}. More generally,
we could allow  $T$ to be a compact topological space with an action of $\Gamma$
by homeomorphisms, obtaining what is usually called a {\it foliated
space} or a  {\it lamination}.
We then consider a family of elliptic differential operators $(\tilde{D}_\theta)_{\theta\in T}$
on the product fibration $\tM\times T\to T$ and we assume that it is $\Gamma$-equivariant;
it therefore yields a leafwise differential operator $D=(D_L)_{L\in V/\maF}$
on $V$, which is elliptic along the
leaves of $\maF$. Notice that, if $\dim T >0$ and $\Gamma=\{1\}$ then we are in the family
situation; if $\dim T=0$ and $\Gamma\not=\{1\}$, then we are in the covering situation;
if $\dim T>0$, $\Gamma\not=\{1\}$ and $T$ admits a $\Gamma$-invariant Borel
measure $\nu$, then we are in the measured foliation situation and if $\dim T>0$, $\Gamma\not=\{1\}$ 
then we are dealing with  a more general foliation. 


In the first three cases, there is first of all
a {\it numeric} index: for families this is quite trivially the integral over $T$ of the locally constant
function that associates to $\theta$ the index of $D_\theta$; for $\Gamma$-coverings
we have the  $\Gamma$-index of  Atiyah and for measured foliations we have
the measured index introduced by Connes. These last two examples involve the
definition of a von Neumann algebra endowed with a suitable trace.
More generally, and this applies also to general foliations,
one can define  {\it higher indices}, obtained by pairing the  index class
defined by an elliptic operator with suitable
(higher) cyclic cocycles. In the case of foliated bundles there is a formula for these higher indices, due to Connes \cite{Co},
and recently revisited by Gorokhovsky and Lott \cite{Go-Lo}
using a generalization  of the
Bismut superconnection \cite{Bismut-inventiones}. Since our main focus here are numeric (versus higher) invariants,
we go back to the case of measured foliated bundles, thus assuming that $T$ admits a $\Gamma$-invariant
measure $\nu$.

The index is of course a global object, defined in terms of the kernel
and cokernel of operators. However, one of its essential features is the possibility
of localizing it near the diagonal using 
the remainders produced 
by a parametrix  for $D$.
On a closed compact manifold this crucial property is encoded in the so-called
Atiyah-Bott formula:
\begin{equation}\label{intro-calderon}
\ind (D)= \Tr (R_0^N)-\Tr (R_1^N)\,,\quad \forall N\geq 1
\end{equation}
if $R_1=\Id-DQ$ and $R_0=\Id-QD$ are the remainders of a parametrix $Q$.
Similar results hold in the other two contexts: $\Gamma$-coverings and measured foliations. 
One important consequence
of formula \eqref{intro-calderon} and of the analogous one on $\Gamma$-coverings
is Atiyah's index theorem on a $\Gamma$-covering $\tM\to M$, 
stating the equality of the index on $M$
and the von Neumann $\Gamma$-index on $\tM$. Informally, the index upstairs is equal
to the index downstairs. On a measured  foliation, for example on a  foliated bundle
$(\tM\times_\Gamma T,\maF)$ associated to a $\Gamma$-space $T$ endowed with
a $\Gamma$-invariant measure $\nu$, we also have an index upstairs and an index downstairs,
depending on whether we consider the $\Gamma$-equivariant family  $(\tilde{D}_\theta)_{\theta\in T}$
or the longitudinal operator $D=(D_L)_{L\in V/\maF}$; the analogue of formula \eqref{intro-calderon}
allows to prove the equality of these two indices. (This phenomenon is well known
to  experts; we explain it
in detail in Section \ref{Section.Index}.)

Now, despite its many geometric applications, the index remains
a very coarse spectral invariant of the elliptic differential operator $D$, depending
only on the spectrum near zero. Especially
when considering geometric operators, such as Dirac-type operators, 
and related geometric questions involving, for example, the diffeomorphism type of manifolds
or the moduli space of metric of positive scalar curvature,
one is led to
consider more involved spectral invariants. The eta invariant,
introduced by Atiyah, Patodi and Singer on odd dimensional manifolds,  
is such an invariant.
This invariant 
is highly non-local (in contrast to the index) and 
involves  the whole spectrum of the operator.
It is,  however,  too sophisticated: indeed, a small perturbation of the operator produces 
a variation of the corresponding
eta invariant.
In geometric questions one  considers rather a more stable invariant, the rho
invariant, typically a difference of eta invariants having the same local variation.
The Cheeger-Gromov rho invariant on a Galois covering
$\tM\to M$ of an odd dimensional
manifold $M$ is the most famous example; it is precisely defined as the difference of the 
$\Gamma$-eta invariant on $\tM$, defined using the $\Gamma$-trace of Atiyah, 
and of the Atiyah-Patodi-Singer
eta invariant of the base $M$. 
Notice that the analogous difference for the indices (in the even dimensional
case) would be equal to zero because of Atiyah's index theorem on  coverings; the 
Cheeger-Gromov rho invariant is thus a genuine {\it secondary invariant}.
The Cheeger-Gromov rho invariant is usually defined for
a Dirac-type  operator $\tD$ and we bound ourselves to this case from now on; we denote
it by $\rho_{(2)}(\tD)$. Here are some
of the stability properties of rho:
\begin{itemize}
\item let $(M,g)$ be an oriented riemannian manifold and let $\tD^{{\rm sign}}$ be 
the signature operator on $\tM$ associated 
to the  $\Gamma$-invariant lift of $g$ to $\tM$ : 
then $\rho_{(2)}(\tD^{{\rm sign}})$ is metric independent and a diffeomorphism 
invariant of $M$;
\item let $M$ be a spin manifold and assume that the space $\mathcal{R}^+ (M)$
of  metrics  with positive scalar curvature is non-empty.
Let $g\in \mathcal{R}^+ (M)$ and let
$\tD^{{\rm spin}}_g$ be the spin Dirac operator associated to the $\Gamma$-invariant 
lift of $g$. Then the function $\mathcal{R}^+(M) \ni g \rightarrow \rho_{(2)} (\tD^{{\rm spin}}_g)$
is constant on the connected components of  $\mathcal{R}^+ (M)$
\end{itemize}
There are easy examples, involving lens spaces, showing that $\rho_{(2)}(\tD^{{\rm sign}})$
is {\it not} a homotopy invariant and that $\mathcal{R}^+(M) \ni g \rightarrow \rho_{(2)} (\tD^{{\rm spin}}_g)$
is {\it not} the constant function equal to zero. For purely geometric applications
of these  two results see, for example, \cite{Chang-W} and \cite{Pia-Sch2}.
These two properties can be proved in general, regardless of the nature of the
group $\Gamma$. However, when $\Gamma$ is {\it torsion-free}, then the
Cheeger-Gromov rho invariant enjoys particularly strong stability properties.
Let $\Gamma=\pi_1 (M)$ and let $\tM\to M$ be the universal cover. Then in a series
of papers  \cite{keswani3}, \cite{Kes1}, \cite{Kes2},  Keswani, extending
work of Neumann \cite{Neumann}, Mathai \cite{Mathai-JFA} and Weinberger \cite{Wei}, establishes the following 
fascinating  theorem:
\begin{itemize}
\item if $M$ is orientable, $\Gamma$ is torsion free and the Baum-Connes map
 $K_* (B\Gamma)\to K_* (C^*_{{\rm max}}\Gamma)$ is an isomorphism, then
$\rho_{(2)}(\tD^{{\rm sign}})$ is a {\it homotopy invariant} of $M$;
\item  if $M$ is in addition spin and $\mathcal{R}^+(M)\not= \emptyset$ then
$\rho_{(2)} (\tD^{{\rm spin}}_g)=0$ for any $g\in \mathcal{R}^+(M)$.
\end{itemize}
(The second statement is not explicitly given in the work of Keswani
but it follows from what he proves; for a different proof of Keswani's result
see the recent paper \cite{Pia-Sch1}.)
Informally: {\it when $\Gamma$ is torsion free and the maximal Baum-Connes map
is an isomorphism, the Cheeger-Gromov rho invariant behaves like an index, i.e.
like a primary invariant:
more precisely, it is a homotopy invariant for the signature operator and it is equal to zero for
the spin Dirac operator associated to a metric of positive scalar curvature.}

Let us now move on in the hierarchy of geometric structures and consider a
foliated bundle $(V:=\tM\times_\Gamma T, \maF)$, with $\tM \to M$ the odd universal cover
and $T$ a compact $\Gamma$-space
endowed with a $\Gamma$-invariant (probability) measure $\nu$.
We are also given a $\Gamma$-equivariant family of Dirac-type operators 
$\tD:=(\tD_\theta)_{\theta\in T}$ on the product fibration $\tM\times T\to T$
and let $D=(D_L)_{L\in V/\maF}$ be the induced longitudinally elliptic operator on $V$.
 One is then led to the 
following natural questions:
\begin{enumerate}
\item Can one define a foliated rho invariant $\rho_\nu (D; V,\maF)$?
\item What are its stability properties if $\tD=\tD^{{\rm sign}}$ and $\tD=\tD^{{\rm spin}}$ ?
\item If $\Gamma$ is torsion free and the maximal Baum-Connes map 
with coefficients $$ K^\Gamma_* (E\Gamma;C(T))\rightarrow K_*(C(T)\rtimes_{{\rm max}} \Gamma)$$
is an isomorphism, is   $\rho_\nu (V,\maF) := \rho_\nu (D^{{\rm sign}}; V,\maF) $  a foliated homotopy invariant ?
\end{enumerate}

\bigskip
\noindent
{\it The goal of this paper is to give an answer to these three questions. Along the way we shall
present in a largely self-contained manner  the main results in index theory and in the theory
of eta invariants on foliated bundles.}

\bigskip
This work is organized as follows. In Section \ref{sec:group} we introduce the maximal
$C^*$-algebra $\maA_m$ 
associated to the $\Gamma$-space $T$ or, more precisely, to the groupoid $\maG:=T\rtimes \Gamma$.
We endow this $C^*$-algebra with two  traces $\tau^\nu_{{\rm reg}}$ and  $\tau^\nu_{{\rm av}}$,
$\nu$ denoting a $\Gamma$-invariant measure on $T$.
We then  define two  von Neumann algebras $W^*_{{\rm reg}}(\maG)$, $W^*_{{\rm av}}(\maG)$
with their respective traces; we define representations $\maA_m\to  W^*_{{\rm reg}}(\maG)$, $
\maA_m\to W^*_{{\rm av}}(\maG)$ and show compatibility of the traces involved.

\medskip
In Section \ref{sec:foliated}
we move to foliated bundles, giving the definition, studying
the structure of the leaves, introducing the monodromy groupoid $G$ and the associated
maximal $C^*$-algebra $\maB_m$. We then introduce two 
von Neumann algebras, $W^*_\nu (G)$ and $W^*_\nu (V,\maF)$,
to be thought of as the one upstairs and the one downstairs
respectively,  with their respective traces. We introduce   representations $\maB_m
\to W^*_\nu (G)$, $\maB_m
\to W^*_\nu (V,\maF)$ and define  two compatible traces, also denoted $\tau^\nu_{{\rm reg}}$ and  $\tau^\nu_{{\rm av}}$, on the $C^*$-algebra
$\maB_m$. We then prove an explicit formula for these two  traces on $\maB_m$.
We end Section \ref{sec:foliated} with a proof of the Morita isomorphism
$K_0 (\maA_m) \simeq K_0 (\maB_m)$ and its compatibility with the morphisms
$$\tau^\nu_{{\rm reg},*},\, \tau^\nu_{{\rm av},*}: K_0 (\maA_m)\to \CC\\,\quad 
\tau^\nu_{{\rm reg},*},\, \tau^\nu_{{\rm av},*}: K_0 (\maB_m)\to \CC$$
induced by the two pairs of traces on $\maA_m$ and $\maB_m$ respectively.

\medskip
In Section \ref{sec:modules} we move to 
more analytic questions. We define a natural $\maA_m$-Hilbert module $\maE_m$
with associated  $C^*$-algebra of compact operators $\maK_{\maA_m} (\maE_m)$
isomorphic to $\maB_m$;
 we show how $\maE_m$  encodes both the $L^2$-spaces of the fibers of the product
fibration $\tM\times T\to T$ and the $L^2$-spaces of the leaves of $\maF$.
We then introduce a $\Gamma$-equivariant pseudodifferential calculus,
showing in particular  how $0$-th order operators extend to bounded 
$\maA_m$-linear operators on $\maE_m$ and how negative order operators
extend to compact operators. We then move to {\it unbounded regular} operators,
for example operators defined by a $\Gamma$-equivariant  Dirac family 
$\tD:=(\tD_\theta)_{\theta\in T}$
and study quite carefully the functional calculus associated to such an operator.
We then treat  Hilbert-Schmidt operators and trace class operators in our two
von Neumann contexts and  give sufficient conditions for an operator to be trace class.
We study once again various compatibility issues (this material will be crucial later on).

\medskip
In Section \ref{Section.Index} we introduce, in the even dimensional case, the two indices 
$\ind_\nu^{{\rm up}} (\tD)$,  
$\ind_\nu^{{\rm down}}(D)$ with $\tD= (\tD_\theta)_{\theta\in T}$
and 
$D:=(D_L)_{L\in V/\maF}$, and show the equality
$$
\ind_\nu^{{\rm up}} (\tD) = \ind_\nu^{{\rm down}} (D)
$$
This is the analogue of Atiyah's index theorem on Galois coverings.
We also introduce the relevant index class, in $K_0 (\maB_m)$,
 and show how the von Neumann
indices can be recovered from it and the  two morphisms,
$$\tau^\nu_{{\rm reg},*}: K_0 (\maB_m)\rightarrow \CC\,,\quad \tau^\nu_{{\rm av},*}: K_0 (\maB_m)\rightarrow \CC$$
defined by the traces $\tau^\nu_{{\rm reg}}: \maB_m\to \CC$,  $\tau^\nu_{{\rm av}}: \maB_m\to \CC$.

\medskip
In Section \ref{sec:foliated-rho}  we introduce the two eta invariants 
$\eta^\nu_{{\rm up}} (\tD)$,  
$\eta^\nu_{{\rm down}}(D)$  and, finally, the foliated rho-invariant
$\rho_\nu (D;V,\maF)$
as the difference of the two. This answers the first question raised above.
We end this section establishing  an important link between the rho invariant
and the determinant of certain paths.

\medskip
In Section \ref{sec:stability}  we study the stability properties of the foliated rho invariant,
showing in particular that for the signature operator  it is metric independent
and a foliated diffeomorphism invariant. This answers the second question raised  above.

\medskip
Finally, in Sections \ref{loops+bott},  \ref{sec:keswani} and \ref{sec:proof}
 we prove the foliated homotopy invariance of
the signature rho-invariant under a Baum-Connes assumption, following
ideas of Keswani. In order to keep this paper in a reasonnable size, we establish this result under the additional assumption
that the foliated homotopy equivalence is induced by an equivariant fiber homotopy equivalence
of the  fibration defining the foliated bundle (we call this foliated homotopy equivalences {\it special}).
Thus,  Section \ref{loops+bott} contains preparatory material
on determinants and Bott-periodicity; Section  \ref{sec:keswani} gives a sketch
of the proof of the homotopy invariance and Section \ref{sec:proof} contains the details.
With these three sections we give an answer, at least partially, to  the third question raised above. 
Most of the material explained
in the previous part of the paper goes into the rather complicated proof. Some of our results
are also meant to clarify statements in the work of Keswani.


\medskip
\noindent
{\bf Acknowledgements.} We thank
Paul Baum, James Heitsch, Steve Hurder, Yuri Kordyukov, Eric Leichtnam,
Herv\'e Oyono-Oyono, Ken Richardson, George Skandalis and Stephane Vassout
for interesting discussions. 
We also thank  the two referees for a careful and critical reading of the original
manuscript and for some valuable suggestions.
Most of this reasearch was carried out while the
first author was visiting Universit\`a di Roma La Sapienza and while the second
author was visiting Universit\'e de Metz. Financial support for these visits was
provided by Universit\'e de Metz, under  the program {\it professeurs invit\'es},
the CNR-CNRS bilateral project GENCO (Noncommutative Geometry)
and  the  {\it Ministero Istruzione Universit\`a Ricerca}, Italy,
under the project {\it Spazi di Moduli e Teorie di Lie}.

\section{Group actions}\label{sec:group}

\subsection{The discrete groupoid $\maG$.}
Let $\Gamma$ be a discrete group.
Let $T$ be a compact topological space on which the group $\Gamma$ acts by
homeomorphisms on the left. We shall assume that $T$ is endowed with 
a $\Gamma$-invariant Borel measure $\nu$; this is a non-trivial hypothesis.
Thus $(T,\nu)$ is a compact Borel measured space on which $\Gamma$ acts 
by measure preserving homeomorphisms.
We shall assume that $\nu$ is  a probability measure.  We consider the crossed product groupoid $\maG:=T\rtimes \Gamma$;
thus the set of arrows is $T\times \Gamma$, the set of units is $T$, 
$$
s(\theta, \gamma) = \gamma^{-1} \theta \quad \text{ and } \quad r(\theta, \gamma) = \theta.
$$
The composition law is given by
$$(\gamma^\prime \theta,\gamma^\prime)\circ (\theta,\gamma)=(\gamma^\prime \theta,\gamma^\prime \gamma)\,.$$
We denote by $\maA_c$ the convolution $\star$-algebra of compactly supported continuous functions on $\maG$ and by $L^1(\maG)$ the Banach $\star$-algebra which is the completion of $\maA_c$ with respect to the Banach norm $\|\cdot\|_1$ defined by
$$
\| f\| _1 := {\rm max}\{ \sup_{\theta\in T} \sum_{\gamma\in \Gamma} | f(\theta, \gamma)|; \sup_{\theta\in T} \sum_{\gamma\in \Gamma} | 
f(\gamma^{-1}\theta, \gamma^{-1})|\}
$$
 The convolution operation and the adjunction are fixed by the following formulae
$$
(f*g)(\theta, \gamma) = \sum_{\gamma_1\in \Gamma} f(\theta, \gamma_1) g( \gamma_1^{-1} \theta,\gamma_1^{-1}\gamma) \text{ and } f^* (\theta, \gamma)   = {\overline{f(\gamma^{-1}\theta, \gamma^{-1})}}
$$
For $\theta\in T$ we shall denote by $\Gamma(\theta)$ the isotropy group of the point $\theta$:
$\Gamma (\theta):= \{ \gamma\in \Gamma\,|\, \gamma\theta=\theta\}$.
So, $\Gamma(\theta)$ is a subgroup of $\Gamma$ and the orbit of $\theta$ under the action of $\Gamma$, denoted $\Gamma\theta$,  can be identified with $\Gamma/\Gamma(\theta)$. 
Finally, we recall that $\maG^\theta:= r^{-1}(\theta)$ and that $\maG_\theta:= s^{-1}(\theta)$.

\subsection{$C^*$-algebras associated to the discrete groupoid $\maG$.}
For any $\theta \in T$, we define the regular $*$-representation $\pi^{\rm{reg}}_\theta$ of $\maA_c$ in the Hilbert space $\ell^2(\Gamma)$, viewed as $\ell^2 (\maG_\theta)$,
 by the following formula
$$
\pi^{\rm{reg}}_\theta (f) (\xi) (\gamma) := \sum_{\gamma'\in \Gamma} f(\gamma \theta,  \gamma {\gamma'} ^{-1}) \xi (\gamma').
$$
It is easy to check that this formula defines a $*$-representation $\pi^{\rm{reg}}_\theta$ which is $L^1$ continuous.  Moreover,  we complete $L^1(\maG)$ with respect to the norm $\sup_{\theta\in T} \|\pi^{\rm{reg}}_\theta (\cdot)\|$ and obtain a $C^*$-algebra $\maA_r$. The $C^*$-algebra $\maA_r$ is usually called the regular $C^*$-algebra of the groupoid $\maG$, it will also be denoted with the symbols $C^*_r (\maG)$ or
$C(T)\rtimes_r \Gamma$.

If we complete the Banach $*$-algebra $L^1(\maG)$ with respect to all continuous $*$-representations, then we get the $C^*$-algebra $\maA_m$, usually called the maximal $C^*$-algebra of the groupoid $\maG$. See \cite{Renault} for more details on these constructions. Other notations for $\maA_m$ are $C^*_m (\maG)$ and $C(T)\rtimes_m \Gamma$.

By construction, any continuous $*$-homomorphism from $L^1(\maG)$ to a $C^*$-algebra $B$ yields a $C^*$-algebra morphism from $\maA_m$ to $B$. In particular, the homomorphism $\pi^{\rm{reg}}$  yields  a $C^*$-algebra morphism 
$$
\pi^{\rm{reg}} : \maA_m \longrightarrow \maA_r.
$$

\subsection{von Neumann algebras associated to the discrete groupoid $\maG$.} \label{DiscreteVN}

At the level of measure theory, recall that we have fixed once for all a $\Gamma$-invariant 
borelian probability measure $\nu$ on $T$. We associate with $\maG$ two von Neumann algebras that will be important for our purpose. 

The first one is the regular von Neumann algebra $W^*_{\rm{reg}}(\maG)$. It is the algebra $L^\infty (T, B(\ell^2\Gamma); \nu)^\Gamma$ of $\Gamma$-equivariant essentially bounded families of bounded operators on $\ell^2\Gamma$, so it acts on the Hilbert space $L^2(T\times \Gamma, \nu)$. An element $T$ of $W^*_{\rm{reg}}(\maG)$ is thus (a class of) a  familly $(T_\theta)_{\theta\in T}$ of operators in $\ell^2(\Gamma)$, which satisfies the following properties:
\begin{itemize}
\item For any $\xi \in L^2(T\times \Gamma)$ the map $\theta \mapsto <T_\theta \xi_\theta, \xi_\theta >$ is Borel measurable where $\xi_\theta (\gamma):= \xi(\theta,\gamma)$;
\item $\theta  \mapsto \|T_\theta\|$ is $\nu$-essentially bounded on $T$;
\item For any $\gamma\in \Gamma$, we have $T_{\gamma\theta} = \gamma T_\theta$.
\end{itemize}
Notice that if we denote by $R^*_\gamma:\ell^2\Gamma \to \ell^2\Gamma$ the operator
$$
(R^*_\gamma \xi ) (\alpha) := \xi(\alpha \gamma),
$$
then $\gamma T:= R^*_\gamma \circ T \circ R^*_{\gamma^{-1}}$ for any $T\in B(\ell^2\Gamma)$.
That $W^*_{\rm{reg}}(\maG)$ is a von Neumann algebra is clear since it is the commutant of a unitary group associated with the action of $\Gamma$. The $*$-representation $\pi^{\rm{reg}}$ is then valued in $W^*_{\rm{reg}}(\maG)$ as can be checked easily, and we have the $*$-representation
$$
\pi^{\rm{reg}}: \maA_{\rm{r}} \longrightarrow W^*_{\rm{reg}}(\maG).
$$
This $*$-representation then  extends to the maximal $C^*$-algebra $\maA_m$. 

The second von Neumann algebra that will be important for us will be called the average von Neumann algebra $W^*_{\rm{av}}(\maG)$ and we proceed now to define it. We set $\maG_0:= (T\times \Gamma)/\sim$ where we identify $(\theta, \gamma)$ with $(\theta, \gamma\alpha)$ whenever $\alpha \theta= \theta$.  Then $\maG_0$ is Borel and an element $T$ of $W^*_{\rm{av}}(\maG)$ is (a class of) a  family $(T_\theta)_{\theta\in T}$ of operators in $\ell^2(\Gamma \theta)$, which satisfies the properties:
\begin{itemize}
\item For any measurable (as a function on $\maG_0$) $\nu$-square integrable section $\xi$ of the Borel field $\ell^2(\Gamma/\Gamma (\theta))$ over $T$, the map $\theta \mapsto <T_\theta \xi_\theta, \xi_\theta >$ is Borel measurable where $\xi_\theta [\gamma]:= \xi[\theta,\gamma]$

\item $\theta  \mapsto \|T_\theta\|$ is $\nu$-essentially bounded on $T$;
\item For any $\gamma\in \Gamma$, we have $T_{\gamma\theta} = \gamma T_\theta$;
\end{itemize}
Here we denote by $R_\gamma^*: \ell^2(\Gamma/\Gamma(\theta)) \to \ell^2(\Gamma/\Gamma(\gamma\theta))$ the isomorphism given by $(R_\gamma^*\xi) [\alpha] := \xi [\alpha\gamma]$, and $\gamma T:= R_\gamma^* \circ T \circ R_{\gamma^{-1}}^*$.
Again $W^*_{\rm{av}}(\maG)$ is  a von Neumann algebra; for more details 
on this constructions see for instance \cite {Dixmier1}, \cite{Dixmier2}

There is an interesting representation $\pi^{av}$ of $L^1(\maG)$ in $W^*_{\rm{av}}(\maG)$ defined as follows. Let
$f\in C_c (\maG)$;
for any $\theta \in T$, we set
$$
\pi^{av}_\theta (f) (\xi) (x = [\alpha]) := \sum_{y\in \Gamma/\Gamma(\theta)}\;\; \; \sum_{[\beta]=y}  f(\alpha \theta,\alpha \beta^{-1})  \xi(y), \quad \xi\in \ell^2 (\Gamma/\Gamma(\theta)).
$$

\begin{remark}\label{average-remark}
If we identify $\Gamma/\Gamma(\theta)$ with the orbit $\Gamma\theta$ then $\pi^{av}$ becomes
$$
\pi^{av} _\theta (f) (\xi) (\theta') = \sum_{\theta''\in \Gamma\theta} \sum_{\alpha \theta'' = \theta'} f(\theta', \alpha) \xi(\theta'')=
\sum_{\alpha\in\Gamma} f(\theta^\prime,\alpha)\xi(\alpha^{-1}\theta^\prime)
$$
\end{remark}

\begin{proposition}
For any $f\in L^1(\maG)$ and any $\theta\in T$, the operator $\pi^{av}_\theta (f)$ is  bounded and the  family $\pi^{av} (f)=(\pi^{av}_\theta (f))_{\theta\in T}$ defines a continuous $*$-representation of $L^1(\maG)$ with values in $W^*_{\rm{av}}(\maG)$. Hence, $\pi^{av}$ yields a $*$-representations of the maximal $C^*$-algebra $\maA_m$ in $W^*_{\rm{av}}(\maG)$.
\end{proposition}

\begin{proof} 
If we set for any $f\in C_c(\maG)$, $f_0(\theta,\theta'):=\sum_{\gamma \theta=\theta '} f(\theta ',\gamma)$, then for $g\in  C_c(\maG)$ we have:
\begin{eqnarray*}
(f*g)_0(\theta, \theta') &=& \sum_{\gamma .\theta =\theta'} (f*g)(\theta ', \gamma)\\
&=& \sum_{\gamma\theta =\theta'} \sum_{\gamma_1\in \Gamma} f(\theta', \gamma_1) g(\gamma_1^{-1}\theta', \gamma_1^{-1} \gamma)\\
&=& \sum_{\theta''\in \Gamma .\theta} \;\;\sum_{\gamma^{-1}_1.\theta'=\theta'' \;,\;\gamma_2^{-1}.\theta''=\theta} f(\theta',\gamma_1) g(\theta'', \gamma_2)\\
&=& \sum_{\theta''\in \Gamma .\theta} f_0(\theta'', \theta') g_0(\theta,\theta'')\\
&=& (f_0*g_0) (\theta, \theta').
\end{eqnarray*}
Since $\pi^{av} (f)$ is simply convolution by the kernel $f_0$, we deduce that $\pi$ is a representation of the convolution algebra $\maA_c$. Now, the kernel $(f^*)_0$ is given by
$$
(f^*)_0 (\theta, \theta') = \sum_{\gamma\theta =\theta'} {\overline f}(\gamma^{-1}\theta ', \gamma^{-1})  = \sum_{\alpha \theta' =\theta} {\overline f}(\theta , \alpha) = {\overline{f_0(\theta',\theta)}}.
$$
It remains  to prove that $\pi^{av}$ is $L^1$-continuous. But, we have:
\begin{eqnarray*}
\|\pi^{\rm{av}}_\theta(f)\xi\|_2^2 & = & \sum_{\theta'\in \Gamma\theta} |\sum_{\gamma\in \Gamma} f(\theta', \gamma) \xi(\gamma^{-1}\theta') |^2 \\
&\leq &  \sum_{\theta'\in \Gamma\theta} (\sum_{\gamma\in \Gamma}  |f(\theta',\gamma)| ) \times (\sum_{\gamma\in \Gamma} |f(\theta',\gamma)| . |\xi(\gamma^{-1} \theta')|^2)\\
&\leq & \|f\|_1 \sum_{\theta'\in \Gamma\theta} \sum_{\gamma\in \Gamma} |f(\theta',\gamma)| . |\xi(\gamma^{-1} \theta')|^2\\
&\leq & \|f\|_1 \sum_{\gamma\in \Gamma} \sum_{\theta''\in \Gamma\theta} |\xi(\theta'')|^2 |f(\gamma \theta'', \gamma)|\\
&\leq & \|f\|_1^2 \|\xi\|_2^2.
\end{eqnarray*}
So, $\|\pi^{\rm{av}} (f)\| = \sup_{\theta\in T} \|\pi^{\rm{av}}_\theta (f) \| \leq \|f\|_1$. 
\end{proof}

We therefore deduce the existence of a $*$-homomorphism of $C^*$-algebras:
$$
\pi^{\rm{av}} : \maA_m \longrightarrow W^*_{\rm{av}}(\maG).
$$

\subsection{Traces}

For any non negative element $T= (T_\theta)_{\theta\in T}$ of the von Neumann algebra $W^*_{\rm{reg}}(\maG)$ (resp. $W^*_{\rm{av}}(\maG)$), we set
$$
\tau^\nu (T) := \int_T <T_\theta (\delta_e), \delta_e > d\nu(\theta),
$$
where in the regular case, $\delta_e$ stands for the $\delta$ function at the unit $e$ of $\Gamma$, while in the second case it is the $\delta$ function of the class $[e]$ in $\Gamma/\Gamma(\theta)$. 

\begin{proposition}
The functional $\tau^\nu$ induces a faithful normal positive finite trace.
\end{proposition}

\begin{proof}
Positivity is clear since $T$ is non negative in the von Neumann algebra if and only if for $\nu$-almost every $\theta$ the operator $T_\theta$ is non negative.  If the non negative element $T= (T_\theta)_{\theta\in T}$ satisfies $\tau^\nu (T)=0$ then $<T_\theta (\delta_e), \delta_e > = 0$ for $\nu$-almost every $\theta$. But, the $\Gamma$-equivariance of $T$ implies that 
$$
<T_\theta (\delta_\gamma), \delta_\gamma > = 0, \quad \forall \gamma\in \Gamma \text{ and } \nu \;\;a.e.
$$
Therefore, $T_\theta = 0$ for $\nu$-almost every $\theta$ and hence $T=0$ in $W^*_{\rm{reg}}(\maG)$. In the second case, the proof is similar again by $\Gamma$-equivariance and by replacing $\delta_\gamma$ by $\delta_{[\gamma]}$.

If $T(n) \uparrow T$ is an increasing sequence of non negative operators which converges in the von Neumann algebra to $T$, then for $\nu$-almost every $\theta$, the sequence $T(n)_\theta$ increases to $T_\theta$. But then since the state $<\cdot (\delta_e), \delta_e >$ is normal, the conclusion follows by Beppo-Levi's property for $\nu$.

If now $T$ is in the von Neumann algebra $W^*_{{\rm reg}} (\maG)$ then writing  $T_\theta$ as an infinite matrix in $\ell^2\Gamma$ and using the $\Gamma$ equivariance we deduce that 
$$
T_{\gamma \theta}^{\alpha, \beta} = T_\theta ^{\alpha\gamma, \beta\gamma}.
$$
If we now consider a second operator $S$ in $W^*_{{\rm reg}} (\maG)$, then we have
$$
(T_\theta S_\theta)^{e, e} =\sum_{\gamma\in \Gamma} T^{e,\gamma}_\theta\;S^{\gamma,e}_\theta= \sum_{\gamma\in \Gamma} \;S_{\gamma\theta}^{e,\gamma^{-1}} T_{\gamma\theta}^{\gamma^{-1}, e},
$$
by   the $\Gamma$-equivariance property.  The $\Gamma$-invariance of measure $\nu$ can now be applied
to yield  that $\tau^\nu (TS)=\tau^\nu(ST)$. 
A similar proof works for the von Neumann algebra $W^*_{{\rm av}}(\maG)$.
\end{proof}

We define the functionals $\tau^\nu_{\rm{reg}}$ and $\tau^\nu_{\rm{av}}$ 
on $\maA_c$  by setting for $f\in \maA_c$
\begin{equation}\label{trace-regular}
\tau^\nu_{\rm{reg}} (f) := \int_T f(\theta, e) d\nu(\theta),
\end{equation}
\begin{equation}\label{trace-average
}\tau^\nu_{\rm{av}} (f):= \int_T \left[ \sum_{g\in \Gamma(\theta)} f(\theta, g)\right] d\nu(\theta).
\end{equation}

\begin{lemma}
\begin{enumerate}
\item We have $\tau^\nu \circ \pi^{\rm{reg}} = \tau^\nu_{\rm{reg}}$ and $\tau^\nu \circ \pi^{\rm{av}} = \tau^\nu_{\rm{av}}$. 
\item Hence, $\tau^\nu_{\rm{reg}}$ and $\tau^\nu_{\rm{av}}$ extend to finite traces on $\maA_r$ and $\maA_m$.
\end{enumerate}
\end{lemma}

\begin{proof}\ 
The statement for the regular trace is classical and we thus omit the (easy) proof. We  consider for any $f\in L^1(\maG)$ the Borel family of operators $(\pi^{\rm{av}}_\theta (f))_{\theta\in T}$  defined in the previous paragraph.

For any $f\in \maA_c$, denote as before by $f_0$ the function 
$$
f_0(\theta,\theta'):=\sum_{\gamma \theta=\theta '} f(\theta ',\gamma).
$$ 
Then we know that $\pi^{\rm{av}}(f)$ is given as convolution with $f_0$. If $f\in \maA_c$, then we have,
using the identification $\Gamma/\Gamma(\theta)\equiv \Gamma \theta$:
\begin{eqnarray*}
\int_T < \pi^{\rm{av}}_\theta (f) \delta_\theta , \delta_\theta > d\nu (\theta)& = & \int_T f_0(\theta, \theta) d\nu (\theta)\\
& = & \int_T \sum_{\gamma\in \Gamma(\theta)} f(\theta, \gamma) d\nu(\theta)\\
& = & \tau^\nu_{\rm{av}} (f).
\end{eqnarray*}
\end{proof}


As a Corollary of the above Lemma notice that the traces $\tau^\nu_{\rm{reg}}: \maA_r\to \CC$
and $\tau^\nu_{\rm{av}}:\maA_m\to \CC$ induce group homomorphisms
\begin{equation}\label{traces-k-discrete}
\tau^\nu_{\rm{reg},*}: K_0 (\maA_r)\rightarrow \RR\;,\quad \tau^\nu_{\rm{av},*}: K_0(\maA_m)\rightarrow \RR
\end{equation}

\section{Foliated spaces}\label{sec:foliated}

\subsection{Foliated spaces}

Let $M$ be a compact manifold without boundary and let $\Gamma$ denote its fundamental group and $\tM$ its universal cover. The group $\Gamma$ acts by homemorphisms on the compact topological space $T$ and hence 
acts on the right, freely and properly, on the  space $\tM \times T$ by the formula
$$
(\tm , \theta) \gamma := (\tm \gamma, \gamma^{-1} \theta), \quad (\tm, \theta)\in \tM\times T \text{ and } \gamma\in \Gamma.
$$
The quotient space of $\tM\times T$ under this action is denoted by $V$. We assume as before 
the existence of a $\Gamma$-invariant probability measure $\nu$.
If we want to be specific about the action of $\Gamma$ on $T$ we shall
consider it as a homomorphism $\Psi: \Gamma\to {\rm Homeo}(T)$. We do not assume
the action to be locally free \footnote{Recall that
an action is locally free if given $\gamma\in \Gamma$ and open set $U$
in $T$ such that $\gamma (\theta)=\theta$ for any $\theta\in U$ then $\gamma=1$.} .


If $p:\tM\times T \to V$ is the natural
projection then the leaves of a lamination on $V$ are given by the projections $L_\theta = p(\tM_\theta)$, where $\theta$ runs through the compact space $T$, and 
\begin{equation}\label{eq:mtheta}
\tM_\theta:= \tM\times \{ \theta\}\,.
\end{equation}
It is easy to check that this is a lamination of $V$ with  smooth leaves and possibly complicated transverse structure according to the topology of $T$, see for instance \cite{BenameurOyono}. By definition, it is easy to check that the leaf $L_\theta$ coincides with the leaf $L_{\theta'}$ if and only if $\theta'$ belongs to the orbit $\Gamma \theta$ of $\theta$ under the action of $\Gamma$ in $T$. We shall refer to this lamination by $(V,\maF)$ and {sometimes shall  call it a foliated space
or, more briefly, a foliation.}  If $\Gamma(\theta)$ is the isotropy group of $\theta\in T$ then we see from the definition of $L_\theta$ that $L_\theta$ is diffeomorphic to the quotient manifold 
$ \tM/\Gamma(\theta)$
through the map $L_\theta \rightarrow \tM/\Gamma(\theta)$ given by 
$[\tm^\prime,\theta^\prime]\rightarrow [\tm^\prime \gamma]$, if $\theta^\prime=\gamma\theta$.
Note however that $L_\theta$ is also diffeomorphic to $\tM /\Gamma(\theta')$ for any $\theta'\in \Gamma\theta$. 
Moreover the {monodromy} cover of a leaf $L$  is obtained by choosing $\theta\in T$ such that $L=L_\theta$ and by using the composite map
$$
\tM \to \tM_\theta \to (\tM_\theta)/\Gamma(\theta) \simeq L_\theta = L.
$$
which is a monodromy cover of $L$ corresponding to  $\theta$. 


Notice that the set of  $\theta\in T$ for which $\Gamma(\theta)$
is non-trivial has in general  positive measure. This is the case, for instance,  when 
there exists a subgroup $\Gamma_1$
of $\Gamma$ whose action on $T$ has the property that $\nu (T^{\Gamma_1})>0$, where  $T^{\Gamma_1}$
is the fixed-point subspace defined by $\Gamma_1$. In fact, one can construct simple examples where
the measure of the  set of  $\theta\in T$ for which $\Gamma(\theta)$
is non-trivial is any value in $(0,1)$. See Example \ref{example 3} for a specific situation.

\begin{example}\label{example 1}
{As an easy example where this situation occurs naturally, 
consider any Galois covering $\tM^\prime$
of $M$ with structure group $\Gamma^\prime$ such that
$\pi_1 (\tM^\prime)\not= 1$. Assume the existence of a locally free $\Gamma^\prime$-action
$\Psi^\prime: \Gamma^\prime\to  {\rm Homeo}(T)$
on $T$ and let $V$ be the resulting foliated space. Assume the existence of an invariant measure $\nu$ on $T$.
Since $\Gamma^\prime$ is a quotient of $\Gamma:=\pi_1(M)$
we have a natural group homomorphism $\pi: \Gamma\to\Gamma^\prime$ and thus an action $\Psi:= \Psi^\prime \circ \pi$
of $\Gamma$ on $T$. By definition $\nu$ is also $\Gamma$-invariant.
The isotropy group of this action at $\theta\in T$ is at least as big as the fundamental group of
 $\tM^\prime$. Notice that one can show that 
 $$ (\tM \times T)/\Gamma \,=  (\tM^\prime \times T)/\Gamma^\prime \equiv V$$
 Summarizing: $V$ is a lamination where the set of leaves with non-trivial monodromy has measure equal
 to $\nu(T)=1$.}
 
 \end{example}

\begin{example}\label{example 3}
{Take $M$ to be any manifold whose fundamental group is a free product of copies of $\ZZ$, for example a connected sum of $\S^1\times \S^2$'s, 
so that  now  $\Gamma$ is the free group of rank $k$.
Let $\{\gamma_1,\gamma_2,\ldots, \gamma_k\}$ be the generators. Let $T$ be $\S^2$.
Let $C\subset \S^2$ be a parallel and let $U\subset \S^2$
one of the two hemispheres bounded by $C$.
 Let $\Psi(\gamma_1)$ be any measure-preserving diffeomorphism  of $\S^2$ that 
fixes  $U$. We then define $\Psi$ on the other generators 
 in an arbitrary measure-preserving way. Then any point $\theta$ in $U$ would have  nontrivial isotropy group $\Gamma(\theta)$. Clearly, one can  jazz up  this example  by selecting any $T$ and finding a single homeomorphism whose fixed point set is a set of nonzero measure. 
}
\end{example}

\begin{example}\label{example 2}
{Following \cite{MS} 
we now give an example of a lamination with the set of leaves with non-trivial monodromy
 of positive measure and, in addition, of a rather complicated sort. 
 Take a (generalized) Cantor set $K$ of positive Lebesgue measure in the unit circle.
 Choose now a
homeomorphism $\phi$ of the circle admitting $K$ as the fixed point set. Let $M$ be any
closed odd dimensional manifold with $\pi_1 (M)=\ZZ$.
Consider the foliated space $V$ obtained by suspension of $\phi$:
thus $V=\tM\times_\ZZ \S^1$ with $\ZZ=\pi_1(\S^1)$ acting on $\S^1$
via $\phi$ and acting by deck transformations on $\tM$.
The set of $\theta\in S^1$ such that
$\{\gamma\in \ZZ | \gamma\theta=\theta\}$
is non-trivial is  equal to
$K$, hence it has positive measure.
Using  \cite{MS} page 105/106,
we can find a  Radon $\phi$-invariant measure $\nu$
on $\S^1$ and $\nu (K)>0$.
\\
Notice that in this class of examples, although the measure
is diffuse, one  can even ensure that the set of leaves with non-trivial holonomy
has positive transverse measure. These laminations show up in the study of aperiodic tillings and especially of quasi-crystals. In \cite{BenameurOyono} for instance, the measured foliated index for such laminations, a primary invariant, is used to solve the gap-labelling conjecture. The authors expect potential applications of the foliated rho invariant to aperiodic solid physics.}


\end{example}



\bigskip

\subsection{ The monodromy  groupoid and the $C^*$-algebra of  the foliation}
\label{sub:monodromy}

Let $\tM$, $\Gamma$ and $T$ be as before.
We define the monodromy groupoid $G$
as the quotient space   $(\tM\times \tM\times T) /\Gamma$
of $\tM\times \tM\times T$ by the right diagonal action  
$$(\tm,\tm^\prime,\theta) \gamma := (\tm\gamma,\tm^\prime\gamma,\gamma^{-1}\theta).$$
The groupoid structure is clear:
 the space of units $G^{(0)}$  is the space $V=\tM\times_\Gamma T$,
the source and range
 maps are given by 
$$
s[\tm, \tm ', \theta] = [\tm ', \theta] \text{ and } r[\tm, \tm ', \theta] = [\tm , \theta],
$$
where the brackets denote equivalence classes modulo the action of the group $\Gamma$

It is not difficult to show that $G$ can be identified in a natural way with 
the usual monodromy groupoid 
associated to the foliated space $(V,\maF)$, as defined, for example, in 
\cite{phillips-homotopy}. More precisely given a smooth path $\alpha:[0,1]\to L$, with $L$ a leaf,
choose any  lift $\tilde{\beta} :[0,1]\to \tM$ of the projection of the path $\alpha$ in $M$
through the natural projection $V\to M$. Then there exists a unique $\theta\in T$
with $\alpha(0)=[\tilde{\beta}(0),\theta]$ and we obtain in this way a well defined  element $[\tilde{\beta}(0),\tilde{\beta}(1),\theta]$ of $G$ 
 which only depends on the leafwise homotopy class
of $\alpha$ with fixed end-points. This furnishes the desired isomorphism.


We  fix now a Lebesgue class measure $dm$ on $M$ and the corresponding $\Gamma$-invariant measure $d\tm$ on $\tM$. We denote by $\maB_c$ the convolution $*$-algebra of continuous compactly supported functions on $G$. For  $f,g\in \maB_c$ we have:
$$
(f*g) [\tm, \tm ', \theta] = \int_\tM f[\tm, \tm'', \theta] g[\tm'', \tm', \theta] d\tm'' \;\; \text{ and } \;\;f^*[\tm, \tm ', \theta] = {\overline{f [\tm ', \tm , \theta]}}.
$$
More generally, let $E$ be a hermitian  continuous longitudinally smooth vector bundle over $V$; thus 
$E$ is a continuous bundle over $V$ such that its restriction to each leaf is smooth \cite{MS}.
Consider $\operatorname{END} (E):= (s^* E)^* \otimes (r^* E)={\rm Hom}(s^* E, r^* E)$,
a bundle of endomorphisms
over $G$.
We consider $\maB_c^E := C^{\infty,0}_c (G, \END (E))$ the space of continuous longitudinally smooth
sections of $\END (E)$;
 this  is also
 a $\ast$-algebra with product and adjoint given by
 $$
 (f_1 \ast f_2)[\tm,\tm ', \theta] = \int_{\tM} f_1[\tm, \tm '', \theta] \circ 
 f_2 [\tm '', \tm', \theta] \, d\tm'' ,
 $$
 $$
 f^\ast [\tm, \tm', \theta] = (f[\tm', \tm, \theta])^\ast.
 $$
 Let $\widehat{E}$ be its lift to $\widetilde{M}\times T$;
denote by $H_\theta$ the Hilbert space 
$
H_\theta=L^2(\widetilde{M}\times \{\theta\}; \widehat{E}_{|\widetilde{M}\times \{\theta\}}).
$ 
Any $f\in \maB^E_c$ can be viewed as a smooth kernel acting on $H_\theta$ by the formula
$$
\pi^{{\rm reg}}_\theta (f) (\xi) (\tm) := \int_\tM f[\tm, \tm', \theta] (\xi (\tm')) d\tm', \quad \text{for any}\;\;\xi\in H_\theta\,,
$$
and this defines a  $*$-representation $\pi^{reg}_\theta$ in $H_\theta$ . We point out that the representation $\pi^{reg}_\theta$ is continuous for the $L^1$ norm defined by:
$$
\|f\|_1 := {\rm max}\{\sup_{(\tm,\theta)\in \tM \times T} \int_\tM \|f[\tm, \tm', \theta]\|_{E} d\tm'\,; 
\sup_{(\tm,\theta)\in \tM \times T} \int_\tM \|f[\tm^\prime, \tm, \theta]\|_{E} d\tm' \}
$$
If we complete 
$\maB_c$ with respect to 
the $C^*$ norm
$$
\|f\|_{\rm{reg}} := \sup_{\theta\in T} \|\pi^{reg}_\theta (f)\|,
$$ 
then we get $\maB^E_r$ , the regular $C^*$-algebra of the groupoid $G$ with coefficients in $E$.
 When $E=V\times \CC$ then we denote this $C^*$-algebra simply by $\maB_r$
In the same way, if we complete  $\maB_c$ with respect to all  $L^1$ continuous $*$-representations,
then we get the maximal $C^*$-algebra of the groupoid $G$, that will be denoted by $\maB^E_m$ and 
simply by $\maB_m$ when $E=V\times \CC$.

\subsection{von Neumann Algebras of foliations.}

The material in this paragraph is classical;  for more details see  for instance \cite{Co}, \cite{Hei-Laz}  \cite{BenameurFack}, \cite{Co-LNM},
\cite{Lenz-et-al}.

The representation $\pi^{\rm{reg}}$ defined above takes value in the regular von Neumann algebra of the groupoid $G$. 
More precisely, the regular von Neumann algebra $W^*_{\nu}(G;E)$ of  $G$ with coefficients in $E$, acts on the Hilbert space $H= L^2(T \times\tM, {\widehat E} ; \nu\otimes d\tm)$, and is  by definition the space of families $(S_\theta)_{\theta\in T}$ of bounded operators on $L^2(\tM, {\widehat E})$ such that
\begin{itemize}
\item For any $\gamma\in \Gamma$, $S_{\gamma \theta} = \gamma S_{\theta}$ where $\gamma S_{\theta}$ is defined using the action of $\Gamma$ on the equivariant vector bundle ${\widehat E}$;
\item The map $\theta \mapsto \|S_\theta\|$ is $\nu$-essentially bounded on $T$; 
\item For any $(\xi, \eta) \in H^2$, the map $\theta \mapsto <S_\theta (\xi_\theta), \eta_\theta >$ is  Borel measurable.
\end{itemize}

The von Neumann algebra $W^*_{\nu}(G;E)$ is a type II$_\infty$ von Neumann algebra as we shall see later. It is easy to see that for any $S\in \maB^E_r$, the operator $\pi^{\rm{reg}} (S)$ belongs to $W^*_{\nu}(G;E)$. 

In the same way we define a leafwise von Neumann algebra that we shall denote by $W^*_\nu (V,\maF;E)$;
this algebra  acts on the Hilbert space  \cite{Dixmier1} $
H=\int^{\oplus} L^2(L_\theta, E|_{L_\theta}) d\nu (\theta)$
where $L_\theta$ is, as before, the leaf  in $V$ corresponding to $\theta$.  
Equivalently, and using the identification of the leaves with quotient of $\tM$ under isotropy, $W^*_\nu (V,\maF;E)$ can be described as the set of families $(S_\theta)_{\theta\in T}$ of bounded operators on $(L^2(\tM_\theta/\Gamma(\theta), E|_{\theta})_{\theta\in T}$ such that 
\begin{itemize}
\item The map $\theta \mapsto \|S_\theta\|$ is $\nu$-essentially bounded on $T$.
\item For any square integrable sections $\xi, \eta$ of the Borel field $(L^2(\tM_\theta/\Gamma(\theta), E_\theta))_{\theta\in T}$, the map $\theta \mapsto <S_\theta (\xi_\theta), \eta_\theta >$ is Borel measurable.
\item $S_{\gamma\theta} = \gamma S_\theta$, for any $(\theta, \gamma)\in T\times \Gamma$.
\end{itemize}
Notice that $\Gamma (\gamma\theta)=\gamma \Gamma (\theta) \gamma^{-1}$ and hence the definition of $\gamma S_\theta$ is clear.

\begin{proposition}
There is a well defined representation $
\pi_{\rm{av}}$  from the maximal $C^*$algebra $\maB^E_m$ to
the leafwise von Neumann algebra $W^*_\nu (V,\maF;E)$
such that 
for $f\in C_c (G,\END(E))$ the operator
$(\pi_{\rm{av}}(f))_\theta $
is given by the kernel
$$
f_0(x,y)=0 \text{ if } L_x\neq L_y \text{ and } f_0([\tm, \theta], [\tm',\theta]) := \sum_{\gamma\in \Gamma (\theta)} f [\tm, \tm'\gamma, \theta].
$$

\end{proposition}

\begin{proof}
For simplicity we take $E$ the product line bundle.
For $f\in C_c (G)$ the formula
$$
\left( (\pi_{\rm{av}} (f))_\theta  \;\xi \right) (x) := \int_{L_\theta} f_0(x,y) \xi(y) dy,\quad \xi\in L^2(L_\theta), x\in L_\theta \subset V.
$$
defines a bounded operator on $L^2 (L_\theta)$. Indeed  
the sum on  the RHS in the definition of $f_0$ is finite since $f$ is compactly supported. 
Moreover, when restricted to the leaf $L_x$ the kernel $f_0$ is  supported within a uniform neighborhood of the diagonal of $L_x$. 
We  have:
\begin{eqnarray*}
\|\pi_{\rm{av}} (f))_\theta \, (\xi)\|_2^2 & = & \int_{L_\theta} | \int_{L_\theta} f_0 (x,x')\xi(x') dx'|^2 dx\\
& \leq & \int_{L_\theta} \left(\int_{L_\theta} |f_0 (x,x')| dx'\right) \left( \int_{L_\theta} |f_0 (x,x')| |\xi(x')|^2 dx'\right) dx\\
& \leq & \| f_0\|_{1} \int_{L_\theta} |\xi (x')|^2 \int_{L_\theta} |f_0(x,x')| dx dx'\\
& \leq & \| f_0\|_{1} ^2  \|\xi\|_2^2.
\end{eqnarray*}
Here $\| f_0\|_{1}$ stands for the $L^1$-norm 
$$
{\text{Max}} (\sup_{x'\in V} \int_{L_x} |f_0 (x,x')| dx, \sup_{x\in V} \int_{L_x} |f_0 (x,x')| dx').
$$
Therefore, we have
$$
\sup_{\theta\in T} \| \pi_{\rm{av}} (f) \| \leq \|f_0\|_{1}.
$$
But now it is easy to check  that $\|f_0\|_{1} \leq \|f\|_{1}$. 
On the other hand  $\pi_{{\rm av}}$ is a $*$-representation;
since for  $f,g\in C_c(G)$ one has, with proof similar to the one given for
$\maG$,
$$
(f*g)_0 = f_0 * g_0 \text{ and } (f^*)_0 = (f_0)^*.
$$
To sum up, these arguments  prove that $\pi_{ av}$ on $\maB_c$  extends to
a continuous $*$-representation of the $\maB_m$ in the von Neumann algebra $W^*_\nu (V,\maF)$. 
This completes the proof.
\end{proof}

\subsection{Traces}\label{subsec:traces}

We fix once and for all a fundamental domain $F$ for  the free and proper action of $\Gamma$ on $\tM$.
Let $\chi$ be  the characteristic function of $F$.
Then we set 
 for  any non-negative element $S \in  W^*_{\nu}(G;E)$, 
 $$
\tau^\nu (S) :=  \int_T \tr(M_\chi \circ S_\theta \circ M_\chi) d\nu (\theta),
$$
where $\tr$ is the usual trace of a non-negative operator on a Hilbert space.

We shall also denote by $\chi$ the induced function $\chi\otimes 1_T$, i.e. the characteristic function
of $F\times T$ in $\tM\times T$.
 Since $F\times T$ is a fundamental domain for the free and proper action of $\Gamma$ on $\tM\times T$, we shall also denote by $\chi_\theta$  the same function $\chi$ but viewed as the characteristic function of $F$  inside a given leaf $L_\theta$, which is  the
 image under the projection $\tM\times T\to V$ of $\tM\times \{\theta\}$.
We define
 a functional $\tau^\nu$ on the leafwise  von Neumann algebra $W^*_\nu (V,\maF;E)$,
  by setting for any non-negative element $S\in W^*_\nu (V,\maF;E)$ 
$$
\tau^\nu_{\maF} (S) :=  \int_T \tr(M_{\chi_\theta} \circ S_\theta \circ M_{\chi_\theta}) d\nu (\theta),
$$
where the $M_{\chi_\theta}$ appearing in the integrand is the multiplication  operator in the $L^2$ space of sections over $\tM_\theta/\Gamma(\theta)$,  by the characteristic function $\chi_\theta$ of $F$ viewed in $\tM_\theta/\Gamma(\theta)$.

 \begin{proposition}\label{traces}
With the above notations we have:
 \begin{itemize}
\item  the functional $\tau^\nu$ yields  a positive semifinite normal faithful trace on 
 $W^*_{\nu} (G,E)$;
 \item the functional $\tau^\nu_{\maF}$ yields a positive  semifinite normal faithful trace on 
 $W^*_{\nu} (V,\maF;E)$.
  \end{itemize}
  \end{proposition}

 \begin{proof}
If $R=S^*S \in W^*_\nu (G;E)$, then for any $\theta\in T$,
$$
M_\chi \circ R_\theta \circ M_\chi = (S_\theta M_\chi)^* (S_\theta M_\chi) \geq 0.
$$
Therefore, $\tr(M_\chi \circ R_\theta \circ M_\chi) \geq 0$ and hence $\tau^\nu (R) \geq 0$. Moreover, $\tau^\nu (R) = 0$ if and only if $M_\chi R_\theta M_\chi = 0$ for $\nu$-almost every $\theta$. The $\Gamma$ equivariance of $R$ implies  the relations 
$$
M_{\gamma_1\chi} R_{\gamma\theta} M_{\gamma_2\chi} = U_\gamma \left[ M_{\gamma^{-1} \gamma_1\chi} R_{\theta} M_{\gamma^{-1} \gamma_2\chi}\right] U_{\gamma^{-1}}, \quad \gamma, \gamma_1, \gamma_2\in \Gamma.
$$
The same relations hold for $S$. In particular,
$$
M_{\gamma\chi} R_{\theta} M_{\gamma\chi} = U_\gamma \left[ M_{\chi} R_{\gamma^{-1}\theta} M_{\chi}\right] U_{\gamma^{-1}} = 0.
$$
Since $\nu$ is $\Gamma$-invariant, we deduce that $M_{\gamma\chi} R_{\theta} M_{\gamma\chi} = 0$ $\nu$ almost everywhere. Thus
$$
\sum_{\gamma'\in \Gamma} (M_{\gamma'\chi} S_\theta M_{\gamma\chi})^* (M_{\gamma'\chi} S_\theta M_{\gamma\chi}) = M_{\gamma\chi} R_{\theta} M_{\gamma\chi} = 0, \quad \nu-\text{a.e.} \;\theta \in T.
$$
As a consequence, we get that for $\nu$ almost every $\theta\in T$  and for any $\gamma,\gamma'\in \Gamma$,
$$
M_{\gamma'\chi} S_\theta M_{\gamma\chi} = 0,
$$
which proves that $S=0$ in  $W^*_{\nu}(G,E)$ and whence $R=0$ in  $W^*_{\nu}(G,E)$. 
On the other hand for any non negative $A,B\in W^*_\nu (G;E)$, we have 
\begin{eqnarray*}
 M_\chi A_\theta B_\theta M_\chi & = & \sum_{\gamma\in \Gamma} M_\chi A_\theta M_{\gamma\chi} B_\theta M_\chi \\
& = &  \sum_{\gamma\in \Gamma} M_\chi A_\theta (U_\gamma M_{\chi} U_{\gamma^{-1}})B_\theta M_\chi\\
& = & \sum_{\gamma\in \Gamma} M_\chi U_\gamma A_{\gamma^{-1}\theta} M_{\chi}  B_{\gamma^{-1}\theta} U_{\gamma^{-1}} M_\chi\\
& = & \sum_{\gamma\in \Gamma} U_\gamma \left[M_{\gamma^{-1}\chi} A_{\gamma^{-1}\theta} M_{\chi}  B_{\gamma^{-1}\theta} M_{\gamma^{-1}\chi}\right] U_{\gamma^{-1}}
\end{eqnarray*}
and so,
\begin{eqnarray*}
\tr (  M_\chi A_\theta B_\theta M_\chi) & = & \sum_{\gamma\in \Gamma} \tr \left[M_{\gamma^{-1}\chi} A_{\gamma^{-1}\theta} M_{\chi}  B_{\gamma^{-1}\theta} M_{\gamma^{-1}\chi}\right]\\
 & = & \sum_{\gamma\in \Gamma} \tr \left[ M_{\chi}  B_{\gamma^{-1}\theta} M_{\gamma^{-1}\chi} A_{\gamma^{-1}\theta} M_{\chi} \right]
\end{eqnarray*}
Now the $\Gamma$-invariance of $\nu$ yields again
\begin{eqnarray*}
\tau^{\nu}(AB) &=& \int_T \tr (  M_\chi A_\theta B_\theta M_\chi) d\nu (\theta) =  \int _T \sum_{\gamma\in \Gamma} \tr \left[ M_{\chi}  B_{\theta} M_{\gamma^{-1}\chi} A_{\theta} M_{\chi} \right] d\nu (\theta) \\& =& \int_T \tr (M_{\chi}  B_{\theta} A_{\theta} M_{\chi}) d\nu (\theta)=\tau^{\nu}(BA).
\end{eqnarray*}
The normality is a consequence of normality of $\tr$ and of the Beppo-Levi property. That $\tau^\nu$ is semi-finite is straightforward. 

Finally, according to our description of the leafwise von Neumann algebra $W^*_\nu (V,\maF;E)$, its elements  are also equivariant Borel families. So, the proof of the first item is readily adapted to take care of the quotients by the isotropy groups.
 \end{proof}

Recall the two $*$-representations 
$$\pi_{{\rm reg}}: \maB_r^E\rightarrow W^*_{\nu}(G,E)\,,\quad 
  \pi_{\rm{av}}: \maB_m^E \rightarrow W^*_\nu (V,\maF;E)\,.$$

 \begin{corollary}\label{prop:composition-of-traces}
The two functionals
$
\tau^\nu_{\rm{reg}} := \tau^\nu \circ \pi_{\rm{reg}} 
\text{ and }  \tau^\nu_{\rm{av}}:= \tau^\nu_{\maF}  \circ \pi_{\rm{av}} 
$
are traces on the $C^*$-algebras  $\maB^E_r$ and $\maB^E_m$ respectively
\footnote{these traces will not be finite in general}. Moreover they are explicitly
given, for $f\in \maB^E_c$ longitudinally smooth  by the formulas 
\begin{equation}\label{trace-regular-fol}
\tau^\nu_{\rm{reg}} (f) := \int_{F\times T} \tr_{E_{[\tm, \theta]}} (f[\tm, \tm , \theta] ) d\tm d\nu (\theta) 
\end{equation}
\begin{equation}\label{trace-average-fol}
\tau^\nu_{\rm{av}} (f) := \int_{F\times T} \sum_{\gamma\in \Gamma(\theta)} \tr_{E_{[\tm, \theta]}} (f[\tm, \tm \gamma, \theta] ) d\tm d\nu (\theta).
\end{equation}

 \end{corollary}

\begin{proof} We only need to show the two formulas \eqref{trace-regular-fol} and \eqref{trace-average-fol}.
The first one is tautological, so  we only sketch the proof of the second one. 
Let then $f\in \maB_c$ longitudinally smooth be fixed. The operator $[\pi_{\rm{av}} (f)]_\theta$ acts on $L^2(L_\theta, E)$ with Schwartz kernel $f_0$ given by
$$
f_0([\tm, \theta], [\tm', \theta]) = \sum_{\gamma\in \Gamma(\theta)} f[\tm, \tm'\gamma,\theta].
$$
Therefore, the operator $M_\chi [\pi_{\rm{av}} (f)]_\theta M_\chi$ has Schwartz kernel supported in $F\times F$ viewed in $L_\theta\times L_\theta$. Recall that $L_\theta$ is identified with $\tM/\Gamma(\theta)$. We deduce 
$$
\tau^\nu [\pi_{\rm{av}} (f)] = \int_{F\times T} f_0([\tm, \theta], [\tm, \theta]) d\mu_\theta(\tm) d\nu(\theta),
$$
with $d\mu_\theta(\tm)$ being the measure induced by $d\tm$ on the leaf through $\theta$. Whence, the formula is readily deduced.
\end{proof}

In the sequel we shall also denote by $\tau^\nu_{{\rm reg}}$ the resulting trace
on the maximal $C^*$-algebra $\maB^E_m$, obtained via the natural epimorphism
$\maB^E_m \to \maB^E_r$.
 
\begin{remark}
The proof of the tracial property   of 
$
\tau^\nu_{\rm{reg}}$
 and  $ \tau^\nu_{\rm{av}}$
can also be  carried out  directly. 
Here are the details
 (we only treat the averaged trace $\tau^\nu_{\rm{av}}$ and for
  simplicity we take $E$ equal to the product line bundle). Let $f, f'$ be two elements of $C_c(G)$. We have:
$$
(f*f') [\tm,\tm',\theta] = \int_F \sum_{\alpha\in \Gamma} f[\tm, \tm''\alpha, \theta] f' [\tm''\alpha, \tm',\theta] d\tm''.
$$
Hence we deduce
\begin{eqnarray*}
\tau^\nu_{\rm{av}} (f*f') & = & \int_{F\times F \times T} \sum_{\gamma\in \Gamma(\theta)} \sum_{\alpha\in \Gamma} f[\tm, \tm'\alpha, \theta] f' [\tm'\alpha, \tm \gamma,\theta] d\tm'  d\tm d\nu(\theta)\\
& = & \int_{F\times F \times T} \sum_{\gamma\in \Gamma(\theta)} \sum_{\alpha\in \Gamma} f' [\tm', \tm \gamma \alpha^{-1}, \alpha \theta] f[\tm \alpha^{-1}, \tm', \alpha \theta]d\tm'  d\tm d\nu(\theta)\\
& = &  \int_{F\times F} \sum_{\alpha\in \Gamma} \int_T \sum_{\gamma'\in \Gamma(\theta ')} f' [\tm' {\gamma'}^{-1}, \tm \alpha^{-1}, \theta '] f[\tm \alpha^{-1}, \tm', \theta ']d\tm'  d\tm d\nu(\theta ')\\
& = & \int_{F\times T} \sum_{\gamma'\in \Gamma(\theta ')} (f'*f) [\tm {\gamma'} ^{-1}, \tm ', \theta'] d\tm'  d\tm d\nu(\theta ').
\end{eqnarray*}
Now note that since $\gamma '\in \Gamma(\theta ')$, we have
$$
(f'*f) [\tm {\gamma'} ^{-1}, \tm ', \theta'] = (f'*f) [\tm , \tm ' \gamma'  , \theta'].
$$
Therefore, we get
$$ 
\tau^\nu_{\rm{av}} (f*f') = \tau^\nu_{\rm{av}} (f'*f).
$$
\end{remark}
%

\begin{proposition}\label{extended-traces}
\begin{enumerate}
\item The trace $\tau^\nu_{\rm{reg}}$  induces a group homomorphism  $\tau^\nu_{\rm{reg},*} : K_0 (\maB^E_r)\longrightarrow \RR.
$
\item The  trace $\tau^\nu_{\rm{av}}$  induces a group homomorphism  
$
\tau^\nu_{\rm{av},*} : K_0 (\maB^E_m) \longrightarrow \RR.$
\end{enumerate}
\end{proposition}

 \begin{proof} 
 We only sketch the proof of this classical result:
one shows, for instance,  that $L^1(W^*_\nu (G;E))\cap \maB_r ^E$, with $L^1(W^*_\nu (G;E))$
the Schatten-ideal of $\tau^\nu$-trace class operators,   is dense holomorphically closed in $\maB_r ^E$.
Similarly $L^1(W^*_\nu (V,\maF;E))\cap \pi_{{\rm av}}(\maB_m ^E)$ is dense and holomorphically
closed in $\pi_{{\rm av}}(\maB_m ^E)$; this finishes the proof by using the definition of 
$
\tau^\nu_{av}$.

\end{proof}

\subsection{Compatibility with Morita isomorphisms}
The goal of this subsection is to prove the compatibility between the different traces defined
so far and the  isomorphisms induced in $K$-theory by Morita equivalence.

Recall the $C^*$-algebras $\maA_r$ and $\maA_m$ associated to the  groupoid $\maG:=T\rtimes \Gamma$.
Let $\maK$ denote as usual the $C^*$-algebra of compact operators on a Hilbert space.

\begin{proposition}
There are isomorphisms of $C^*$-algebras:
\begin{equation}\label{iso-morita}
\maB_r \simeq \maA_r\otimes \maK \;,\quad\quad \maB_m \simeq \maA_m\otimes \maK.
\end{equation}
\end{proposition}
\begin{proof}
We fix $\tm_0\in \tM$ and consider the  subgroupoid $G(\tm_0)$ consisting of the elements which start and end in the image of $\{\tm_0\}\times T$ in $V$:
$$
G(\tm_0) = \{ [\tm_0, \tm_0\alpha, \theta], \theta\in T \text{ and } \alpha\in \Gamma\}.
$$
Notice that the composition in $
G(\tm_0)$ can be expressed in the following way:
$$[\tm_0, \tm_o \alpha^\prime,\theta^\prime]\circ [\tm_0,\tm_o \alpha,\alpha^\prime \theta^\prime]=[\tm_0,\tm_0\alpha\alpha^\prime,
\theta^\prime]\,.$$
Then 
 there is a {\it groupoid isomorphism} between $G(\tm_0)$ and the groupoid $\cG$ given by
$$
[\tm_0, \tm_0\alpha, \theta] \longmapsto (\theta, \alpha^{-1}).
$$

In particular the reduced (respectively  maximal) $C^*$-algebras  associated to $G(\tm_0)$ and $\maG$ are isomorphic:
$C^*_r(G(\tm_0))\simeq \maA_r$ (respectively $C^*_m ( G(\tm_0))\simeq \maA_m$).
Now the main result in \cite{Hilsum-Skandalis-stabilite}, see also
\cite{Moulay-triangulation}, together with the fact that the image of $\{\tm_0\}\times T$ in $V$ intersects every leaf of the foliation, we deduce that the stable $C^*$-algebra $\maB_r$ is isomorphic to the tensor product $C^*$-algebra $\maA_r\otimes \maK$. In the same way, the $C^*$-algebra $\maB_m$ is isomorphic to the tensor product $C^*$-algebra $\maA_m\otimes \maK$, using  the maximal version of the stability theorem which is valid as pointed out in \cite{Hilsum-Skandalis-stabilite}.

\end{proof}

Denote by $\maM_{\rm{r}}:K_0(\maA_r) \to K_0(\maB_r)$ and $\maM_{\rm{m}}:K_0(\maA_m) \to K_0(\maB_m)$ the isomorphisms induced in $K$-theory by the isomorphisms \eqref{iso-morita}

\begin{proposition}\label{prop:morita-compatible}   The following diagrams are commutative

\begin{picture}(400,70)

\put(60,30){$K_0(\maA_r)$}
\put(95,40){$\nearrow$}
\put(95,20){$\searrow$}
\put(75,45){$\maM_{\rm{r}}$}
\put(75,12){$\tau^\nu_{\rm{reg},*}$}

\put(100,55){$K_0(\maB_r)$}
\put(120,30){$\downarrow$}
\put(120,5){$\RR$}
\put(125,30){$\tau^\nu_{\rm{reg},*}$}

\put(230,30){$K_0(\maA_m)$}
\put(275,40){$\nearrow$}
\put(275,20){$\searrow$}
\put(255,45){$\maM_{\rm{m}}$}
\put(255,12){$\tau^\nu_{\rm{av},*}$}

\put(280,55){$K_0(\maB_m)$}
\put(300,30){$\downarrow$}
\put(300,5){$\RR$}
\put(305,30){$\tau^\nu_{\rm{av},*}$}

\end{picture}

\end{proposition}

\begin{proof}
Let us identify $T$ with a  fiber of the flat bundle $V=\tM\times _\Gamma T \to M$. Let $\Omega$ be an open connected submanifold of $\tM$ contained in a fundamental domain $F$ of the action of $\Gamma$. Let $U$ be the projection  in $V$ of $\Omega\times T$.  Then $U \to T$ is an open neighborhood of $T$ in $V$ such that the induced  foliation on $U$ is given by the fibres of $U\to T$. 
The subgroupoid $G_U^U$ of $G$ consisting
 of homotopy classes of paths drawn in leaves, starting and ending in $U$, can be describe as
$$
G_U^U = \{[\tm, \tm' \gamma, \theta] \in \frac{\Omega\times \tM\times T}{\Gamma}, [\tm,\theta]\in U \text{ and }[\tm ' \gamma, \theta]\in U\}.
$$
An easy inspection of the groupoid laws in $G_U^U$ shows that  the bijection  
$$
[\tm, \tm' \gamma, \theta] \longmapsto (\tm,\tm',\theta,\gamma^{-1})\in \Omega\times \Omega \times (T\rtimes \Gamma),
$$
is an isomorphism  of groupoids, so that
the reduced (resp. maximal) $C^*$-algebra of $G_U^U$ is isomorphic to $\maK(L^2 \Omega)\otimes [C(T) \rtimes_r \Gamma]$ (resp. $\maK(L^2\Omega)\otimes [C(T) \rtimes_m \Gamma]$). Recall that $\maK(L^2\Omega)$ denotes the nuclear $C^*$-algebra of compact operators in the Hilbert space $L^2\Omega$. 

If we now fix a continuous   compactly supported 
function $\varphi$ on $\Omega$ with $L^2$ norm equal to $1$ then for any continuous  compactly supported function $\xi\in \maA_c$, we set:
$$
T(\xi)[\tm,\tm', \theta] := \sum_{\gamma, \gamma'\in \Gamma} \varphi (\tm \gamma) {\overline{\varphi (\tm ' {\gamma '})}} \xi (\gamma^{-1} \theta, \gamma^{-1}  {\gamma'}).
$$
Since $\varphi$ is supported in a fundamental domain, it is clear that only one couple $(\gamma, \gamma')$ gives a non trivial contribution. Moreover, the function $T(\xi)$ is well defined on $G$ and is supported inside $G_U^U$. The map $T$ is a $*$-homomorphism from the algebra $\maA_c$ to the algebra $\maB_c $ which implements the Morita isomorphisms $\maM_{\rm{r}}$ and $\maM_{\rm{m}}$ in $K$-theory. Indeed, we have:
\begin{eqnarray*}
T(\xi) * T(\xi') [\tm,\tm', \theta] & = & \int_\tM T(\xi)[\tm, \tm'', \theta] T(\xi') [\tm '', \tm',\theta] d\tm''\\
& = & \sum_{\alpha,\alpha',\beta, \beta' \in \Gamma} \varphi (\tm \alpha){\overline{\varphi (\tm ' \beta')}} \int_\tM {\overline{\varphi (\tm'' \beta)}} \varphi (\tm'' \alpha')   d\tm'' \times \\
& &  \xi (\alpha^{-1} \theta, \alpha^{-1}  \beta) \xi' ({\alpha'} ^{-1}\theta,{ \alpha'} ^{-1} {\beta'})\\
& = & \sum_{\alpha, \alpha'\in \Gamma} \varphi (\tm \alpha) {\overline{\varphi (\tm ' {\alpha '})}} \sum_{\beta \in \Gamma} \xi (\alpha^{-1} \theta, \alpha^{-1}
 \beta) \xi ' (\beta^{-1} \theta, \beta^{-1} {\alpha'}) \\
& = & \sum_{\alpha, \alpha'\in \Gamma} \varphi (\tm \alpha) {\overline{\varphi (\tm ' {\alpha '})}} (\xi * \xi') (\alpha^{-1}\theta, \alpha^{-1} {\alpha'})\\
& = & T(\xi * \xi') [\tm, \tm', \theta].
\end{eqnarray*}
Hence, we conclude that 
$$
T(\xi) * T(\xi') = T(\xi * \xi').
$$
In a similar  way one checks that $(T(\xi))^*= T(\xi^*)$ .\\
 $T$  extends to a morphism between the corresponding reduced $C^*$-algebras. More precisely, let $f\in L^2(\tM)$, then the regular representation $\pi^{{\rm reg}}$ is given for any $\tm\in \tM$  by:
$$
(\pi^{{\rm reg}} T(\xi))_\theta (f) (\tm) = \int_\tM \sum_{\gamma' , \gamma\in \Gamma} \varphi (\tm \gamma) 
{\overline{\varphi (\tm' {\gamma'})}} \xi (\gamma^{-1}\theta, \gamma^{-1}{\gamma'}) f(\tm') d\tm '.
$$
Denote by $g : \Gamma \to \CC$ the function given by
$$
g (\gamma^\prime) :=  \int_{\tM}  {\overline{\varphi (\tm' {\gamma'}^{-1})}} f(\tm') d\tm',
$$
then, one easily shows that the function $g$ belongs to $\ell^2(\Gamma)$ and that its $\ell^2$-norm can be estimated as follows:
\begin{eqnarray*}
\|g \|_2^2 &=& \sum_{\gamma^\prime}| g(\gamma^\prime)|^2=\sum_{\gamma^\prime} \left| \int_{\tM} \overline{\phi(\tm^\prime {\gamma^\prime}^{-1})}f(\tm^\prime)\,d\tm \right|^2\\
&=& \sum_{\gamma^\prime} \left| \int_{F\gamma^\prime}  \overline{\phi(\tm^\prime {\gamma^\prime}^{-1})} f(\tm^\prime)\,d\tm^\prime \right|^2\leq
\sum_{\gamma^\prime} \int_{F\gamma^\prime} |f(\tm^\prime)|^2 d\tm^\prime=\|f\|^2_2
\end{eqnarray*}

If we recall  the regular representation of the algebra $\maA_c$, denoted  also by $\pi^{{\rm reg}}$,  then, using $g$  we can write:
\begin{equation*}
(\pi_{{\rm reg}} T(\xi))_\theta  (f) (\tm) = \sum_{\gamma\in\Gamma} \phi(\tm \gamma) \sum_{\gamma^\prime\in \Gamma}
\xi(\gamma^{-1}\theta,\gamma^{-1}\gamma^\prime) g({\gamma^\prime}^{-1} )=\sum_{\gamma\in\Gamma}
\phi(\tm\gamma) (\pi^{{\rm reg}} _\theta (\xi)) (g) (\gamma^{-1})
\end{equation*}

Therefore, if we compute the $L^2$-norm of the function $(\pi^{{\rm reg}} T(\xi))_\theta  (f)$  we get:
\begin{eqnarray*}
\| (\pi_{{\rm reg}} T(\xi))_\theta  (f) \|_2^2 & = & 
\int_{\tM}\left|  \sum_{\gamma\in \Gamma} \phi(\tm \gamma)  \pi^{{\rm reg}}_\theta (\xi)  (g) (\gamma^{-1}) \right|^2 d\tm \\
&=& \sum_{\alpha\in \Gamma} \int_{F\alpha^{-1}} \left| \phi(\tm \alpha) \pi^{{\rm reg}}_\theta (\xi) (g) (\alpha^{-1}) \right|^2 d\tm\\
&=& \sum_{\alpha\in\Gamma} \left| \pi^{{\rm reg}}_\theta (\xi) (g) (\alpha^{-1}) \right|^2 \int_{F\alpha^{-1}} |\phi(\tm \alpha)|^2 d\tm\\
&=& \|\pi^{{\rm reg}}_\theta (\xi) (g)\|^2_2\\
&\leq&  \|\xi\|_{\maA_r} ^2  \|g\|_2^2  \leq  \|\xi\|_{\maA_r} ^2 \|f\|_2^2;
\end{eqnarray*}
Summarizing: $\sup_{\theta\in T} \|(\pi^{{\rm reg}}(T\xi))_\theta \| \leq \|\xi\|_{\maA_r}$ so that $\|T(\xi)\|_{\maB_r} \leq \|\xi\|_{\maA_r}$ as required.\\
It thus remains to show compatibility of the traces with respect to the homomorphism $T$, and only on the compactly supported functions. Let us  start with the regular trace. We have:
\begin{eqnarray*}
\tau^\nu_{\rm{reg}} (T(\xi)) & = & \int_{F\times T} T(\xi) [\tm, \tm, \theta] d\tm d\nu(\theta)\\
& = & \int_T \xi ( \theta, 1) \int _\tM |\varphi(\tm)| ^2 d\tm d\nu (\theta)\\
& = & \int_T \xi ( \theta, 1) d\nu (\theta)\\
& = & \tau^\nu_{\rm{reg}} (\xi) .
\end{eqnarray*}
Note that when $\tm\in \Omega$, only $\gamma= 1$ contributes to the sum defining $T(\xi)$.

Let us now check, briefly, that $T$ induces a morphism between the maximal $C^*$-algebras.
It suffices to show that $T$ is continuous with respect to the $L^1$-norms on the groupoids
$\maG$ and $G$. 
But for $\xi\in \maA_c$ and for any $\tm\in\Omega$ we have 
\begin{eqnarray*}
\int_{\tM} | (T\xi) [\tm,\tm^\prime,\theta]|\,d\tm^\prime & \leq & |\phi (\tm)|\int_{\tM} |\phi(\tm^\prime)|\,d\tm^\prime\,\left(
\sum_{\gamma^\prime\in \Gamma}  |\xi(\theta,\gamma^\prime)|\right)\\
& \leq & \|\phi\|_1 \|\phi \|_{\infty} \|\xi\|_1 
\end{eqnarray*}
Hence, 
$$ \|T(\xi)\|_1 \leq \|\phi\|_1 \|\phi \|_{\infty} \|\xi\|_1 \,.$$
Now let us check the compatibility with the average trace $\tau^\nu_{\rm{av}}$. We have, for $\xi\in \maA_c$:
\begin{eqnarray*}
\tau^\nu_{\rm{av}} (T(\xi)) & = & \int_{F\times T} \sum_{\gamma\in\Gamma(\theta)} T(\xi) [\tm, \tm \gamma, \theta] d\tm \;d\nu(\theta)\\
&=&
\int_{\Omega\times T} \sum_{\gamma\in\Gamma(\theta)} T(\xi) [\tm, \tm \gamma, \theta] d\tm \;d\nu(\theta)\\
& = & \int_{T} \sum_{\gamma\in\Gamma(\theta)} \xi (\theta, \gamma) \int_{\Omega} |\varphi(\tm)|^2 d\tm\;  d\nu(\theta)\\
& = & \int_{T} \sum_{\gamma\in\Gamma(\theta)} \xi (\theta, \gamma) d\nu(\theta)\\
& = & \tau^\nu_{\rm{av}} (\xi).
\end{eqnarray*}
Note that in the expression $T(\xi) [\tm, \tm \gamma, \theta]$ for $\tm \in \Omega$, only the couple $(1,\gamma)$ contributes non trivially to the sum.
\end{proof}

\section{Hilbert modules and Dirac operators.}\label{sec:modules}

\subsection{Connes-Skandalis Hilbert module}

   Recall that $V=\tM \times_\Gamma T$ where $\tM\to M$ is the universal $\Gamma$-covering of the closed manifold $M$ and where $\Gamma$ acts by homeomorphisms 
   on the compact space  $T$. We fix a  hermitian vector bundle $E$ over $V$ and we denote by $\what{E}$ its pull-back to $\tM \times T$. 
   We define a right action of the convolution algebra  $\maA_c=C_c(T\rtimes \Gamma)\equiv C_c(\maG)$
   on the space $\maE_c=C_c^{\infty,0}(\widetilde{M}\times T; \widehat{E})$, of compactly supported sections of the vector bundle $\widehat{E}$ which are smooth with respect to the $\tM$ variable and continuous with respect to the $T$ variable, as follows.
   $$
   (\xi f)(\tm,\theta) =\sum_{\gamma \in \Gamma}
    \;\xi(\tm \gamma^{-1}, \gamma \theta ) f(\gamma \theta , \gamma), \quad
   \xi \in \maE_c, \quad
   f \in \maA_c.
   $$ 
A $\maA_c$-valued inner product $<. ; .>$
   on $\maE_c$ is also defined by \cite{Hilsum-Skandalis-stabilite}
   $$
   <\xi_1 ; \xi_2>(\theta, \gamma) :=
   \int_{\widetilde{M}}
   <\xi_1(\tm , \gamma^{-1} \theta) ; \xi_2(\tm \gamma^{-1} , \theta)>_{E_{[\tm,\theta]}}
   d\tm,
   $$ where $< . ; .>_{E}$ is the hermitian scalar product that we have fixed
   of the vector bundle $E$. A classical computation shows that these operations endow the space $\maE_c$ with the structure of a pre-Hilbert module over the algebra $\maA_c$.

  As in the previous sections, we denote by $\maA_r$ and  $\maA_m$ the reduced and maximal
  $C^*$-algebras of the groupoid $\maG$ . Recall that there is a natural $C^*$-algebra
  morphism
$$
\lambda: \maA_m \longrightarrow \maA_r.
$$
The pre-Hilbert $\maA_c$-module $\maE_c$ can be completed with respect to the reduced $C^*$-norm to yield a right Hilbert $C^*$-module over $\maA_r$ that we shall denote by $\maE_r$. In the same way, we can complete $\maE_c$ with respect to the maximal $C^*$-norm and define the Hilbert $C^*$-module $\maE_m$ over the $C^*$-algebra $\maA_m$. It is then clear that the natural map
$
\maE_c \longrightarrow \maE_r,
$
extends to a morphism of Hilbert modules $\maE_m \rightarrow \maE_r$. More precisely, we have a well defined linear map 
$$
\varrho: \maE_m \longrightarrow \maE_r \text{ such that } 
\varrho (\xi f) = \varrho (\xi) \lambda (f) \;\; \;\;\; f\in \maA_m \text{ and } \xi \in \maE_m.
$$

We denote as in the previous sections by $G$ the  monodromy groupoid 
$$
G:= \dfrac{\tM \times \tM \times T}{\Gamma}.
$$
The algebra $\maB_c^E$ of smooth compactly supported sections of the bundle $\END(E)$ over $G$ is faithfully represented in $\maE_c$ by the formula \cite{ConnesSkandalis}
$$
\chi (\varphi) (\xi) (\tm, \theta) := \int_\tM  \varphi [\tm, \tm', \theta] \xi (\tm', \theta) d\tm', \quad \varphi\in C_c^\infty (G,\END(E)), \xi \in \maE_c.
$$

Recall that $\maB_r^E$ and $\maB^E_m$ are respectively the reduced and maximal $C^*$-algebras associated with $G$ and with coefficients in $E$. Given a $C^*$-algebra $A$ and a Hilbert $A$-module $\maE$, the algebra $B_A(\maE)$ consists of bounded adjointable $A$-linear morphisms of $\maE$. Recall also that the $C^*$-algebra $\maK_{A}(\maE)$ of $A$-compact operators is the completion in $B_A(\maE)$ of the subalgebra of $A$-finite rank operators. The following proposition is proved in \cite{Hilsum-Skandalis-stabilite}, see also \cite{Moriyoshi-Natsume} and
\cite{Moulay-triangulation}.

\begin{proposition}\label{prop:compact}
For any $\varphi\in \maB^E_c$, the map $\chi(\varphi):\maE_c \to \maE_c$ is $\maA_c$-linear and the morphism $\chi$ extends to  continuous $*$-representations
$$
\chi_r: \maB^E_r \longrightarrow \maK_{\maA_r}(\maE_r)\text{ and } \chi_m: \maB^E_m \longrightarrow \maK_{\maA_m}(\maE_m),
$$
which are $C^*$-algebra isomorphisms.
\end{proposition}

Notice that the proof of this Proposition is usually given for the holonomy groupoid
of the foliation; however the same argument applies to the monodromy groupoid.
Note also that the proof is usually given for the reduced $C^*$-algebra but it remains valid
for the maximal $C^*$algebra too \cite{Hilsum-Skandalis-stabilite} [Remarque 5].

For any $\theta\in T$, we have defined in  Subsection \ref{DiscreteVN} a representation $\pi_\theta^{av}$ of the maximal $C^*$-algebra $\maA_m$ in the Hilbert space $\ell^2(\Gamma/\Gamma ( \theta))$. By using Remark \ref{average-remark} we can write 
$$
\pi^{av}_\theta (f) (\xi) (\theta') := \sum_{\theta''\in \Gamma .\theta}\;\; \sum_{\gamma \theta''=\theta'} f(\theta',\gamma)  \xi(\theta''), \quad f\in \maA_c, \xi\in \ell^2 (\Gamma \theta) \text{ and } \theta'\in \Gamma\theta.
$$

The family of representations $(\pi_\theta^{{\rm av}})_{\theta \in T}$ then yields the average representation of $\maA_m$ in the leafwise von Neumann algebra 
$W^*_\nu (V,\maF;E)$. Using the $\maA_m$-Hilbert module $\maE_m$ together with the representation $\pi_\theta^{av}$, one defines the Hilbert space 
$$
\maH_\theta^{av}:=\maE_m\otimes_{\pi_\theta^{{\rm av}}} \ell^2(\Gamma\theta).
$$
Similarly 
$$
\maH_\theta^{reg}:=\maE_m\otimes_{\pi_\theta^{reg}} \ell^2(\Gamma).
$$

\begin{lemma}\label{lemma:identification}
There exists an isomorphism of Hilbert spaces, $\Phi_\theta$, between $\maH_\theta^{av}$ and the Hilbert space $L^2(L_\theta, E)$ of square integrable sections of the vector bundle $E$ over the leaf $L_\theta$ through $\theta$, induced by the formula
$$
\Phi_\theta (\xi \otimes f) (\tm, \theta) := \sum_{\gamma\in \Gamma} f(\gamma \theta) \;\left[\xi(\tm\gamma^{-1}, \gamma \theta)\right], \quad \xi\in 
\maE_c
\text{ and } f\in C_c(\Gamma\theta).
$$
Similarly there exists an isomorphism $\Psi_\theta$ of Hilbert spaces between 
$\maH_\theta^{reg}$ and $L^2(\tM_\theta,\widehat{E})$
induced  by the formula 
$$
\Psi_\theta (\xi\otimes \delta_\gamma) (\tm) := \xi (\tm \gamma^{-1}, \gamma \theta).
$$
where $\delta_\gamma$ denotes the delta function at $\Gamma$.
\end{lemma}

\begin{proof} 
If $\alpha\in \Gamma(\theta)$ then we can write for $\xi\in\maE_c$:
\begin{eqnarray*}
\Phi_\theta (\xi \otimes f) (\tm \alpha^{-1}, \theta) & =  & \sum_{\gamma\in \Gamma} f(\gamma \theta) \; \left[\xi(\tm \alpha^{-1} \gamma^{-1}, \gamma \theta)\right]\\
& = & \sum_{\beta\in \Gamma} f(\beta \theta) \; \xi (\tm\beta^{-1}, \beta \theta)\\
& = &  \Phi_\theta (\xi\otimes f) (\tm, \theta)
\end{eqnarray*}
Hence, $\Phi_\theta (\xi\otimes f)$ is a smooth section of ${\hat E}$ over $\tM_\theta$ which is $\Gamma(\theta)$-invariant. Moreover, if $f=\delta_{\gamma\theta}$ and if we denote by  $K_{\gamma\theta}$ the (compact) support of $\xi$ in $\tM\times \{\gamma\theta\}$ then the support of $\Phi_\theta (\xi \otimes \delta_{\gamma\theta})$ is contained in $$
[K_{\gamma\theta}\cdot \gamma] \cdot \Gamma(\theta),
$$
and hence is $\Gamma(\theta)$-compact.

Let now $g\in \maA_c$ be given. Then we have
\begin{eqnarray*}
\Phi_\theta (\xi g \otimes f) (\tm, \theta) & =  & \sum_{\gamma\in \Gamma} f(\gamma \theta) (\xi g)(\tm \gamma^{-1}, \gamma \theta)\\
& = & \sum_{\gamma\in \Gamma} f(\gamma \theta) \sum_{\alpha\in \Gamma} g(\alpha\gamma \theta, \alpha) \xi(\tm \gamma^{-1}\alpha^{-1}, \alpha \gamma \theta)\\
& = & \sum_{\theta', \theta''\in \Gamma \theta} \sum_{\gamma \theta=\theta', \alpha \theta' = \theta''} f(\theta') g(\theta'', \alpha)  \xi(\tm (\alpha \gamma)^{-1}, \theta'')\\
& = & \sum_{\theta', \theta''\in \Gamma \theta} \sum_{\beta \theta=\theta'', \alpha \theta' = \theta''} f(\theta') g(\theta'', \alpha)  \xi(\tm \beta^{-1}, \theta'').
\end{eqnarray*}
On the other hand, we compute 
\begin{eqnarray*}
\Phi_\theta (\xi  \otimes \pi_\theta^{{\rm av}} (g)(f)) (\tm, \theta) & =  & \sum_{\theta''\in \Gamma\theta}  \pi_\theta^{{\rm av}} (g)(f)(\theta'') \sum_{\gamma_1\theta = \theta''} \xi (\tm \gamma_1^{-1}, \theta'')\\
& = & \sum_{\theta'', \theta'\in \Gamma\theta} \sum_{\delta \theta'=\theta'', \delta_1\theta = \theta''} f(\theta') g(\theta'',\delta) \xi (\tm \delta_1^{-1}, \theta'')
\end{eqnarray*}
Hence, we obtain the equality
$
\Phi_\theta (\xi g \otimes f) = \Phi_\theta (\xi  \otimes \pi_\theta^{{\rm av}} (g)(f)).
$

In order to finish the proof, we need to identify the scalar product on the Hilbert space $\maH_\theta^{av}$. We have
\begin{eqnarray*}
<\xi \otimes f, \xi \otimes f > & = & <\pi_\theta^{av} (<\xi, \xi>) (f) , f >\\
& = &  \sum_{\theta'\in \Gamma\theta} \pi_\theta^{av} (<\xi, \xi>) (f) (\theta') {\overline{f(\theta')}}\\
& = & \sum_{\theta'\in \Gamma\theta}  {\overline{f(\theta')}} \sum_{\beta\in \Gamma} <\xi, \xi> (\theta',\beta) f(\beta^{-1}\theta')\\
& = & \sum_{\theta', \theta'' \in \Gamma\theta} {\overline{f(\theta')}} f(\theta'') \sum_{\beta \theta''=\theta'} \int_\tM <\xi(\tm , \beta^{-1} \theta') , \xi(\tm \beta^{-1}, \theta')  > d\tm\\
& = & \sum_{\theta',\theta''\in \Gamma\theta} {\overline{f(\theta')}} f(\theta'') \sum_{\alpha \theta'=\theta''} \int_\tM <\xi(\tm , \alpha \theta') , \xi(\tm \alpha , \theta')> d\tm  
\end{eqnarray*}
On the other hand, if we view $\Phi_\theta (\xi\otimes f)$ as a section over the leaf $L_\theta$ through $\theta$, then we can use a fundamental domain $F_\theta$ for the free and proper action of the isotropy group $\Gamma(\theta)$ on $\tM$ and write
\begin{eqnarray*}
<\Phi_\theta(\xi\otimes f), \Phi_\theta (\xi\otimes f)> & = & \int_{F_\theta} <\Phi_\theta (\xi\otimes f)(\tm, \theta), \Phi_\theta (\xi\otimes f)(\tm, \theta) > d\tm\\
& = & \sum_{\theta_1, \theta_2\in \Gamma\theta} f(\theta_1) {\overline {f(\theta_2)}}\sum_{\gamma_1 \theta = \theta_1, \gamma_2\theta=\theta_2} \int_{F_\theta} <\xi(\tm \gamma_1^{-1}, \theta_1), \xi(\tm\gamma_2^{-1}, \theta_2) > d\tm
\end{eqnarray*}
We fix a section $\varphi:\Gamma \theta \to \Gamma$ of the map $\gamma \mapsto \gamma \theta$. Then $\beta = \varphi (\theta_1)^{-1} \gamma_1$ is an element of the isotropy group $\Gamma(\theta)$  and we have
\begin{eqnarray*}
<\Phi_\theta(\xi\otimes f), \Phi_\theta (\xi\otimes f)> & = &  \sum_{\theta '' ,\theta '\in \Gamma\theta} f(\theta '' ) {\overline {f(\theta ')}} \sum_{\gamma_2\theta = \theta '} \sum_{\beta\in \Gamma(\theta)} \int_{F_\theta} < \xi (\tm \beta^{-1} \varphi(\theta '' )^{-1}, \theta ''), \xi (\tm \gamma_2^{-1}, \theta ') > d\tm\\
& = & \sum_{\theta '',\theta ' \in \Gamma\theta} f(\theta '') {\overline {f(\theta ' )}} \sum_{\gamma_2\theta = \theta' } \sum_{\beta\in \Gamma(\theta)} \int_{F_\theta \beta^{-1} \varphi(\theta '')^{-1}} \times\\
& &< \xi (\tm_1, \theta ''), \xi (\tm_1 \varphi(\theta '' )\beta\gamma_2^{-1}, \theta ') > d\tm_1\\
& = & \sum_{\theta '' ,\theta ' \in \Gamma\theta} f(\theta '' ) {\overline {f(\theta ' )}} \sum_{\alpha \theta '  = \theta '' } \sum_{\beta\in \Gamma(\theta)} \int_{F_\theta \beta^{-1} \varphi(\theta '' )^{-1}} < \xi (\tm_1, \theta '' ), \xi (\tm_1 \alpha, \theta ' ) > d\tm
\end{eqnarray*}
Setting  $\delta=\varphi(\theta_1) \beta^{-1} \varphi(\theta_1)^{-1}$ and  noticing that a fundamental domain $F_{\theta ''}$
is equal to $F_{\theta} \varphi(\theta '')^{-1}$
we get
\begin{eqnarray*}
\Phi_\theta(\xi\otimes f), \Phi_\theta (\xi\otimes f)> & = &  \sum_{\theta',\theta''\in \Gamma\theta} f(\theta'') {\overline {f(\theta')}} \sum_{\alpha \theta '= \theta ''} \sum_{\delta\in \Gamma(\theta'')} \int_{(F_\theta \varphi(\theta'')^{-1}) \delta} < \xi (\tm_1, \theta''), \xi (\tm_1 \alpha, \alpha^{-1}\theta '') > d\tm\\
& = &  \sum_{\theta',\theta''\in \Gamma\theta} f(\theta'') {\overline {f(\theta')}} \sum_{\alpha \theta '= \theta ''} \sum_{\delta\in \Gamma(\theta'')} \int_{F_{\theta ''} \delta} < \xi (\tm_1, \theta''), \xi (\tm_1 \alpha, \alpha^{-1}\theta '') > d\tm\\
& = & \sum_{\theta',\theta''\in \Gamma\theta} f(\theta'') {\overline {f(\theta')}} \sum_{\alpha \theta '= \theta ''} \int_{\tM} < \xi (\tm_1, \alpha \theta'), \xi (\tm_1 \alpha, \theta ') > d\tm\\
\end{eqnarray*}
Hence $
<\xi\otimes f, \xi\otimes f > = <\Phi_\theta (\xi\otimes f), \Phi_\theta (\xi \otimes f)>.$
It now remains to show that $\Phi_\theta$ is surjective. Let $\eta$ be a smooth compactly supported section over the leaf $L_\theta$ and denote by ${\tilde \eta}$ its lift into a $\Gamma(\theta)$-invariant section over $\tM\times \theta$ and by $\xi_0$ any extension of ${\tilde \eta}$ into a leafwise smooth continuous section over $\tM\times T$. Let $\varphi$ be a smooth function on $\tM$ such that $\sum_{\alpha\in \Gamma(\theta)} \alpha \varphi = 1$ and such that for any compact set $K$ in $L_\theta\simeq \tM/\Gamma(\theta)$, the intersection of the support of $\varphi$ with the inverse image of $K$, under the projection $\tM\to L_\theta$, is compact in $\tM$. We view $\varphi$ as a  function on $\tM\times T$ independent of the $T$ variable and set
$$
\xi:= \varphi \xi_0.
$$
Then $\xi\in C_c^{\infty, 0} (\tM\times T, \tE)$ and one checks immediately that  $\Phi_\theta (\xi\otimes \delta_\theta) = \eta$.
The proof of the second isomorphism is simpler and is left as an exercise. 
\end{proof}

Recall that we have defined two representations, that we have both denoted $\pi^{av}$, respectively of the $C^*$-algebras $\maA_m$ and $\maB^E_m$ in the corresponding  von Neumann algebras of the discrete groupoid $\maG$ and of the monodromy groupoid $G$ with coefficients in the vector bundle $E$:
$$\pi^{{\rm av}}:\maA_m \rightarrow W^*_{{\rm av}}(\maG)\,,\quad \pi^{{\rm av}}:\maB_m^E \rightarrow W^*_{\nu}(V,\maF;E)\,.$$ 
 Recall also that we have defined a $*$-representation $\chi_m$ of $\maB_m^E$ in the compact operators of the Hilbert module $\maE_m$:
 $$\chi_m : \maB_m^E \rightarrow \maK_{\maA_m}(\maE_m)\,.$$

\begin{proposition}\label{PhiComp}\
Let $S$ be a given element of $\maB_m^E$. Then we have
$$
\pi_\theta^{av} (S) = \Phi_\theta \circ \left[ \chi_m(S) \otimes_{\pi_\theta ^{av}} I_{\ell^2(\Gamma\theta)}\right] \circ \Phi_\theta^{-1}. 
$$
with $\Phi_\theta: 
\maE_m\otimes_{\pi_\theta^{{\rm av}}} \ell^2(\Gamma\theta)\rightarrow L^2(L_\theta,E)
$
the isomorphism given in Lemma \ref{lemma:identification}.
In the same way, we have
$$
\pi_\theta^{reg} (S) = \Psi_\theta \circ \left[ \chi_r(S) \otimes_{\pi_\theta ^{reg}} I_{\ell^2(\Gamma)}\right] \circ \Psi_\theta^{-1}. 
$$
with $\Psi_\theta: \maE_m\otimes_{\pi_\theta^{reg}} \ell^2(\Gamma)\rightarrow L^2(\tM_\theta,\widehat{E})$
the second isomorphism given in Lemma \ref{lemma:identification}.

\end{proposition}

\begin{proof}
Let us fix an element $k \in C_c^{\infty, 0} (G;\END (E))$ and give the proof for $S=k$. We  compute for $\xi\in \maE_c$ and $f\in C_c [\Gamma\theta]$:
\begin{eqnarray*}
\Phi_\theta (\chi (k) (\xi) \otimes f) (\tm, \theta) & = & \sum_{\gamma\in \Gamma} f(\gamma\theta)  \chi (k)(\xi) (\tm \gamma^{-1}, \gamma \theta)\\
& = & \sum_{\gamma\in \Gamma} f(\gamma\theta) \int_\tM  k[\tm\gamma^{-1}, \tm',\gamma \theta] \xi(\tm', \gamma\theta) d\tm'\\
& = & \sum_{\gamma\in \Gamma} f(\gamma\theta) \int_\tM k[\tm, \tm'\gamma, \theta] \xi(\tm', \gamma\theta) d\tm'\\
& = & \sum_{\gamma\in \Gamma} f(\gamma\theta) \int_\tM k[\tm, \tm_1, \theta] \xi(\tm_1\gamma^{-1}, \gamma\theta) d\tm_1.
\end{eqnarray*}
On the other hand,  we have:
\begin{eqnarray*}
\pi^{av}_\theta (k) (\Phi_\theta (\xi\otimes f)) (\tm, \theta) & = & \int_{F_\theta} \sum_{\alpha\in \Gamma (\theta)}  k[\tm, \tm' \alpha, \theta]  \sum_{\gamma\in \Gamma} f(\gamma\theta)  \xi (\tm'\gamma^{-1}, \gamma\theta) d\tm'\\
& = & \sum_{\gamma\in \Gamma} f(\gamma\theta) \sum_{\alpha\in \Gamma(\theta)} \int_{F_\theta\alpha} k[\tm, \tm'',\theta]  \xi (\tm''\alpha^{-1}\gamma^{-1}, \gamma\theta) d\tm''\\
& = & \sum_{\gamma'\in \Gamma} f(\gamma'\theta) \sum_{\alpha\in \Gamma(\theta)} \int_{F_\theta \alpha} k[\tm, \tm'', \theta]  \xi (\tm''{\gamma'}^{-1}, \gamma'\theta) d\tm''\\
& = & \sum_{\gamma'\in \Gamma} f(\gamma'\theta) \int_{\tM} k[\tm, \tm'', \theta] \xi (\tm''{\gamma'}^{-1}, \gamma'\theta) d\tm''.
\end{eqnarray*}
So we get 
$$\Phi_\theta (\chi (k) (\xi) \otimes f) =\pi^{av}_\theta (k) (\Phi_\theta (\xi\otimes f))$$
which proves the first statement  by continuity. We  omit the proof of the second statement as it is  similar and in fact   easier.
\end{proof}

\subsection{$\Gamma$-equivariant pseudodifferential operators}\label{subsect:dirac}

This subsection is devoted to a brief overview of the  pseudodifferential 
calculus 
relevant to our study.
All stated results are known and we therefore only sketch the proofs.

Let $\E_c$ be as before $C^{\infty,0}_c (\tM\times T,\widehat{E})$ endowed with its structure
of pre-Hilbert $\maA_c$-module.
Recall that if we complete the prehilbertian module $\E_c$ with respect to the regular norm on $\maA_c$ then we get a Hilbert $C^*$-module $\E_r$ over the regular $C^*$-algebra $\maA_r$. 
In the same way, completing $\maA_c$ with respect to the maximal $C^*$-norm yields a Hilbert 
$C^*$-module $\E_m$ over the maximal $C^*$-algebra $\maA_m$. 
We fix two vector bundles $E$ and $F$ over $V$ and we denote by $\what{E}$ and $\what{F}$ their pullbacks to $\tM\times T$ into $\Gamma$-equivariant vector bundles;
we let $\what{E}_\theta$ be the restriction of $\what{E}$ to $\tM_\theta$. We set,
as before, $\tM_\theta:= \tM_\theta$.

\begin{definition}\label{def:pseudo}
Let $P:C^{\infty,0}_c (\tM\times T,\widehat{E})\to C^{\infty,0} (\tM\times T,\widehat{F})$ be a linear map.
We shall say that $P$ defines a  
pseudodifferential operator of order $m$ on the monodromy groupoid $G$
 if there is a  continuous family  of order $m$ pseudofifferential operators $(P(\theta))_{\theta\in T}$,
 $$P_\theta: C^\infty_c  (\tM_\theta,\what{E}_\theta)
\to C^\infty (\tM_\theta,\what{F}_\theta)$$
satisfying:\\
(1) $(P\xi)(\tm,\theta)= (P_\theta \xi (\cdot,\theta))(\tm\times\{\theta\})$\\
(2) $P$ is $\Gamma$-equivariant: $R^*_\gamma P R^*_{\gamma^{-1}} = P$;\\
(3)  the Schwartz kernel of $P$, $K_P$, which can be thought of as a $\Gamma$-invariant distributional
section on $\tM\times \tM\times T$, is of $\Gamma$-compact support, i.e. the image of the support
in $(\tM\times \tM\times T)/\Gamma =:G$ is a compact set. 
\end{definition}
Notice that (2)
can be then  restated as: $P_{\gamma \theta} = \gamma P_\theta$ $\forall \theta \in T $, $\forall \gamma\in \Gamma$,
exactly as in the definition of the regular von Neumann algebra. The notion of continuity for families of pseudodifferential
operators is classical and will not be recalled here, see, for example, \cite{Benameur-Nistor-JFA},
\cite{NWX}, \cite{LaMoNi-Documenta}, \cite{Vassout-these}, \cite{Vassout-jfa}. 
Finally, because of the third condition $P$ maps 
$C^{\infty,0}_c (\tM\times T,\widehat{E})$ into $C^{\infty,0}_c (\tM\times T,\widehat{F})$.

Notice that a $\Gamma$-equivariant continuous family of differential operators
acting between the sections of two equivariant vector bundles is an example of
a pseudodifferential operator on $G$.

If $m\in\ZZ$,
we shall denote  by $\Psi^m_c(G;\what{E},\what{F})$ 
the space of pseudodifferential operators of  order $\leq m$ from $\what{E}$ to $\what{F}$ \footnote{The
notation for this space of operators is not unique: in \cite{LPETALE} it is denoted $\Psi^*_{\rtimes,c}(\tM\times T; \widehat{E},\widehat{F})$ with $\rtimes$ denoting equivariance
and $c$ denoting again  {\it of $\Gamma$-compact support}; in \cite{Moriyoshi-Natsume} it is simply denoted
as $\Psi_\Gamma^* (\widehat{E},\widehat{F})$} . We set
$$
\Psi^\infty_c(G;\what{E},\what{F}) := \bigcup_{m\in \ZZ} \Psi^m_c(G;\what{E},\what{F}) \text{ and } \Psi^{-\infty}_c (G;\what{E},\what{F}) := \bigcap_{m\in \ZZ} \Psi^m_c(G;\what{E},\what{F}).
$$

Using condition (3) , it is not difficult to check that the space $\Psi^\infty_c (G;\what{E},\what{E}) $ 
is a filtered algebra.
Moreover, assigning to $P$ its
formal adjoint $P^*= (P_\theta^*)_{\theta\in T}$ gives $\Psi^\infty_c (G;\what{E},\what{E})$ the structure of an involutive algebra; the formal adjoint is defined also for $P \in \Psi^m_c (G;\what{E},\what{F}) $ 
and it is then an alement in $\Psi^m_c (G;\what{F},\what{E})$. 

\begin{remark}\label{rk:psuedo=pseudo}
Notice  that Definition \ref{def:pseudo} fits into the general framework
of  pseudodifferential calculus on groupoids,
as developed by Connes and many others. More precisely,
let $P=(P_\theta)_{\theta\in T}$ be a pseudodifferential operator on $G$ as in Definition 
\ref{def:pseudo}. For any $\theta\in T$ and any $x=[\tm,\theta]\in L_\theta$  the diffeomorphism
$$\rho_{x,\theta}: \tM\rightarrow G^x=r^{-1}(x)\;\; \text{ given by } \;\; \rho_{x,\theta} (\tm^\prime)= [\tm,\tm^\prime,\theta]\,,$$
allows to define a pseudodifferential  operator $P_x$ on $G^x$ with coefficients in $s^* E$, viz. $P_x:= (\rho_{x,\theta}^{-1})^* \circ P_\theta
\circ (\rho_{x,\theta})^*$. It is easy to check that $P_x$ only depends on $x$ and that the family $(P_x)_{x\in V}$ is a pseudodifferential
operator on $G$ in the sense of Connes. Conversely if we are given now a pseudodifferential operator $(P_x)_{x\in V}$
in the sense of Connes, then a choice of a base point $m_0$ in $M$ allows to construct $P=(P_\theta)_{\theta\in T}$ satisfying
the assumptions of Definition  \ref{def:pseudo}, viz.
$P_\theta:= \rho_{x(\theta),\theta}^* \circ P_{x(\theta)} \circ ( \rho_{x(\theta),\theta}^{-1})^* $ with $x(\theta)=[\tm_0,\theta]$ and
$[\tm_0]=m_0$.\\
\end{remark}

\begin{remark}\label{rk:psuedo=pseudo-bis}
According to \cite{Co} a psedodifferential operator as in Connes, admits a well defined
distributional Schwartz kernel over $G$. It is easy to check that this Schwartz kernel coincides
with our $K_P$ when the two families correspond as in the previous remark.

\end{remark}

\begin{remark}\label{rk:vn=vn}
The construction explained in  remark \ref{rk:psuedo=pseudo}.
also allows to establish an identification between Connes' von Neumann
algebra \cite{Co} for the groupoid $G$ and our von Neumann algebra $W^*_\nu (G,E)$. It is easy to check that Connes' trace \cite{Co} corresponds
to our trace $\tau^\nu$  through this identification.
\end{remark}

\begin{lemma}
A  pseudodifferential operator $P$ of order $m$ yields  an $\maA_c$-linear map $\maP: \maE_c\to \maF_c$. Moreover, the following identity holds in $\maA_c$: $<\maP u, v>=<u,\maP^* v>$
$\forall u\in \maE_c$, $\forall v\in \maF_c$.
\end{lemma}

\begin{proof} 

Let $\xi\in \E_c$ and  let $f \in \maA_c$. By definition $(\xi f)(\cdot) = \sum_{\gamma} (R^*_{\gamma^{-1}} \xi)(\cdot) f( \gamma \pi(\cdot), \gamma)$ with $\pi:\tM\times T\to T$ the projection. Hence:
\begin{eqnarray*}
\maP  (\xi f) &=& \maP \left( \sum_{\gamma} (R^*_{\gamma^{-1}} \xi)(\cdot) f( \gamma \pi(\cdot),\gamma) \right)\\
&=& \sum_{\gamma} \left( \maP \left( R^*_{\gamma^{-1}} \xi \right) (\cdot) \right) f(\gamma \pi(\cdot),\gamma)\\
&=&  \sum_{\gamma}  \left( R^*_{\gamma^{-1}} \maP \xi \right)(\cdot)  f(\gamma \pi(\cdot),\gamma)=(\maP\xi)f
\end{eqnarray*}
where in the second equality  we have used the fact that $\maP$ commutes with multiplication by functions in  $C (T)$ 
(indeed, $\maP$ is given by  a continuous  family) and in the third equality we have used the $\Gamma$-equivariance. The equality $<\maP u, v>=<u,\maP^* v>$ is established in a straightforward
way.

\end{proof}

\begin{proposition}\ Let $\maP$ be a pseudodifferential operator of order $m$ between $\maE_c$ and $\maF_c$. Then we have:
\begin{enumerate}
\item If $m\leq 0$ then $\maP$ extends to a bounded adjointable 
$\maA_m$-linear operator  $\overline{\maP}_m$ from $\maE_m$ to $\maF_m$ and to a
bounded adjointable 
 $\maA_r$-linear operator  $\overline{\maP}_r$ from $\maE_r$ to $\maF_r$.
\item 
 If $m<0$, then 
 $\overline{\maP}_m$ 
 is an $\maA_m$-linear compact operator from  $\maE_m$ to $\maF_m$
and 
$\overline{\maP}_r$   is an $\maA_r$-linear compact operator from $\maE_r$ to $\maF_r$

\end{enumerate}
\end{proposition}

\begin{proof}
We only sketch the arguments, following \cite{Vassout-these}. For simplicity we take
$E$ and $F$ to be the trivial line bundles, so that $\maE_c=\maF_c$. We give the proof
for the maximal completion, the proof for the regular completion being the same.\\
For the first item, one applies the classical argument of H\"ormander, see for example
\cite{shubin}, reducing the continuity
of order zero pseudodifferential operators to that of the smoothing operators.
We omit the details.\\
For the second item, one starts with $\maP$ of order $<-n$, with $n$ equal to the dimension of $\tM$.
Then $\maP$ is given by integration against a continuous compactly supported
element in $G$; in other words $\maP= \chi(K)$, with $K\in C_c (G)$.
We already know that such an element extends to a  compact operator   $\overline{\maP}$ on $\maE_m$, see \ref{prop:compact}. If $\maP$ is of order less than $-n/2$ then we consider $\maQ:= \maP^*  \maP$
which is of order less then $-n$ and symmetric. We know that $\maQ$ extends to
a (compact) bounded  operator on $\maE_m$; thus if $f\in \maE_c$ then, in particular, 
$\|\maP f\|^2 \leq C \| f \|^2$ which means that $\maP$ extends to  a bounded operator 
$\overline{\maP}$ on $\maE_m$.
Similarly $\maP^*$ extends to a bounded operator $\overline{\maP^*}$ and by density
we obtain that $\overline{\maP}$ is adjointable with adjoint equal to $\overline{\maP^*}$.
Now, again by density, we have $\overline{\maQ}= \overline{\maP}\,^* \overline{\maP}$; 
thus we can take the square root of $\overline{\maQ}$ which will be again compact
since $\overline{\maQ}$ is. Using the polar decomposition for $\overline{\maP}$
we can finally conclude that $\overline{\maP}$ is compact which is what we need to prove.\\
If the order of $\maP$ is $m<0$ then we fix $\ell\in \NN$ such that $m 2^\ell<-n$; then we  proceed inductively, considering $(\maP^* \maP)^{2^\ell}$ and applying the above argument.\\




\end{proof}

Let $P=(P_\theta)_{\theta\in T}$ be an element in  $\Psi^\ell_c (G)$; its 
principal symbol $\sigma_\ell (P)$ defines a $\Gamma$-equivariant function
on the vertical cotangent bundle $T^*_V (\tM\times T)$
 to the trivial fibration $\tM\times T\to T$; equivalently,
$\sigma_\ell (P)$ is a function on the longitudinal cotangent bundle $T^*\maF$  to the 
foliation $(V,\maF)$.
If, more generally, $P\in \Psi^\ell_c (G; \what{E},\what{F})$, then its principal symbol
will be a $\Gamma$-equivariant section of the bundle $\Hom (\pi_V^* (\what{E}),
\pi^*_V (\what{F})):=$
$\pi_V^* (\what{E^*})\otimes
\pi^*_V (\what{F})$ with $\pi_V:T^*_V (\tM\times T)\to (\tM\times T)$ the natural projection;
equivalently, $\sigma_\ell (P)$ is a section of the bundle $\Hom (\pi_{\maF}^* E, \pi_{\maF}^* F)$
over the longitudinal cotangent bundle $\pi_{\maF}: T^* \maF\to V$.
We shall say that $P$ is elliptic
if its principal symbol $\sigma_\ell (P)$ is invertible on non-zero cotangent vectors. We end this
Subsection by stating the following  fundamental and classical result whose proof can be found,
for example 
in the work of Connes \cite{Co-LNM}, see also \cite{MS}. 
(Notice that in this particular case 
 the proof can be easily done directly, mimicking the  classic one
 on a closed compact manifold.)

\begin{theorem}\label{th:parametrix}
Let $P\in \Psi^\ell_c (G; \what{E},\what{F})$ be elliptic; then there exists $Q\in \Psi^{-\ell}_c (G; \what{F},\what{E})$ such that 
\begin{equation}\label{eq:parametrix}
\Id-P Q\in \Psi^{-\infty}_c(G;\what{F},\what{F}) \,,\quad \quad
\Id - Q P \in \Psi^{-\infty}_c(G;\what{E},\what{E})
\end{equation}
\end{theorem}

Notice that in our definition elements
in $\Psi^{-\infty}_c$ are of $\Gamma$-compact support: this applies 
in particular to both $S:= \Id- PQ$ and $R:= \Id-QP$ .

We end this subsection by observing that it is also possible to introduce Sobolev
modules $\maE^{(k)}$ and prove the usual properties of pseudodifferential operators,
see \cite{Vassout-jfa}. For simplicity we consider the case $k\in \NN$.
In order to give the definition, we fix an elliptic differential operator of order $k$, $P$;
for example $P=D^k$, with $D$ a Dirac type operator. This is a regular
unbounded operator  (see the next subsection).
We consider the domain of its extension ${\rm Dom}\maP_m$ and
we endow it with the $\A_m$-valued scalar product
$$<s,t>_k := <s,t> + <\maP_m s, \maP_m t>$$
This defines the Sobolev module of order $k$, $\maE^{(k)}$.
One can prove for these modules the usual properties:
\begin{itemize}
\item different choices of $P$ yield compatible Hilbert module structures;
\item if $k>\ell$ we have  $\maE^{(k)}\hookrightarrow \maE^{(\ell)}$ and the inclusion
in $\maA_m$-compact
\item if $R\in \Psi^{m}_c(G,E)$ then $R$ extends to a bounded operator $\maE^{(k)}
\rightarrow \maE^{(k-m)}$
\end{itemize}
Since we shall make little use of these properties, we omit the proofs.





\subsection{Functional calculus for Dirac operators}\label{subsect:functional-dirac}

Let $\tD=(\tD_\theta)_{\theta\in T}$ be a $\Gamma$-equivariant family of
Dirac-type operators acting on the sections of a $\Gamma$-equivariant 
vertical  hermitian Clifford module $\what{E}$ endowed with a $\Gamma$-equivariant
connection. We shall make the usual assumptions on the connection and on the
Clifford action ensuring that each $\tD_\theta$ is formally self-adjoint.
Recall that $\tD=(\tD_\theta)_{\theta\in T}\in\Psi^1_c (G;\what{E})$ and that
 $\tD$  induces a 
$\maA_c$-linear operator on 
$\E_c$ that we have denoted by $\maD$.



\begin{proposition}
The operator $\maD$ is closable in $\maE_r$ and in $\maE_m$. Moreover, the closures $\maD_r$ and $\maD_m$ on the Hilbert modules $\maE_r$ and $\maE_m$ respectively, are regular and self-adjoint operators.
\end{proposition}

\begin{proof}
We give a classical proof based on general results described for instance in \cite{Vassout-these}.  
Since the densely defined operator $\maD$ is formally self-adjoint, it is  closable with symmetric closures in $\maE_r$ and $\maE_m$ respectively. 
Let $\tQ\in \Psi^{-1}_c(G, \what{E}) $ be a formally self-adjoint  parametrix for $\tD$:
$$\Id - \tD \tQ= \tilde{S}\,,\quad \quad  \Id-\tQ \tD= \tilde{R}\,.$$ 
 For simplicity, we denote by $\pi$ the regular or the maximal representation,
 by $\maE_\pi$ the corresponding Hilbert module and by $\maD_\pi$ the closure of $\maD$.
  Since $\tQ$ has negative order, it extends into a bounded operator on $\maE_\pi$, denoted by 
$\overline{\maQ}_\pi$, or simply by  $\maQ_\pi$. 
On the other hand, we know that the zero-th order pseudodifferential operator $\tD \tQ$ 
extends to a bounded $\maA_\pi$-linear operator on $\maE_\pi$.
 If $\xi$ belongs to the domain of this closure (which is $\maE_\pi$) then there exists $\xi_n$ in $C_c^{\infty,0} (\tM\times T, \what{E})$ converging in the $\pi$-norm to $\xi$ and such that $(\tD\tQ)\xi_n$ is convergent in the $\pi$ norm. 
 We deduce that $\maQ_\pi (\xi)$ is well defined and is the limit of $\tQ\xi_n$. Hence we deduce that $\maQ_\pi \xi$ belongs to the domain of $\maD_\pi$
 and that ${\rm Im} \maQ_\pi\subset \Dom \maD_\pi$. Hence, $\maD_\pi \maQ_\pi$ is a bounded operator which coincides with the extension of $\tD \tQ$ and we have with obvious notation,
$$
\maD_\pi \maQ_\pi = I - \maS_\pi\,, $$
so  $\maQ_\pi^* \maD_\pi^* \subset (\maD_\pi \maQ_\pi)^* = I - \maS_\pi^*$
and hence  $\Dom (\maD_\pi^*) \subset {\rm Im} (\maQ_\pi^*) +  {\rm Im} (\maS_\pi^*)$. Since $\maQ_\pi$ is self-adjoint  we deduce that 
$$
\Dom(\maD_\pi^*)  \subset {\rm Im}(\maQ_\pi) + {\rm Im} (\maS_\pi^*) \subset \Dom(\maD_\pi).
$$
The last inclusion is a consequence 
of the fact that $\maS_\pi^*$ is induced by a smoothing $\Gamma$-compactly supported operator.
So $\maD_\pi$ is self-adjoint.
Now, the graph of $\maD_\pi$, $G(\maD_\pi)$, is given by
$$
G(\maD_\pi) = \{ (\maQ_\pi (\eta) + \maS_\pi (\eta'), \maD_\pi (\maQ_\pi(\eta))+ \maD_\pi(\maS_\pi(\eta')), \eta, \eta'\in \maE_\pi\}.
$$
Hence $G(\maD_\pi)$, which is closed in $\maE_\pi\times \maE_\pi$, coincides with the image of a bounded morphism  $U$ of $\maA_\pi$-modules given by
$$
U= \left(\begin{array}{cc} \maQ_\pi & \maS_\pi \\ \maD_\pi \maQ_\pi & \maD_\pi \maS_\pi \end{array} \right)
$$
Now, as a general fact, the image of such morphism, when closed, is always orthocomplemented.
Thus
$\maD_\pi$ is regular. 
\end{proof}

Recall 
that we established  in Lemma \ref{lemma:identification}
 isomorphisms of Hilbert spaces 
$$\Phi_\theta: 
\maE_m\otimes_{\pi_\theta^{{\rm av}}} \ell^2(\Gamma\theta)\rightarrow L^2(L_\theta,E)\,,\quad
\Psi_\theta: \maE_m\otimes_{\pi_\theta^{reg}} \ell^2(\Gamma)\rightarrow L^2(\tM_\theta,\widehat{E})
 .$$



%

\begin{proposition}\label{F-Comp}
Let $\psi:\RR \to \CC$ be a continuous bounded function. Then for any $\theta\in T$, the bounded operator, acting on $L^2(L_\theta, E)$, given by
$$
\Phi_\theta \circ \left[ \psi (\maD_m) \otimes_{\pi^{av}_\theta} I_{\ell^2(\Gamma\theta)} \right]\circ \Phi_\theta^{-1},
$$
coincides with the operator $\psi(D_{L_\theta})$ where $D_{L_\theta}$ is our Dirac type operator acting on the leaf $L_\theta$.\\
In the same way the operator, acting on $L^2(\tM_\theta,\what{E}_\theta)$, given by 
$$
\Psi_\theta \circ \left[ \psi (\maD_m) \otimes_{\pi^{reg}_\theta} I_{\ell^2(\Gamma)} \right]\circ \Psi_\theta^{-1},
$$
coincides with the operator $\psi(\tD_\theta)$.

\end{proposition}

\begin{proof}
We prove only the first result, the proof of the second is similar.
Since the operator $\maD_m$ is a regular self-adjoint operator, its continuous functional calculus is well defined. See \cite{Vassout-these}. Let $\xi\in \maE_c$ and let $f\in C_c[\Gamma\theta]$, then we have
\begin{eqnarray*}
\Phi_\theta (\maD_m (\xi)\otimes f) (\tm, \theta) & = &  \sum_{\gamma\in \Gamma} f(\gamma\theta) (\maD \xi) (\tm\gamma^{-1}, \gamma\theta)\\
& = & \sum_{\gamma\in \Gamma} f(\gamma\theta) [R_{\gamma^{-1}}^* (\maD \xi)] (\tm,\theta)\\
& = & \sum_{\gamma\in \Gamma} f(\gamma\theta) \maD(R_{\gamma^{-1}}^* \xi) (\tm,\theta)
\end{eqnarray*}
On the other hand, the action of the operator $D_{L_\theta}$ on the image of $\Phi_\theta$ is given by
$$
(D_{L_\theta}\circ \Phi_\theta) (\xi\circ f) (\tm, \theta) = \sum_{\gamma\in \Gamma} f(\gamma\theta) {\tilde D}_\theta ([\gamma^{-1}\xi]_\theta) (\tm). 
$$
Since by definition of $\maD$ we have $\maD(\gamma^{-1} \xi) (\tm,\theta) = {\tilde D}_\theta ([\gamma^{-1}\xi]_\theta) (\tm)$ we obtain that
$$
\Phi_\theta \circ (\maD_m \otimes I)\circ \Phi_\theta^{-1} = D_{L_\theta}.
$$
If $\psi$ is as above then we get as a consequence of the definition of  functional calculus,
\begin{eqnarray*}
\psi (D_{L_\theta}) & = & \psi\left( \Phi_\theta \circ (\maD_m \otimes I)\circ \Phi_\theta^{-1} \right)\\
& = & \Phi_\theta \circ \psi (\maD_m \otimes I)\circ \Phi_\theta^{-1}
\end{eqnarray*}
 By uniqueness of the functional calculus we also deduce that $
\psi (\maD_m \otimes I) = \psi(\maD_m) \otimes I,
$
and hence the proof is complete.
\end{proof}

Before proving the main result of this Subection, we 
recall two technical results about trace class operators.
First we establish two useful Lemmas. The first one is classical
and generalizes \cite{shubin} Proposition A.3.2 while  the second one is an easy extension of 
similar results for coverings established in \cite{Atiyah-covering}.

\begin{lemma}\label{lemma:atiyah2}
Let $S\in W^*_\nu(G,E)$; then
the following statements are equivalent:
\begin{itemize}
\item $S$ is $\tau^\nu$ Hilbert-Schmidt (i.e. $\tau^\nu (S^*  S)<+\infty)$;
\item there exists a measurable section $K_S$ of $\END (E)$ over $G$ such that
for $\nu$-almost every $\theta$ the operator $S_\theta$ is given on $L^2 (\tM_\theta,\what{E}_\theta)$ by
$(S_\theta \xi)(\tm)=\int_{\tM} K_S (\tm,\tm^\prime,\theta) \xi(\tm^\prime) d\tm^\prime$, with
$$\int_{\tM\times F\times T} \tr 
\left( K_{S}(\tm,\tm^\prime,\theta)^* 
K_{S}(\tm,\tm^\prime,\theta) \right) \,d\tm\,d\tm^\prime\,d\nu(\theta)\,<\,+\infty
$$
where we interpret
$K_S$ as a $\Gamma$-equivariant section on $\tM\times\tM\times T$.
\end{itemize}
Moreover in this case the $\tau^\nu$ Hilbert-Schmidt norm of $S$, 
$\|S\|^2_{\nu-{\rm HS}}:=\tau^\nu (S^*S)$,
is given by
$$\|S\|^2_{\nu-{\rm HS}}=\int_{\tM\times F\times T} \tr 
\left( K_{S}(\tm,\tm^\prime,\theta)^* 
K_{S}(\tm,\tm^\prime,\theta) \right) \,d\tm\,d\tm^\prime\,d\nu(\theta)\,.$$
\end{lemma}

\begin{proof}
We have, by definition,
$$\|S\|^2_{\nu-{\rm HS}}=\tau^\nu (S^* S)=\int_T \|S_\theta M_\chi\|^2_{{\rm HS}}\,d\nu (\theta)$$
where the integrand involves the usual Hilbert-Schmidt norm in $L^2 (\tM_\theta,\what{E}_\theta)$.
Therefore the proof is easily deduced  using   \cite{shubin}[page 251]
\end{proof}

\begin{lemma}\label{lemma:atiyah1}
Let $S$ be a positive selfadjoint operator in $W^*_\nu(G,E)$; 
then
the following statements are equivalent:
\begin{itemize}
\item $\tau^\nu (S) < +\infty$;
\item for any smooth compactly supported function $\phi$ on $\tM$, the measurable function
$$T\ni \theta\longrightarrow \Tr (M_{\overline{\phi}}\circ S_\theta \circ M_\phi)$$
is $\nu$-integrable on $T$, where the trace is the usual trace for bounded  operators on the
Hilbert space $L^2 (\tM, \what{E})$;
\item for any smooth compactly supported function $\phi$ on $\tM$,
the  function
$$T\ni \theta\longrightarrow \|S^{1/2}_\theta \circ M_\phi \|_{{\rm HS}}^2$$
is $\nu$-integrable on $T$.
\end{itemize}
\end{lemma}

\begin{proof}
We follow the techniques in \cite{Atiyah-covering} and use Lemma \ref{lemma:atiyah2}. The second and third items are clearly equivalent. Assume that $\tau^\nu (S) < +\infty$ and let $\phi$ be a smooth compactly supported function on $\tM$ with uniform norm $\|\phi\|_\infty$. We let $\Gamma_\phi$ be a finite subset of $\Gamma$ such that the support of $\phi$ lies in the union $\cup_{\gamma\in \Gamma_\phi} F\gamma$. Here $F$ is a fundamenal domain as before. Then $S^{1/2}$ is $\tau^\nu$ Hilbert-Schmidt and if $K_{S^{1/2}}$ is its Schwartz kernel, then we easily deduce
\begin{eqnarray*}
& & \int_{\tM\times \tM \times T} |\phi (\tm^\prime)|^2 \,\, \tr \left( K_{S^{1/2}}(\tm, \tm^\prime, \theta)^* K_{S^{1/2}}(\tm, \tm^\prime, \theta) \right) \,\, d\tm\, d\tm^\prime \, d\nu(\theta) \\
& = &  \sum_{\gamma\in\Gamma_\phi} \int_{\tM\times F\gamma\times T} |\phi (\tm^\prime)|^2 \,\,
\tr \left( K_{S^{1/2}}(\tm, \tm^\prime, \theta)^* K_{S^{1/2}}(\tm, \tm^\prime, \theta) \right) \, d\tm\, d\tm^\prime \, d\nu(\theta) \\
& \leq & \|\phi\|^2_\infty \times \sum_{\gamma\in\Gamma_\phi} \int_{\tM\times F\gamma\times T} 
\tr \left( K_{S^{1/2}}(\tm, \tm^\prime, \theta)^* K_{S^{1/2}}(\tm, \tm^\prime, \theta) \right) \, d\tm\, d\tm^\prime \, d\nu(\theta) \\
& = & \|\phi\|^2_\infty \times \sum_{\gamma\in\Gamma_\phi} \int_{\tM\times F \times T} 
\tr \left( K_{S^{1/2}}(\tm \gamma^{-1}, \tm^\prime, \gamma \theta)^* K_{S^{1/2}}(\tm \gamma^{-1}, \tm^\prime, \gamma\theta) \right) \, d\tm\, d\tm^\prime \, d\nu(\theta) \\
& = & \|\phi\|^2_\infty \times \text{\rm Card}(\Gamma_\phi) \times \tau^\nu (S)\quad <+\infty.
\end{eqnarray*}
Conversely, let $\phi$ be any nonnegative smooth compactly supported function on $\tM$ such that $\phi \chi = \chi$, where $\chi$ is the characteristic function of $F$. Then we have
\begin{eqnarray*}
\tau^\nu (S) & = &
\int_T \Tr (M_\chi \circ S_\theta \circ M_\chi) \,d\nu (\theta)\\
&=& \int_T \Tr (M_\chi \circ M_\phi  \circ S_\theta \circ M_\phi  \circ M_\chi) \,d\nu (\theta)\\
&\leq&   \int_T \Tr (M_\phi  \circ S_\theta \circ M_\phi ) \,d\nu (\theta) < +\infty 
\end{eqnarray*}
\end{proof}

\begin{proposition}\label{TraceNoyau}
Let $S=(S_\theta)_{\theta\in T}$ be an element of the von Neumann algebra $W^*_\nu (G;E)$. We assume 
that  
 $S_\theta$ is an integral operator with smooth kernel for any $\theta$ in $T$
and that 
the  resulting Schwartz kernel $K_S$ is  a Borel bounded section over $G$.
\begin{itemize}
\item If $S$ is positive and self-adjoint, then 
$S$ is $\tau^\nu$ trace class and we have
\begin{equation}\label{eq:integral-trace}
\tau^\nu (S) = \int_{F\times T} \tr (K_S (\tm,\tm,\theta)) d\tm d\nu (\theta),
\end{equation}
where $F$ is a fundamental domain in $\tM$ and where in the right hand side we interpret
$K(S)$ as a $\Gamma$-equivariant section on $\tM\times\tM\times T$.
\item 
If $S$ is assumed to be $\tau^\nu$ trace class, then 
formula \eqref{eq:integral-trace} holds.
\end{itemize}
\end{proposition}

\begin{proof}
Let us prove the first item.  Let $\phi$ be a smooth compactly supported function on
$\tM$. The operator $M_\phi \circ S_\theta \circ M_{\overline{\phi}}$ 
acting on $L^2 (\tM_\theta,\what{E})$ has a smooth compactly
supported Schwartz kernel and is therefore  trace class with
$$\Tr (M_\phi \circ S_\theta \circ M_{\overline{\phi}})=\int_{\tM_\theta} |\phi  (\tm)|^2 K_S (\tm,\tm,\theta)
\,d\tm\,.$$ Since $K_S$ is bounded as a section over $G$ and since $\nu$ is a borelian measure,
 we have
$$\int_T \Tr (M_\phi \circ S_\theta \circ M_{\overline{\phi}}) d\nu(\theta) \, <\, + \infty\,.$$
This shows, using Lemma \ref{lemma:atiyah1}, that $S$ is $\tau^\nu$ trace class and also that
$S^{1/2}$ is $\tau^\nu$ Hilbert-Schmidt. By Lemma \ref{lemma:atiyah2} we deduce that 
the $S^{1/2}$  is an integral operator with measurable Schwartz kernel $K_{S^{1/2}}$
satisfying
$$\| S^{1/2} \|^2_{{\rm HS}}:=   \int_{\tM\times F \times T} \tr \left( K_{S^{1/2}}(\tm,\tm^\prime,\theta)^* 
K_{S^{1/2}}(\tm,\tm^\prime,\theta) \right) d\tm\,d\tm^\prime d\nu (\theta) \,<\,+\infty\,.$$
On the other hand we also have 
\begin{eqnarray*}
K_S (\tm,\tm,\theta) & = & \int_{\tM}  K_{S^{1/2}}(\tm,\tm^\prime,\theta)  K_{S^{1/2}}(\tm^\prime,\tm,\theta)
\,d\tm\\
& = & \int_{\tM} K_{S^{1/2}}(\tm,\tm^\prime,\theta)  K_{S^{1/2}}(\tm, \tm^\prime,\theta)^*
\,d\tm
\end{eqnarray*}
The last equality employs the fact that $S^{1/2}$ is selfadjoint.
Taking pointwise traces we get:
\begin{eqnarray*}\tr K_S (\tm,\tm,\theta) &=& \int_{\tM} \tr\left( K_{S^{1/2}}(\tm,\tm^\prime,\theta)  K_{S^{1/2}}(\tm, \tm^\prime,\theta)^* \right)
\,d\tm\\
&=& 
\int_{\tM} \tr\left(  K_{S^{1/2}}(\tm, \tm^\prime,\theta)^* K_{S^{1/2}}(\tm,\tm^\prime,\theta)  \right)
\,d\tm \,.
\end{eqnarray*}
Therefore 
$$\tau^\nu (S)= \| S ^{1/2} \|^2_{{\rm HS}}= \int_{F\times T} \tr K_S (\tm,\tm,\theta) d\tm d\nu(\theta)\,.$$
This finishes the proof of the first item.\\
Regarding the second item, assume now that $S$ is $\tau^\nu$ trace class i.e. $\tau^\nu (|S|)$
is finite. Write $S=U |S|$ for the polar decomposition of $S$ in $W^*_\nu (G,E)$. Then the operators
$|S|^{1/2}$ and $U |S|^{1/2}$ are $\tau^\nu$ Hilbert-Schmidt and thus have $L^2$ Schwartz
kernels $K_{|S|^{1/2}}$ and $K_{U |S|^{1/2}}$. Using Lemma \ref{lemma:atiyah2}
and the polarization identity we deduce:
\begin{eqnarray*}
<U |S|^{1/2},|S|^{1/2}>_{{\rm HS}} &=& \int_{\tM\times F \times T} \tr K_{U |S|^{1/2}}(\tm,\tm^\prime,\theta)  K_{|S|^{1/2}}(\tm,\tm^\prime,\theta)^* \,d\tm\,d\tm^\prime \, d\nu(\theta)\\
&=& \int_{\tM\times F \times T} \tr K_{U |S|^{1/2}}(\tm,\tm^\prime,\theta)  K_{|S|^{1/2}}(\tm^\prime,\tm,\theta) \,d\tm\,d\tm^\prime \, d\nu(\theta)\\
&=& \int_{ F \times T} \tr K_S (\tm,\tm,\theta) \,d\tm \, d\nu(\theta)\,.
\end{eqnarray*}
Hence $$\tau^\nu (S)=<U |S|^{1/2},|S|^{1/2}>_{{\rm HS}} = \int_{ F \times T} \tr K_S (\tm,\tm,\theta) \,d\tm \, d\nu(\theta)\,.$$
The proof is complete.
\end{proof}

\begin{remark}\label{rk:smoothing}
A  proof similar to the one given above shows,  as
in  \cite{Atiyah-covering} (Proposition 4.16), that if $R=(R_\theta)_{\theta\in T}$ has a continuous 
(or even Borel bounded)  leafwise smooth
Schwartz kernel with $\Gamma$-compact support, then $R$ is $\tau^\nu$ trace class with 
$\tau^\nu (R)=\int_{F\times T} \tr K_R (\tm, \tm, \theta) d\tm \,d\nu(\theta)$. \\
A similar statement holds for a leafwise operator in $W^*_\nu (V,\maF;E)$ with a Borel bounded leafwise smooth Schwartz kernel which is supported within a uniform $C$-neighbourhood, $C\in \RR, C>0$,
 of the diagonal
of every leaf.
\end{remark}

\begin{proposition}\label{Rapid}
Let $\psi:\RR\to\CC$ be a measurable rapidly decreasing  function. The
  the operator
  $\psi (\tD):=(\psi(\tD_\theta))_{\theta\in T}$ 
  satisfies the assumption of Proposition \ref{TraceNoyau} (second item). In particular
  $\psi (\tD)$ has a bounded fiberwise-smooth Schwartz kernel $K_\psi$ and we
   have $$
\tau^\nu (\psi(\tD)) = \int_{F\times T} \tr( K_\psi [\tm, \tm, \theta]) d\tm\,\, d\nu (\theta).
$$ 
\end{proposition}

\begin{proof}
Using \cite{MS}, Theorem 7.36 (which is in fact valid for any measurable rapidly
decreasing function) we know that $K_\psi$ is bounded and fiberwise smooth 
and that $\psi (\tD)\in W^*_\nu (G,E)$ . Therefore it remains to show that $\psi (\tD)$ is $\tau^\nu$ trace class
since then we can simply apply Proposition \ref{TraceNoyau} (second item).
But $|\psi (\tD)|=|\psi| (\tD)|$ and $|\psi|$ is a measurable rapidly decreasing function; therefore
$|\psi (\tD)|$ has a bounded fiberwise smooth Schwartz and thus satisfies the assumptions
of Proposition \ref{TraceNoyau} (first item). We conclude that $\psi (\tD)$ is $\tau^\nu$ trace class.


%
\end{proof}


A statement similar to the one just proved holds for the leafwise Dirac-type  operator $D:=(D_L)_{L\in V/\maF}$.
In order to keep this paper  to a reasonable size we state the corresponding proposition without
proof. 
See \cite{roe-book}

\begin{proposition}\label{Rapid-down}
Let $\psi:\RR\to\CC$ be a measurable rapidly decreasing  function. Then
  the operator
  $\psi (D):=(\psi(D_L))_{L\in V/\maF}$ 
  is $\tau^\nu_{\maF}$ trace class, has a leafwise smooth Schwartz kernel $K_\psi$
  which is bounded as a measurable section over the equivalence relation $\cup_{L\in V/\maF} \,L\times L$,and we
   have $$
\tau^\nu_{\maF}  (\psi(D)) = \int_{F\times T} \tr(K_\psi ([\tm, \theta], [\tm,\theta])) d\tm\,\, d\nu (\theta) $$ 
where now $F\times T$ is viewed as a subset in $V$.
\end{proposition}

We are now in position to prove the main results of this section.

\begin{theorem}\label{theo:functional-calculus}
Let for simplicity $\psi:\RR\to \CC$ be a Schwartz class function. Then $\psi(\maD_m)\in \maK_{\maA_m} (\maE_m)$ and the element $\chi_m^{-1} (\psi(\maD_m))\in \maB_m^E$ admits a finite $\tau^\nu_{av}$ trace and also a finite $\tau^\nu_{reg}$ trace. Moreover
\begin{itemize}
\item $\tau^\nu_{av} ( \chi_m^{-1} (\psi(\maD_m)) = \tau^\nu_{\maF} \left[(\psi(D_L))_{L\in V/F}\right]$ where $(\psi(D_L))_{L\in V/F}$ is the corresponding element in the leafwise von Neumann algebra $W^*_{\nu} (V,\maF;E)$ and $\tau^\nu_{\maF}$ is the trace on this von Neumann algebra as defined in Subsection \ref{subsec:traces}.
\item $\tau^\nu_{reg} ( \chi_m^{-1} (\psi(\maD_m)) = \tau^\nu \left[(\psi({\tilde D}_\theta))_{\theta\in T}\right]$ where $(\psi({\tilde D}_\theta))_{\theta\in T}$ is the corresponding element in the regular von Neumann algebra $W^*_{\nu} (G,E)$ and $\tau^\nu$ is the trace on this von Neumann algebra as defined in Subsection \ref{subsec:traces}.
\end{itemize}
\end{theorem}
\begin{proof}
%
We know from Corollary \ref{prop:composition-of-traces} that $\tau^{\nu}_{av} = \tau^\nu_{\maF} \circ \pi^{av}$. Therefore
\begin{eqnarray*}
\tau^{\nu}_{av} ( \chi_m^{-1} (|\psi(\maD_m)|)) & = & \tau^\nu_{\maF} \left[ (\pi^{av}\circ \chi_m^{-1})  (|\psi(\maD_m)|)\right]\\
& = & \tau^\nu_{\maF} \left((\Phi_\theta\circ \left[ |\psi|(\maD_m)\otimes_{\pi^{av}_\theta} I\right] \circ \Phi_\theta^{-1})_{\theta\in T} \right)
\end{eqnarray*}
The last equality is a consequence of Proposition \ref{PhiComp}. Now, using Proposition 
\ref{F-Comp}, we finally deduce
$$
\tau^{\nu}_{av} ( \chi_m^{-1} (|\psi(\maD_m)|)) 
= \tau^\nu_{\maF} \left( (|\psi|(D_L))_{L\in V/F}  \right) \quad <+\infty.
$$
Hence we see from Proposition   \ref{prop:compact} that $\chi_m^{-1} (|\psi(\maD_m)|)$ is trace class and the same computation with $\psi$ instead of $|\psi|$ finishes the proof of the first item.
The second item is proved repeating the same argument.
\end{proof}


\section{Index theory}\label{Section.Index}

Let $\tM$, $\Gamma$, and $T$ be as in the previous sections and let $(V,\maF)$, with $V=\tM\times_\Gamma T$, the
associated foliated bundle. We assume in this section only that the manifold $M$ is even dimensional and hence that the leaves of our foliation are even dimensional. Let $E$ be a continuous longitudinally smooth hermitian vector bundle
on $V$ and let $\widehat{E}$ be its lift to $\tM\times T$.
Let $\tD=(\tD_\theta)_{\theta\in T}$ be  as in the previous section  a $\Gamma$-equivariant 
continuous family of Dirac-type operators.
The bundle $E$ is $\ZZ_2$-graded, $E=E^+ \oplus E^-$, and the operator $\tD$ is odd and essentially self-adjoint, i.e. 
$$\tD_\theta= \left(\begin{array}{cc} 0 & \tD^-_\theta\\ \tD^+_\theta& 0 \end{array} \right)\quad \forall \theta\in T$$
and $(\tD^-_\theta)^*=\tD^+_\theta$.
Let $D:=(D_L)_{L\in V/F}$ be the longitudinal operator induced by $\tD$ on the leaves of the foliation $(V,\maF)$. 

\subsection{The numeric index}
We consider for each $\theta$ the orthogonal projection $\tilde{\Pi}_\theta^\pm$ onto the $L^2$-null space
of the operator $\tD^\pm_\theta$. Similarly, on each leaf $L$, we consider the orthogonal
projections $\Pi^\pm_L$ onto the $L^2$-null space of the operator $D_L$. It is well known that these orthogonal projections are smoothing
operators, but of course are not localized in a compact neighborhood of the unit space $V$, viewed as a subspace of the graph of the foliation equivalence relation.

\begin{proposition}\
\begin{itemize}
\item 
The  family $\tilde{\Pi}^\pm := (\tilde{\Pi}_\theta^\pm)_{\theta\in T}$ belongs to the regular von Neumann algebra
$W^*_\nu (G,E^\pm)$. Moreover they are  $\tau^\nu$ trace class operators.
\item The family $\Pi^\pm:= (\Pi^\pm_L)_{L\in V/F}$ belongs to the leafwise von Neumann algebra $W^*_\nu (V,\maF;E^\pm)$. 
Moreover they are $\tau^\nu_\maF$ trace class operators.
\end{itemize} 
\end{proposition}

\begin{proof}
As we have already mentioned, for any Borel bounded function $f:\RR \to \CC$, the operator $f(\tD)$ (respectively $f(D)$) belongs to the von Neumann algebra $W^*_\nu (G,E)$ (to the von Neumann algebra $W^*_\nu (V,\maF;E)$). Hence, $\tilde{\Pi}^\pm$ belongs to $W^*_\nu (G,E^\pm)$ and $\Pi^\pm$ belongs to $W^*_\nu (V,\maF;E^\pm)$. 

Recall on the other hand from Propositions \ref{Rapid} 
\ref{Rapid-down}
that $e^{-\tD^2}$ is $\tau^\nu$ trace class and that $e^{-D^2}$ is $\tau_\maF^\nu$ trace class. Hence the proof is complete since 
$$
\tilde{\Pi} = \tilde{\Pi} e^{-\tD^2} \quad \text{ and } \quad \Pi = \Pi e^{-D^2}.
$$
\end{proof}

\begin{definition}
We define the monodromy index of $\tD$ as 
\begin{equation}\label{index-up}
\ind^\nu_{{\rm up}}(\tD):=\tau^\nu (\tilde{\Pi}^+) -\tau^\nu(\tilde{\Pi}^-)
\end{equation}
We define the leafwise index of $D$ as
\begin{equation}\label{index-down}
\ind^\nu_{{\rm down}}(D):=\tau^\nu_{\maF} (\Pi^+) -\tau^\nu_{\maF} (\Pi^-)
\end{equation}
\end{definition}

As  $\tD^+$ is elliptic,   we can find a $\Gamma$-equivariant family of
parametrices $\tilde{Q}:= (\tilde{Q}_\theta)_{\theta\in T}$ of $\Gamma$-compact support
with remainders $\tilde{R}_+$ and $\tilde{R}_-$; the remainder families  are $\Gamma$-equivariant,
smoothing and of $\Gamma$-compact support, i.e.
$$
\tR_+ = I - \tQ \tD^+ \quad \text{ and } \quad \tR_- = I -  \tD^+\tQ\,;\quad \tR_\pm \in \Psi^{-\infty}_c
(G,E^\pm)\,.
$$
We know that $\tR_\pm$ are both $\tau^\nu$ trace class.
Let $Q, R_+, R_-$ be the longitudinal operators induced on $(V,\maF)$; thus  $Q, R_+ \in W^*_\nu (V,\maF;E^+)$ and $R_- \in W^*_\nu (V,\maF;E^-)$ with $R_\pm$ $\tau^\nu_\maF$ trace class,
see  Remark \ref{rk:smoothing}.

\begin{proposition}\label{prop:calderon}
For any  $N\in \NN, N\geq 1$, the following formulas hold:
\begin{equation}\label{eq:calderon}
\ind^\nu_{{\rm up}}(\tD)=\tau^\nu (\tilde{R}_+)^N -\tau^\nu (\tilde{R}_-)^N\,,\quad\quad 
\ind^\nu_{{\rm down}}(D)=\tau^\nu_{\maF} (R_+)^N-\tau^\nu_{\maF} (R_-)^N
\end{equation}
\end{proposition}

\begin{proof}
Let $N=1$; then the proof given by Atiyah in \cite{Atiyah-covering} extends easily to the present
context. Replacing the parametrix $\tQ$ by  $\tQ_N: =\tQ(1+\tR_-+ \tR_-^2+\cdots+\tR_-^{N-1}$,
which is again a parametrix,
reduces the general case to the one treated by Atiyah.

\end{proof}

Using these formulas we shall now sketch the proof
of the precise analogue of Atiyah's index theorem on coverings.
 
\begin{proposition}\label{prop:index-equal}
The  monodromy  index and the leafwise index coincide:

\begin{equation}\label{index-equal}
\ind^\nu_{{\rm up}}(\tD)=
\ind^\nu_{{\rm down}}(D)
\end{equation}
\end{proposition}

\begin{proof}
Given
$\epsilon>0$ we can choose
a parametrix $\tilde{Q}\in \Psi^{-1}_c(G;\widehat{E}^- , \widehat{E}^+) $ with the property that the two remainders $\tilde{R}_\pm = (\tilde{R}_\pm)_\theta$, $\theta\in T$ are such that 
each $(\tilde{R}_\pm)_\theta$ is supported within an $\epsilon$-neighbourhood
of the diagonal in $\tM_\theta\times \tM_\theta$. 
Let $\maR_{\pm,m}: \maE_m^\pm \to \maE_m^\mp$ be the induced operators on the $\maA_m$-Hilbert
modules $\maE^\pm_m$; since $\tilde{R}_\pm$ are smoothing and of $\Gamma$-compact support
we certainly know that $\maR_{\pm,m}$ are $\maA_m$-compact operators. Let $K^\pm:= \chi_m^{-1}
(\maR_{\pm,m}) \in \maB^{E^\pm}_m$; $K^\pm$ is simply given by the Schwartz kernel of 
$\tilde{R}_\pm$ and is in fact an element in $\maB^{E^\pm}_c$.
In particular  
$K^\pm$ 
has finite $\tau^\nu_{{\rm reg}}$ trace and
$\tau^\nu_{{\rm av}}$ trace. 
  By arguments very similar
(in fact easier) to those establishing Theorem \ref{theo:functional-calculus}
we know that
\begin{equation}\label{eq:von-c-star}
\tau^\nu (\tilde{R}_\pm)=\tau^\nu_{{\rm reg}} (\chi_m^{-1}\maR_{\pm,m})\equiv 
\tau^\nu_{{\rm reg}} (K^\pm)\,,\quad
\tau^\nu_{\maF} (R_\pm)=\tau^\nu_{{\rm av}} (\chi_m^{-1}\maR_{\pm,m})\equiv 
\tau^\nu_{{\rm av}} (K^\pm)\,.
\end{equation}
Thus, from \eqref{eq:calderon}, it suffices to show that 
$$\tau^\nu_{{\rm reg}} (K^\pm)=\tau^\nu_{{\rm av}} (K^\pm)\,.$$
We can write
\begin{eqnarray*}
\tau^\nu_{{\rm av}} (K^\pm)&=&\int_{F\times T} \sum_{\gamma\in\Gamma(\theta)}
K^\pm [\tm,\tm\gamma,\theta]\,d\tm\,d\nu (\theta)\\
&=& \int _{F\times T} 
K^\pm [\tm,\tm,\theta]\,d\tm\,d\nu (\theta) + \int_{F\times T} \sum_{\gamma\in\Gamma(\theta) ;  \gamma \not= e}
K^\pm [\tm,\tm\gamma,\theta]\,d\tm\,d\nu (\theta)\\
&=&\tau^\nu_{{\rm reg}} (K^\pm) +   \int_{F\times T} \sum_{\gamma\in\Gamma(\theta) ; \gamma\not= e}
K^\pm [\tm,\tm\gamma,\theta]\,d\tm\,d\nu (\theta)
\end{eqnarray*}
Choosing $\epsilon$ small enough we can ensure that $K^\pm [\tm,\tm\gamma,\theta]=0$
$\forall \gamma\in \Gamma(\theta)$, $\gamma \not= e$. The proof is complete.

\end{proof}

\begin{remark}
The possibility of localizing a parametrix in an arbitrary small neighbourhood
of the diagonal plays a crucial role in the proof of the above Proposition. There are 
more general situations, for example foliated flat bundles $\tM\times_\Gamma T$
with $\tM$ a manifold with boundary, where it is not possible to localize the parametrix.
In these cases the analogue of Atiyah's index theorem does {\bf not} hold.

\end{remark}

\subsection{The index class in the maximal $C^*$-algebra}

Let $\tD^+$ be as in the previous subsection. As before we consider a parametrix $\tilde{Q}:= (\tilde{Q}_\theta)_{\theta\in T}
\in \Psi^{-1}_c(G;\what{E}^-,\what{E}^+)$
with remainders $\tilde{R}_+$ and $\tilde{R}_-$.
The family  $\tilde{Q}$ defines a bounded  $\maA_m$-linear operator
$\maQ_m$ from $\E_m^-$ to $\E_m^+$.
The families $\tilde{R}_+$ and $\tilde{R}_-$ define  $\maA_m$-linear compact  operators 
$\maR_{\pm,m}$ on the Hilbert modules $\maE^\pm_m$ respectively.

We now define  idempotents $p, p_0$ in $M_{2\times 2} ( \maK_{\maA_m} (\maE_m)\oplus \CC)$ by setting 
\begin{equation}\label{eq:IND}
p:= \begin{pmatrix} \maR_{+,m}^2 & \maR_{+,m}  {(I+\maR_{+,m})\maQ_m }  \cr {\maR_{-,m} \maD^+_m} &
I-\maR_{-,m}^2 \cr
\end{pmatrix}\,,\quad p_0:= \begin{pmatrix} 0 & 0
\cr 0  & I \cr
\end{pmatrix} 
\end{equation}
We thus get a $K_0$-class $[p-p_0]\in 
K_0 (\maK_{\maA_m} (\maE_m)).$

\begin{definition}
The (maximal) index class ${\rm IND} (\maD_m)\in  K_0 (\maB_m)$ associated to the family $\tD$ is, by definition,
the image under the composite isomorphism 
$$K_0 (\maK_{\maA_m} (\maE_m))\rightarrow K_0 (\maB_m^E)\rightarrow  K_0 (\maB_m)$$
of the class $[p-p_0]$.
\end{definition}

One also considers the index class in $K_0 (\maA_m)$:
\begin{equation}\label{eq:Ind}
{\rm Ind} (\maD_m) :=\maM_{{\rm max}}^{-1}({\rm IND} (\maD_m))\in K_0 (\maA_m)
\end{equation}
 with $\maM_{{\rm max}}: K_0 (\maA_m)\to K_0 (\maB_m)$ the Morita isomorphism considered
 in Proposition \ref{prop:morita-compatible}.
 
 \medskip
 Recall now the morphisms $\tau^\nu_{{\rm av},*}: K_0 (\maB_m)\to \CC$ and 
 $\tau^\nu_{{\rm reg},*}: K_0 (\maB_r)\to \CC$. Using the natural morphism $K_0 (\maB_m)\to K_0 (\maB_r)$
 we view both morphisms with domain $K_0 (\maB_m)$: 
 \begin{equation}\label{eq:k-tr-on-max}
\tau^\nu_{{\rm av},*}: K_0 (\maB_m)\to \CC\,,\quad \quad \tau^\nu_{{\rm reg},*}: K_0 (\maB_m)\to \CC
 \end{equation}
Recall also that using the natural morphism $K_0 (\maA_m)\to K_0 (\maA_r)$
we have induced morphisms
\begin{equation}\label{eq:k-tr-on-max}
\tau^\nu_{{\rm av},*}: K_0 (\maA_m)\to \CC\,,\quad \quad \tau^\nu_{{\rm reg},*}: K_0 (\maA_m)\to \CC
 \end{equation}
 
 \begin{proposition}\label{Numeric}
 Let ${\rm IND} (\maD_m) \in K_0 (\maB_m)$ and ${\rm Ind} (\maD_m) \in K_0 (\maA_m)$
be  the two index classes introduced above. 
Then 
 the following formulas hold:
 \begin{eqnarray}\label{index-from-K}
 \ind^\nu_{{\rm up}}(\tD) &=&\tau^\nu_{{\rm reg},*}({\rm IND} (\maD_m))
 = \tau^\nu_{{\rm reg},*}({\rm Ind} (\maD_m))\\
 \ind^\nu_{{\rm down}}(D) &=& \tau^\nu_{{\rm av},*}({\rm IND} (\maD_m))
 = \tau^\nu_{{\rm av},*}({\rm Ind} (\maD_m))
 \end{eqnarray}
 Consequently, from \eqref{index-equal}, we have the following fundamental equality:
  \begin{equation}\label{c-index-equality}
 \tau^\nu_{{\rm reg},*}({\rm Ind} (\maD_m))
= \tau^\nu_{{\rm av},*}({\rm Ind} (\maD_m))
 \end{equation}
 \end{proposition}

\begin{proof}
We only need to prove the first equality in each equation, for the second
one is a consequence of the definition of ${\rm Ind} (\maD_m) \in K_0 (\maA_m)$
and the compatibilty result explained in Proposition \ref{prop:morita-compatible}.
For the first equality we apply \eqref{eq:calderon} with $N=2$ 
and the parametrix $\tQ$.  Using now \eqref{eq:von-c-star}, \eqref{eq:IND}
we get
$$ \ind^\nu_{{\rm up}}(\tD) = \tau^\nu ((\tR_+)^2)- \tau^\nu ((\tR_-)^2)=
\tau^\nu_{{\rm reg}}( (\maR_{+,m})^2)-\tau^\nu_{{\rm reg}}( (\maR_{-,m})^2)=
\tau^\nu_{{\rm reg},*}({\rm IND} (\maD_m))$$
The proof of the other one is similar.

\end{proof}

\begin{remark}
The equalities in Proposition \ref{Numeric} can be rephrased as the equality between the numeric $C^*$-algebraic index and the von Neumann index.
Notice once again that there are more general situations
where this Proposition does not hold, in the sense that 
there exists a well defined 
von Neumann index but there does not exist a well-defined $C^*$-algebraic index
we can equate it to.
The simplest example is given by a fibration of compact manifolds 
$L\to V\to T$ with $V$ and $L$ manifolds
with boundary. The von Neumann index defined by the family of Atiyah-Patodi-Singer
boundary conditions is certainly  well defined (this is the integral over $T$ of the
function that assigns to $\theta\in T$ the APS index of $D^+_\theta$).
On the other hand, unless the boundary family associated to  $(D^+_\theta)_{\theta\in T}$
is invertible, there is not a well defined Atiyah-Patodi-Singer index class in $K_0 (C(T))=K^0 (T)$.
\end{remark}

\subsection{The signature operator for odd foliations}

We briefly review the definition of the leafwise signature operator in the odd case. Recall that when $\dim (M) = 2m-1$, the leafwise signature operator is defined as  the operator $D^{\rm{sign}}$ acting on leafwise differential forms on $V$, defined on even forms of degree $2k$ by
$$
D_{ev}^{\rm{sign}} = i^m (-1)^{k+1}(*\circ d - d\circ *),
$$
and on odd forms of degree $2k-1$ by 
$$
D_{od}^{\rm{sign}} = i^m (-1)^{m-k} (d \circ * + * \circ d),
$$
where $d$ is the leafwise de Rham differential and $*$ is the usual Hodge operator along the leaves associated with the Riemannian metric on the foliation \cite{MS}. An easy computation shows that the two operators $ D_{odd}^{\rm{sign}}$ and $D_{ev}^{\rm{sign}}$ are conjugate so that their invariants coincide and it is sufficient to work with one of them. In contrast with \cite{APS2},  $D^{\rm{sign}}$ will be in the sequel the operator $D_{od}^{\rm{sign}}$. Using the lifted structure to the fibers of the monodromy covers $\tM\times \{\theta\}$ of the leaves, we consider in the same way the $\Gamma$-equivariant family of signature operators $\tD^{\rm{sign}}= (\tD^{\rm{sign}}_\theta)_{\theta\in T}$ which actually coincides with the lift of $D^{\rm{sign}}$ as can be easily checked. The following is well known, see  \cite{APS1}, \cite{APS2} for the first part and \cite{HiSka} for the second:

Recall that the $K_1$ index of $D^{\rm{sign}}$ is the class of the Cayley transform of $D^{\rm{sign}}$, see for instance \cite{HiSka}.

\begin{proposition}\
The operator $D^{\rm{sign}}$ is a leafwise elliptic essentially self-adjoint operator whose $K_1$ index class is a leafwise homotopy invariant of the foliation.
\end{proposition}

The square of $D^{\rm{sign}}$ is proportional to the Laplace operator along the leaves and hence it is leafwise elliptic. The proof that $D^{\rm{sign}}_{ev}$ is formally self-adjoint is straightforward, see \cite{APS2}, and classical elliptic theory on foliations of compact spaces allows to deduce that it is essentially self-adjoint. Now $D^{\rm{sign}}$ is unitarily equivalent to $D^{\rm{sign}}_{ev}$ and hence is also formally self-adjoint. We shall get back to the  index class later on. The homotopy invariance means that if $f: (V,\maF) \to (V', \maF')$ is a leafwise oriented leafwise homotopy equivalence between odd dimensional foliations, then with obvious notations we have
$$
f_* \Ind (D^{\rm{sign}}) = \Ind ({D^{\rm{sign}}} ')
$$
where $f_*$ is the isomorphism induced by the Morita equivalence implemented by $f$ \cite{HiSka}.


%

%


%

\section{Foliated rho invariants}\label{sec:foliated-rho}

Recall that $T$ is a compact Hausdorff space on which the discrete countable group $\Gamma$ acts by homeomorphisms, $M$ is a compact closed manifold with fundamental group $\Gamma$
and universal cover $\tM$ and that $V=\tM\times_\Gamma T$ is the induced  foliated space. 
We are also given a Borel measure $\nu$ on $T$ which is $\Gamma$-invariant. 
We assume in the present section that $M$ is odd dimensional and whence that the leaves of the induced foliation $\maF$ of $V$ are odd dimensional. 
We fix as in the previous section a Dirac-type operator along the leaves of the foliation $(V,\maF)$ acting on the vector bundle $E$. We denote by $D$ this operator acting leafwise, so $D=(D_L)_{L\in V/F}$ where each $D_L$ is an elliptic Dirac-type operator on the leaf $L$ acting on the restriction of $E$ to $L$. We also consider the lifted operator $\tD$ to the monodromy groupoid $G$ of the foliation $(V,F)$ as defined in Section \ref{subsect:dirac}. More precisely, $\tD = (\tD_\theta)_{\theta\in T}$ is a $\Gamma$-equivariant continuous family of Dirac type operators on $\tM$.

\subsection{Foliated eta and rho invariants}\label{sub:fol-eta-rho}

The construction of foliated eta invariants was first given independently in the two references \cite{Ra} \cite{Peric}  and the two definitions work in fact  for general measured foliations. 
Notice that \cite{Ra}
works with the measurable groupoid defined by foliation, whereas \cite{Peric}
works with the holonomy groupoid. As we shall clarify in a moment, the choice of the groupoid
does make a difference for these non-local invariants.
We give in this paragraph a self-contained 
treatment of these two definitions following  our set-up, but using 
 the monodromy groupoid instead of  the holonomy groupoid  considered in  \cite{Peric}. 

We  denote by $k_t$ and $\tk_t$ the longitudinally smooth uniformly bounded Schwartz kernels of the operators $\varphi_t (D)$ and $\varphi_t(\tD)$ obtained using the  function $\varphi_t(x) := x e^{-t^2 x^2}$ for $t>0$. See Lemma \ref{Rapid}.

\begin{lemma}\ (Bismut-Freed estimate)
There exists a constant $C \geq 0$ such that for any $(\tm, \theta)\in \tM\times T$, we have:
$$
|\tr (k_t([\tm, \theta],[\tm,\theta])) | \leq C  \quad  \text{ and } \quad  |\tr (\tk_t([\tm,\tm, \theta]))|  \leq C, \text{ for  }t\leq 1.
$$
\end{lemma}

\begin{proof}
A proof of these estimates appear already in  \cite{Ra}. We give nevertheless
a sketch of the argument.\\
The Bismut-Freed estimate on a closed odd dimensional 
compact manifold $M$ is a pointwise estimate on the vector-bundle trace of the Schwartz
kernel of $D\exp(-t^2 D^2)$ restricted to the diagonal. See the original article
\cite{BF2} but also \cite{Melrose}. As explained for example in the latter reference
the Bismut-Freed estimate is ultimately a consequence of Getzler rescaling 
for the heat kernel of a Dirac laplacian on the even dimensional manifold obtained by 
crossing $M$ with $S^1$. Since these arguments are purely local, they easily extend
to our foliated case, using the compactness of $V:=\tM\times_\Gamma T$ in order to control
uniformly the constants appearing in the poinwise estimate.
\end{proof}
The operators $D^2$ and $\tD^2$ (as well as the operators $|D|$ and $|\tD|$) are non negative operators which are affiliated 
respectively with the von Neumann algebra $W^*_\nu (V,\maF;E)$ and the von Neumann algebra $W^*_{\nu}(G;E)$. (This means that their sign operators as well as their  spectral
projections belong to the von Neumann algebra.)
 Moreover, according to the usual pseudodifferential estimates along the leaves (see for instance \cite{Vassout-these}, \cite{BenameurFack}), the resolvents of these operators belong respectively to the $C^*$-algebras $\maK(W^*_\nu (V,\maF;E), \tau_\maF^{\nu})$ of $\tau_\maF^\nu$-compact elements in $W^*_\nu (V,\maF;E)$ and $\maK(W^*_{\nu}(G;E), \tau^\nu)$ of $\tau^\nu$-compact elements of $W^*_{\nu}(G;E)$. We recall that these compact operators are roughly defined using for instance the vanishing at infinity of the singular numbers, and we refer, for example,  to \cite{BenameurFack} for the precise definition of these ideals. 
Set 
$$
D^2=\int_0^{+\infty} \lambda dE_\lambda \quad \text{ and } \quad \tD^2 = \int_0^{+\infty} \lambda d\tE_\lambda,
$$
for the spectral decompositions in their respective von Neumann algebras. So $E_\lambda$ and $\tE_\lambda$ are the spectral projections  corresponding to $(-\infty,\lambda)$. Since the traces are normal on both von Neumann algebras, 
$$
N (\lambda) = \tau^\nu_{\rm{av}} (E_\lambda) \quad \text{ and } \quad  \tN (\lambda) = \tau^\nu_{\rm{reg}} (\tE _\lambda),
$$
are well defined finite (Proposition \ref{Measurability} in the next subsection)
non-decreasing and non-negative functions, which are right continuous.
 Hence there are Borel-Stieljes  measures $\vartheta$ and ${\tilde \vartheta}$ on $\RR$, such that:
$$
\tau^\nu_{\maF} (f(D)) = \int _\RR f(x) d\vartheta(x) \text{ and } \tau^\nu (f(\tD)) = \int _\RR f(x) d{\tilde \vartheta}(x),
$$
for any Borel function $f:\RR \to [0,+\infty]$. Since $N$ and $\tN$ are finite, the measures $\vartheta$ and $\tilde \vartheta$ are easily proved to be $\sigma$-finite.

\begin{proposition}
The functions $t\mapsto \tau^\nu_{\maF} (D e^{-t^2 D^2})$ and $t\mapsto \tau^\nu (\tD e^{-t^2 \tD^2})$ are Lebesgue integrable on $(0,+\infty)$.
\end{proposition}

\begin{proof} 
We have
$$
|\tau^\nu_{\maF} (D e^{-t^2 D^2})| \leq \tau^\nu_{\maF} (|D| e^{-t^2 D^2}) \quad \text{ and } \quad |\tau^\nu (\tD e^{-t^2 \tD^2})| \leq \tau^\nu (|\tD| e^{-t^2 \tD^2}).
$$
Therefore and since the function $x\mapsto |x| e^{-t^2 x^2}$ is rapidly decreasing, we know from Propositions \ref{Rapid}
and  \ref{Rapid-down} that for any $t>0$
$$
\tau^\nu_{\maF} (|D| e^{-t^2 D^2}) < +\infty \text{ and } \tau^\nu (|\tD| e^{-t^2 \tD^2}) < +\infty.
$$
We also have the formulae
$$
\tau_\maF^\nu (|D| e^{-t^2 D^2}) = \int_{\RR_+} \sqrt{x} e^{-t^2 x} d\vartheta (x)\quad \text{ and }
\quad  \tau^\nu (|\tD| e^{-t^2 \tD^2}) = \int_{\RR_+} \sqrt{x} e^{-t^2 x} d{\tilde \vartheta} (x).
$$
Therefore, by Tonelli's theorem
\begin{eqnarray*}
\int_1^{+\infty} \tau^\nu_{\maF} (|D| e^{-t^2 D^2}) dt & = & \int_0^\infty \sqrt{x} \int_1^\infty  e^{-t^2 x}\, dt\, d\vartheta (x)\\
& = &  \int_0^\infty \sqrt{x} e^{-x} \int_1^\infty e^{-(t^2-1) x} \,dt\, d\vartheta (x)\\
& = & \frac{1}{2}\int_0^\infty \sqrt{x} e^{-x} \int_0^\infty x^{-1/2} (u+x)^{-1/2} e^{-u} \,du \,d\vartheta (x)\\
& \leq  & \frac{1}{2} \left(\int_0^\infty  e^{-x} d\vartheta (x)\right) \left(\int_0^\infty u^{-1/2} e^{-u} du\right)\\
& = & \frac{\sqrt{\pi}}{2} \tau^\nu_{\maF} (e^{-D^2}).
\end{eqnarray*}
The same proof show that 
$$
\int_1^{+\infty} \tau^\nu (|\tD| e^{-t^2 \tD^2}) dt < +\infty.
$$
On the other hand, we have
\begin{eqnarray*}
\int_0^{1} |\tau^\nu_{\maF} (D e^{-t^2 D^2})| dt & \leq & \int_0^{1} \int_{F\times T} |\tr (k_t([\tm, \theta],[\tm, \theta])| \,d\tm \,d\nu (\theta)\, dt\\
& \leq & \int_0^{1} \int_{F\times T} C \,d\tm \,d\nu (\theta)\, dt \\
& = &  C \times {\rm vol} (V, d\tm\otimes \nu) <+\infty.
\end{eqnarray*}
Again, the same proof works as well for the regular trace and the regular von Neumann algebra.
\end{proof}

We are now in position to define the foliated eta invariants.

\begin{definition}\label{def:eta}
We define the up and down eta invariants of our longitudinal Dirac type operator by the formulae
$$
\eta_{\rm{up}}^{\nu} (\tD) := \frac{2}{\sqrt \pi} \int_0^{+\infty} \tau^\nu (\tD e^{-t^2 \tD^2}) dt \quad \text{ and } \quad  \eta_{\rm{down}}^\nu (D) := \frac{2}{\sqrt \pi} \int_0^{+\infty} \tau^\nu_{\maF} (D e^{-t^2 D^2}) dt.
$$  
\end{definition}

Since the traces on both von Neumann algebras are positive, the two eta invariants are real numbers.

\begin{definition}\label{def:rho}
The foliated rho invariant associated to the longitudinal Dirac type operator $D$
on the foliated flat bundle $(V,\maF)$ is defined as 
$$
\rho^\nu (D;V,\maF):=\eta_{\rm{up}}^{\nu} (\tD)\,-\,\eta_{\rm{down}}^\nu (D) $$  
with $\tD$ the lift of $D$ to the monodromy cover.
\end{definition}

We are mainly interested in the present paper in the leafwise signature operator $D^{\rm{sign}}$ and its leafwise lift to the monodromy groupoid $\tD^{\rm{sign}}$. In this case, we can state the following convenient result.

\begin{lemma}\ Denote by $\Delta_j$ the Laplace operator on leafwise $j$-forms. Then the foliated eta invariant of the operator $D^{\rm{sign}}$ on $(V, \maF)$ is given by
$$
\eta^\nu (D^{\rm{sign}}; V, \maF) = \frac{1}{\sqrt\pi} \int_0^{+\infty} \tau^\nu_\maF (* d e^{-t^2 \Delta_{m-1}}) dt = \frac{1}{\sqrt\pi} \int_0^{+\infty} \tau^\nu_\maF (d * e^{-t^2 \Delta_{m}}) dt.
$$
Similar statements hold for the lifted family $\tD^{\rm{sign}}$.
\end{lemma}

\begin{proof}
This is an immediate consequence of a straightforward leafwise version of the computation made in \cite{APS1}[p. 67-68].
\end{proof}

\subsection{Eta invariants and  determinants of paths}\label{subsect:winding}
We review the notion of determinants of paths, adapting the work of
de La Harpe-Skandalis \cite{dlH-Ska} to our context. Recall that $M$ is odd dimensional.
For any von Neumann algebra $\maM$ endowed with a positive semifinite faithful normal trace $\tau$,
we denote  by $L^1(\maM, \tau)$ the Schatten space of  summable $\tau$-measurable operators in the sense of \cite{FackKosaki}. Recall that $L^1(\maM, \tau)\cap \maM$ is a two sided $*$-ideal in $\maM$.  By Propositions \ref{Rapid}, \ref{Rapid-down} 
we have for any rapidly decreasing Borel function $\psi$

 $$
\psi(\tD) := (\psi(\tD_\theta))_{\theta\in T} \in L^1(W^*_{\nu}(G;E), \tau^\nu) \cap W^*_{\nu}(G;E) $$
$$\psi(D) := (\psi(D_{L_{\theta}}))_{\theta\in T} \in L^1(W^*_\nu (V,\maF;E), \tau^\nu_{\maF}) \cap W^*_\nu (V,\maF;E).
$$
We set $\tD=\tU |\tD|$ and $D=U|D|$ for the polar decompositions in the corresponding von Neumann algebras. Then, this decomposition obviously  coincides with the leafwise decompositions 
$$
\tD_\theta = \tU_\theta |\tD_\theta| \text{ and } D_L = U_L |D_L|.
$$
For any $\theta \in T$ with $L=L_\theta$, we write the spectral decompositions:
$$
|\tD_\theta|=\int_0^{+\infty} \lambda d\tE^\theta_{\lambda} \text{ and }|D_{L}| = \int_0^{+\infty} \lambda dE^L_{\lambda}.
$$
As we have already remarked, the collection of partial isometries $\tU= (\tU_\theta)_{\theta\in T}$ (resp. $U = (U_{L_\theta})_{\theta\in T}$) as well as that of spectral projections $\tE_{\lambda}= (\tE^\theta_{\lambda})_{\theta\in T}$ (resp. $E_\lambda= (E^{L_\theta}_{\lambda})_{\theta\in T}$), all belong to $W^*_{\nu}(G;E)$ (resp. $W^*_\nu (V,\maF;E)$). 

\begin{proposition} \label{Measurability}
We have $\tau^\nu(\tE_{\lambda})<+\infty$ and $\tau^\nu_{\maF}(E_{\lambda})<+\infty$ for any $\lambda \in \RR_+$.
\end{proposition}

\begin{proof}
We know that  for any $\lambda < 0$ the operator $(|\tD| - \lambda)^{-1}$ is $\tau^\nu$-compact in $W^*_{\nu}(G;E)$.
In the same way, the operator $(|D| - \lambda)^{-1}$ is $\tau^\nu_{\maF}$-compact in $W^*_\nu (V,\maF;E)$ \cite{Co-LNM}. Hence
the spectral projections of $(|\tD| - \lambda)^{-1}$ are  $\tau^\nu$-finite and 
the spectral projections of $(|D| - \lambda)^{-1}$ are  $\tau^\nu_{\maF}$-finite.
This completes the proof.
\end{proof}

For any $t>0$, the $t$-th singular number of the operator  $\tD$ with respect to the probability measure $\nu$ is defined by \cite{FackKosaki}
$$
\mu_t(\tD)=\inf \{ \| |\tD| {\tilde p}\|, \,{\tilde p}={\tilde p}^2={\tilde p}^* \in W^*_{\nu}(G;E) \text{ and } \int_T \tr (M_\chi (I- {\tilde p}_\theta) M_\chi)d\nu(\theta) \leq t\}.
$$
In the same way, we define
$$
\mu_t(D)=\inf \{ \| |D| p\|,\, p=p^2=p^* \in W^*_\nu (V,\maF;E) \text{ and } \int_T \tr (M_\chi (I- p_{L_\theta}) M_\chi) d\nu(\theta) \leq t\}.
$$
From Proposition \ref{Measurability}, we deduce that $0\leq\mu_t(\tD)=\mu_t(|\tD|)<+\infty$ and $0\leq\mu_t(D)=\mu_t(|D|)<+\infty$. The spectral measure  of $|\tD|$ with respect to the probability measure $\nu$ is denoted $\tmu$, while the spectral measure  of $|D|$ is denoted $\mu$. So for $\tD$ for instance we have
\begin{multline*}
\mu(B)=\int_T \tr ( M_\chi 1_B(|D_{L_\theta}|) M_\chi) d\nu(\theta), \text{ for any Borel subset $B$ of the spectrum of }|D| \text{ and}\\ 
\tmu(\tB)=\int_T \tr ( M_\chi 1_\tB(|\tD_{\theta}|) M_\chi) d\nu(\theta), \text{ for any Borel subset $\tB$ of the spectrum of } |\tD|.
\end{multline*}

We denote by $\mathcal{I}\maK_{E,\rm{reg}}$ (resp. $\mathcal{I}\maK_{E,\rm{triv}}$) 
the subgroup of invertible operators in $W^*_{\nu}(G;E)$ (resp. in $W^*_\nu (V,\maF;E)$) which differ from the identity by an element of the ideal $\maK(W^*_{\nu}(G;E),\tau^\nu)$ (resp. $\maK(W^*_\nu (V,\maF;E),\tau^\nu_{\maF})$). 
The subgroup of bounded operators which differ from the identity by a $\tau^\nu$-summable 
(resp. $\tau^\nu_{\maF}$-summable) operator will be denoted $\mathcal{I}L^{1}_{E,\rm{reg}}$ (resp. $\mathcal{I}L^1_{E,\rm{triv}}$).  

\medskip

{\bf Whenever possible we shall refer to both von Neumann algebras $W^*_{\nu}(G;E)$ and $W^*_\nu (V,\maF;E)$
as $\maM$. We shall then  use  the notation $\mathcal{I} \maK$ (resp. $\mathcal{I} L^1$) and denote by  $\tau$  the corresponding trace.}

\begin{lemma}
The space $\mathcal{I}L^1$ (resp. $\mathcal{I}\maK$) is a subgroup of the group of invertibles $\GL(\maM)$ of the von Neumann algebra $\maM$. 
\end{lemma}

\begin{proof}
We only need to check the stability for taking inverses. Let then $I + T$ be an invertible element in $\maM$ such that $T\in L^1(\maM, \tau)$ (resp. $\maK(\maM, \tau)$). Then we can write
$$
(I+T)^{-1} - I = (I+T)^{-1} (I - (I+T)) = -(I+T)^{-1} T \in L^1(\maM, \tau) \text{ ( resp. $\maK(\maM,\tau)$)}.
$$
\end{proof}

\begin{proposition}\label{prop:approximation}
Let $\gamma:[0,1] \to \mathcal{I}\maK$ be a continuous path for the uniform norm. For any $\ep >0$, there exists a continuous piecewise affine path $\gamma_\ep: [0,1] \to \mathcal{I} L^1$  such that for any $t\in [0,1]$ we have $\|\gamma (t) - \gamma_\ep (t)\| \leq \ep$. Moreover, if $\gamma(0)$ and $\gamma (1)$ belong to $\mathcal{I}L^1$, then we can insure that $\gamma_\ep (i)=\gamma (i)$ for $i=0,1$.
\end{proposition}

\begin{proof}
Since $\gamma$ is continous for the operator norm, we can find $\delta > 0$ such that
$$
|t- s| \leq \delta \Rightarrow \|\gamma(t) - \gamma (s)\| \leq \ep/3.
$$
We subdivide $[0,1]$ into $0=x_0 < x_1 < \cdots < x_n=1$ so that $|x_{j+1} - x_j| \leq \delta$ for any $j$. On the other hand, the ideal $L^1(\maM, \tau)\cap \maM$  is dense in $\maK(\maM,\tau)$ for the uniform norm. Therefore, for any $j=0,\cdots, n$, we can find $\gamma_\ep (x_j) \in B(\gamma(x_j), \ep/9)$, the ball centered at $\gamma(x_j)$ with radius $\ep /9$, such that $\gamma_\ep (x_j)\in \mathcal{I}L^1$. We then define a path $\gamma_\ep:[0, 1] \to \mathcal{I}L^1$ which is affine on every interval
$[x_j,x_{j+1}]$ and prescribed by the values $\gamma_\ep (x_j)$ for $j=0, \cdots, n$. The path $\gamma_\ep$ is then continuous and differentiable outside the finite set $\{x_j, j=0, \cdots, n\}$. Moreover, for $t\in [x_j, x_{j+1}]$ we have
$$
\|\gamma_\ep (t) - \gamma_\ep(x_j)\| =  t \times \| \gamma_\ep(x_{j+1}) - \gamma_\ep (x_j)\| \leq \| \gamma_\ep(x_{j+1}) - \gamma_\ep (x_j)\| \leq \|\gamma (x_{j+1}) - \gamma(x_j)\| + 2 \ep/9 \leq 5\ep /9. 
$$
Therefore,
$$
\|\gamma(t) - \gamma_\ep (t)\| \leq \|\gamma (t) - \gamma (x_j)\| + \|\gamma(x_j) - \gamma_\ep (x_j)\| + \|\gamma_\ep (x_j) - \gamma_\ep (t)\| \leq \ep /3 + \ep /9 + 5\ep /9 = \ep. 
$$
\end{proof}

\begin{definition}
Given a continuous piecewise $C^1$ path $\gamma: [0,1] \to \mathcal{I}L^1$ for the $L^1$-norm
in $\maM$, we define the determinant $w^\tau(\gamma)$ by the formula
$$
w^\tau(\gamma) := \frac{1}{2\pi {\sqrt{-1}}} \int_0^1 \tau (\gamma(t)^{-1} \gamma ' (t)) dt.
$$
{\bf When $\maM$ is $W^*_\nu (G,E)$ this determinant will be denoted by
$w^\nu(\gamma)$ while when $\maM$ is equal to $W^*_\nu (V,\maF;E)$
this determinant will be denoted $w^\nu_{\maF} (\gamma)$.}
\end{definition}

We summarize the properties of the determinant in the following

\begin{proposition}
Let $\gamma: [0,1] \to  \mathcal{I}L^1$ be a continuous piecewise $C^1$ path for the $L^1$-norm.
\begin{enumerate}
\item Assume that 
$$
\| \gamma (t) - I \|_1 < 1, \quad \text{ for any } t\in [0,1].
$$
Then, for any $t\in [0,1]$ the operator $\Log(\gamma(t))$ is well defined in the von Neumann algebra and we have
$$
w^\tau(\gamma) = \frac{1}{2 \pi {\sqrt{-1}}} \left[\tau (\Log (\gamma(1))) - \tau (\Log(\gamma(0))) \right].
$$
\item There exists $\delta_\gamma > 0$ such that for any continuous piecewise 
$C^1$ path $\alpha:[0,1] \to  \mathcal{I}L^1$ for the $L^1$ norm, with 
$$
\|\alpha (t) - \gamma(t) \|_1 \leq \delta_\gamma \text{ and } \alpha (i)=\gamma(i), i=0,1,
$$
we have $w^\tau (\alpha) = w^\tau (\gamma)$.
\item  If $\gamma$ is a continuous piecewise $C^1$ path for the uniform norm, then the determinant $w^\tau(\gamma)$ is well defined. Moreover, $w^\tau (\gamma)$ only depends on the homotopy class of $\gamma$  with fixed endpoints and with respect to the uniform norm.
\end{enumerate}
\end{proposition}

\begin{proof}
This proposition is a straightforward extension of the corresponding results in \cite{dlH-Ska}.
We give a brief outline of the proof here for the benefit of the reader.  It is clear in the first item, since $\tau$ is a positive trace, that  the function $t\mapsto \Log(\gamma(t))$ is well defined (using for instance the series) and is a piecewise smooth path. Moreover, we have
$$
\frac{d}{dt} \tau (\Log(\gamma(t)) = \tau (\gamma^{-1}(t) \frac{d\gamma}{dt} (t).
$$
This completes the proof of the first item. 

Let $\alpha$ be a continuous piecewise $C^1$ path satisfying the assumptions of the second item. We consider the continuous piecewise $C^1$ loop $\beta:[0, 1 ] \to \mathcal{I}L^1$ given by $\beta (t) = \gamma(t)^{-1} \alpha(t)$ which satisfies $\beta(0)=\beta (1)= I$. We have
$$
\|\beta(t) - I\|_1 \leq \|\gamma(t)^{-1}\| \times \|\gamma(t) - \beta(t)\|_1.
$$  
Therefore, with  $\delta_\gamma = \frac{1}{\inf_{t\in [0,1]} \|\gamma (t)^{-1}\|}$, we are done using the first item. 

The rest of the proof is similar and is omitted. 
\end{proof}

\begin{definition}
Let $\gamma:[0,1] \to \mathcal{I}\maK$ be a continous path for the uniform norm
such that $\gamma(0)$ and $\gamma (1)$ are in $\mathcal{I}L^1$.
We define the determinant $w^\tau (\gamma)$ by
$
w^\tau (\gamma) := w^\tau ( \alpha),
$
for any continuous piecewise $C^1$ path $\alpha: [0, 1] \to \mathcal{I}L^1$ such that 
$$
\|\alpha (t) - \gamma(t) \|_1 \leq \delta_\gamma \text{ and } \alpha (i)=\gamma(i), i=0,1.
$$
\end{definition}

\begin{remark}
It is clear from the previous proposition that the above definition is well posed.
\end{remark}

We now set
$$
\varphi (x):= \frac{2}{\sqrt{\pi}} \int_0^x e^{-s^2} ds, \quad \psi_t(x):= -e^{i\pi \varphi(tx)}\quad \text{ and } f_t(x):= x e^{-t^2x^2} \quad \text{ for } x\in \RR, \text{ and any }t \geq 0.
$$
Then the function $1-\psi_t$, the derivative $\psi_t '$ and the function $f_t$ are Schwartz class functions for any $t>0$. Using the results of the previous sections, we deduce that the operators $I-\psi_t(\maD_m)$, $\psi_t \,' (\maD_m)$ and $f_t(\maD_m)$ are $\maA_m$-compact operators on the Hilbert module $\maE_m$. Moreover, their images under the  representations in the von Neumann algebras $W^*_{\nu}(G;E)$ and $W^*_\nu (V,\maF;E)$ are trace class operators. Note also that the operator $\psi_t(\maD_m)$ is invertible  with inverse given by $-e^{-i \pi \varphi(t\maD_m)}$, so $\psi_t(\maD_m)$ is a smooth path of invertibles in $\mathcal{I}\maK_{\maA_m}(\maE_m)$ whose image under $\pi^{reg}\circ \chi_m^{-1}$ in $W^*_{\rm{reg}} (G;E)$ is also a smooth path of invertibles in $\mathcal{I}L^1_{E, \rm{reg}}$. The same result holds for the image under 
$\pi^{av}\circ \chi_m^{-1}$ in $W^*_\nu (V,\maF;E)$. We denote by 
$$
\gamma^{\rm{reg}} (\maD_m)\equiv (\gamma_t^{\rm{reg}} (\maD_m))_{t\geq 0}:= \left((\pi^{reg}\circ \chi_m^{-1})(\psi_t(\maD_m))\right)_{t\geq 0} $$
and
$$
\gamma^{\rm{av}} (\maD_m)
\equiv (\gamma_t^{\rm{av}} (\maD_m))_{t\geq 0}
:= \left((\pi^{av}\circ \chi_m^{-1})(\psi_t(\maD_m))\right)_{t\geq 0}
$$
the resulting smooth  paths in the two von Neumann algebras. Using the traces $\tau^\nu$
and $\tau^\nu_{\maF}$, we define 
$$
w^\nu_{\rm{reg},\epsilon} (\maD_m) := w^\nu (\gamma^{\rm{reg},\epsilon} (\maD_m)) \text{ and } 
w^\nu_{\rm{av},\epsilon} (\maD_m) := w^\nu_{\maF} (\gamma^{\rm{av},\epsilon} (\maD_m)).
$$
with $\gamma^{\rm{reg},\epsilon} (\maD_m)$ the path
$ \left((\pi^{reg}\circ \chi_m^{-1})(\psi_t(\maD_m))\right)^{t\leq 1/\epsilon} _{t\geq \epsilon}$
and similarly for  $\gamma^{\rm{av},\epsilon} (\maD_m)$

\begin{theorem}
The following relations hold:
$$\lim_{\epsilon\to 0} w^\nu_{\rm{reg},\epsilon} (\maD_m)=\frac{1}{2} \eta^\nu_{{\rm up}} (\tD)\;\;\text{ and}\;\;
\lim_{\epsilon\to 0}  w^\nu_{\rm{av},\epsilon}(\maD_m)= \frac{1}{2} \eta^\nu_{{\rm down}} 
(D) $$
and hence
$$2 \rho^\nu (D;V,\mathcal F)=\lim_{\epsilon\to 0} [w^\nu_{\rm{reg},\epsilon} (\maD_m)- w^\nu_{\rm{av},\epsilon}(\maD_m)].
$$
\end{theorem}

\begin{proof}
We have by definition and by straightforward computation
\begin{eqnarray*}
\gamma_t^{\rm{reg}} (\maD_m)^{-1} \frac{d}{dt}\gamma_t^{\rm{reg}} (\maD_m) & = & (\pi^{\rm{reg}} \circ \chi_m^{-1})\left(i\pi \maD_m \frac{2}{\sqrt{\pi}} e^{-t^2\maD_m^2}\right)\\
& = & 2i\sqrt{\pi} (\pi^{\rm{reg}} \circ \chi_m^{-1})\left(f_t(\maD_m)\right)
\end{eqnarray*}
But we know by Proposition \ref{F-Comp}
that 
$$
(\pi^{\rm{reg}} \circ \chi_m^{-1})\left(f_t(\maD_m)\right) = (f_t(\tD_\theta))_{\theta\in T},
$$
where $(\tD_\theta)_{\theta\in T}$ is the $\Gamma$-invariant Dirac type family. Hence we get
$$
\gamma_t^{\rm{reg}} (\maD_m)^{-1} \frac{d}{dt}\gamma_t^{\rm{reg}} (\maD_m) = 2i\sqrt{\pi} (f_t(\tD_\theta))_{\theta\in T},
$$
where this equality holds in the von Neumann algebra $W^*_{\nu}(G;E)$. Applying the trace $\tau^\nu$, integrating over $(0,+\infty)$ and dividing by $2i\pi$, we obtain
$$
\lim_{\epsilon\to 0} w^\nu_{\rm{reg},\epsilon} (\maD_m) 
 = \lim_{\epsilon\to 0} \frac{1}{\sqrt{\pi}} \int_{\ep}^{1/\ep} \tau^\nu ((f_t(\tD_\theta))_{\theta\in T}) dt = 
 \frac{1}{\sqrt{\pi}} \int_{0}^{+\infty} \tau^\nu ((f_t(\tD_\theta))_{\theta\in T}) dt =\frac{1}{2} \eta^\nu_{\rm{up}} (\tD)
$$
The proof of the second equality is similar and one uses the equality 
$$
(\pi^{\rm{av}} \circ \chi_m^{-1})\left(f_t(\maD_m)\right) = (f_t(D_L))_{L\in V/F},
$$
which is proved in Proposition \ref{F-Comp}
.
\end{proof}

\section{Stability properties of $\rho^\nu$ for the signature operator}\label{sec:stability}

\subsection{{Leafwise homotopies}}

Let $\Gamma$, $T$ and $\tM$ be as in the previous sections. Let $V:= \tM\times_\Gamma T$
be the associated foliated flat bundle.
 Assume that 
 $\tM'$ 
 is another $\Gamma$-coverings and  let $T'$ be  a compact space endowed with a continuous action of $\Gamma$ by homeomorphisms. We consider $\tM'\times T$ and the foliated flat bundle
 $V':= \tM' \times_\Gamma T$.

\begin{definition}\label{Leafwise}
Let $(V,\maF)$ and $(V', \maF')$ be two foliated spaces. A leafwise map $f:(V,\maF) \to (V',\maF')$ is a continuous map such that
\begin{itemize}
\item The image under $f$ of any leaf of $(V,\maF)$ is contained in a leaf of $(V',\maF')$. 
\item The restriction of $f$ to any leaf of $(V,\maF)$ is a smooth map between smooth leaves.
\end{itemize}
\end{definition}

{
\begin{remark}
\begin{enumerate}
\item We do not assume, that the leafwise derivatives to all orders of $f$ are also continuous.
\item If  $V$ and $V'$ are smooth manifolds and $f: V\to V'$ is a differentiable map, then $f$ is a leafwise map if and only if $f_*: T(V) \to T(V')$ sends $T\maF$ to $T\maF'$.
\end{enumerate}
\end{remark}
}
Roughly speaking, a leafwise map induces a ''continuous map'' between the quotient spaces of leaves. When the foliations are trivial, a leafwise map $f: M\times T \to M'\times T'$ is given by 
$$
f(m, \theta) = (h(m, \theta), k(\theta)), \quad (m, \theta) \in M\times T,
$$
where $k$ and $h$ are continous and $h$ is smooth with respect to the first variable. In the case of foliated bundles considered in the present paper, i.e.
$$
V=\tM \times_\Gamma T \quad \text{ and }\quad V'= \tM' \times_{\Gamma'} T',
$$
we obtain a  leafwise map $f:V\to V'$ as a continuous lift of a map $\varphi: T/\Gamma\to T'/\Gamma'$ which satisfies the regularity properties of Definition \ref{Leafwise}.

{ An easy example of a leafwise map occurs when $f$ is the quotient of a leafwise map ${\tilde f}:\tM\times T \to \tM' \times T'$ between the two trivial foliations, which is $(\Gamma, \Gamma')$-equivariant with respect to a group homomorphism $\alpha:\Gamma\to \Gamma'$. We shall get back to this example more explicitely later on. 
It is easy to construct a leafwise map between $V$ and $V'$ which is not the quotient of a $(\Gamma, \Gamma')$ equivariant leafwise map ${\tilde f}$. Moreover, if  
 ${\tilde f}$ exists then it is not unique: indeed, for example, if $\delta
 \in Z(\Gamma)\subset \Gamma$ is an element in  the center of $\Gamma$, then the leafwise map ${\tilde f}_\delta:= {\tilde f} \circ \delta^*$ (where $\delta^*: \tM \times T\to  \tM \times T$ is the diffeomorphism
 induced by the action of $\delta$ on the right), is equivariant with respect to the same homomorphism $\alpha: \Gamma\to \Gamma'$ (because $\delta\in Z(\Gamma)$) and also induces $f$.
}

{Given a foliated space $(V, \maF)$ in the sense of \cite{MS}, a subspace $W$ of $(V,\maF)$ will be called  a transversal to the foliation if for any $w\in W$ there exists a distinguished neighborhood $U_w$ of $w$ in $V$ which is homeomorphic to $\RR^p \times (U_w\cap W)$. Then one can show that the intersection of $W$ with any leaf $L$ of $(V, \maF)$ is a discrete subspace of $L$, that is a zero dimensional submanifold of $L$. Such a transversal is complete if it intersects all the leaves. In our example of foliated bundle $V=\tM \times_\Gamma T$, any fiber of $V\to M$ is a complete transversal which is in addition compact, and any open subset of such fiber is a transversal.}

{ 
\begin{definition}
\begin{enumerate}
\item Let $(V,\maF)$ be a foliated space. Two leafwise maps $f, g:(V,\maF) \to (V',\maF')$ are leafwise homotopic if there exists a leafwise map $H:(V\times [0,1], \maF\times [0,1]) \to (V',\maF')$ such that $H(\cdot, 0) = f$ and $H(\cdot, 1) =g$. 
\item Let $(V,\maF)$ and $(V',\maF')$ be two foliated spaces. A leafwise map $f:(V,\maF)\to (V',\maF')$ is a leafwise homotopy equivalence, if there exists a leafwise map $g:(V',\maF')\to (V,\maF)$ such that 
\begin{itemize}
\item $g\circ f$ is leafwise homotopic to the identity of $(V,\maF)$.
\item $f\circ g$ is leafwise homotopic to the identity of $(V',\maF')$.
\end{itemize}
\item We shall say that the foliations $(V,\maF)$ and $(V', \maF')$ are (strongly) leafwise homotopy equivalent if there exists a leafwise homotopy equivalence from $(V,\maF)$ to $(V', \maF')$.
\end{enumerate}
\end{definition}
Note that according to the above definition, the homotopies in (2) are supposed to preserve the leaves.}

{
It is a classical fact that two leafwise homotopy equivalent compact foliated spaces $(V,\maF)$ and $(V',\maF')$ have necessarily the same leaves dimension \cite{BH3}.
Note also that each leafwise homotopy equivalence sends a transversal to a transversal.
{
\begin{lemma}
A  leafwise homotopy equivalence
 induces a local homeomorphism between transversals to the foliations.
\end{lemma}

\begin{proof} See also \cite{BH3}.
Let $f$ be the  leafwise homotopy equivalence with homotopy inverse $g$, and denote by $h:[0,1]\times V \to V$ the $C^{\infty, 0}$ homotopy between $g f$ and the identity. Let $w\in V$. 
 Let $W$ be an open  transversal of $(V, \maF)$ through $w\in W$. Take a distinguished chart $U'$ in $(V', \maF')$ which is an open neighborhood of  $f(w)$ and which is homeomorphic to $D'\times W'$ for some transversal $W'$ at $f(w)$. Then one finds an open distinguished chart $U$ in $(V, \maF)$ such that $f(U)\subset U'$. Reducing $W$ if necessary we can assume that $U$ is homeomorphic to $D\times W$ for some disc $D$. Now, it is clear that since $f$ is leafwise, it  induces a map ${\hat f}:W\to W'$. By the same reasonning, we can assume furthermore that $g(U')$ is contained in a distinguished chart $U_1$ in $(V,\maF)$, homeomorphic to $D_1\times W_1$. 

The homotopy $h$ induces a continous map ${\hat h}: W\to W_1$ and this map (or its reduction to a smaller domain) is simply the holonomy of the path $t\mapsto h(t,w)$. Hence ${\hat h}$ is locally invertible and it s clear that ${\hat h}^{-1} {\hat g}$ is a continuous inverse for ${\hat f}$. 
\end{proof}

 {When $V=\tM\times_\Gamma T$ and $V'=\tM'\times_{\Gamma'} T'$, a particular case of leafwise homotopy equivalence is given by the quotient of an equivariant leafwise homotopy equivalence between $\tM\times T$ and $\tM'\times T'$. Recall that a fiberwise smooth map $\tf:\tM\times T \to \tM' \times T'$ is a continous map which can be written in the form
$$
f(\tm, \theta) = (h(\tm, \theta), k(\theta)), \quad (\tm, \theta) \in \tM\times T,
$$
with $h$ smooth with respect to the first variable. If $\alpha: \Gamma\to \Gamma'$ is a group homomorphism, then the fiberwise
map 
$\tf:\tM\times T \to \tM' \times T'$ is $\alpha$-equivariant if $\tf( (\tm,\theta) \gamma)=(\tf (\tm,\theta)) \alpha(\gamma)$.}

In the following definition we extend the action of $\Gamma$ and $\Gamma'$ on $\tM\times T$ and
$\tM'\times T'$ to $\tM  \times [0,1] \times T $ and $\tM'  \times [0,1]\times T' $ respectively, by
declaring the action trivial on the $[0,1]$ factor.

\begin{definition}
We shall say that $f:(V,\maF)\to (V',\maF')$
is a special homotopy equivalence
 if there exist continuous maps $\tf:\tM\times T \to \tM' \times T'$, $\tg:\tM'\times T' \to \tM \times T$,
$H: \tM  \times [0,1] \times T \rightarrow \tM\times T$, $ H': \tM'  \times [0,1]\times T' \rightarrow \tM'\times T',
$
 and group homomorphisms $\alpha:\Gamma\to \Gamma'$, $\beta: \Gamma' \to \Gamma$
 such that:
\begin{itemize}
\item $\tf$, $\tg$, $H$ and $H'$ are fiberwise smooth;
\item $\tf$ is $\alpha$-equivariant; $\tg$ is $\beta$-equivariant; $H$ is $\Gamma$-equivariant, $H'$ is $\Gamma'$-equivariant;
\item the restriction of $H$ to $\tM\times \{0\} \times T$ (resp. of $H'$ to $\tM'\times \{0\} \times T'$) is the identity map and the restriction of $H$ to $\tM\times \{1\} \times T$ (resp. of $H'$ to $\tM'\times \{1\} \times T'$) is $\tg\circ \tf$ (resp. $\tf\circ \tg$);
\item $f:(V,\maF)\to (V',\maF')$ is induced by  $\tf:\tM\times T \to \tM' \times T'$.
\end{itemize}
\end{definition}

If there exists such a special homotopy equivalence, we say that $(V,\maF)$ and $(V',\maF')$ are special homotopy equivalent.

\begin{lemma}
{If the pairs $(V,\maF)$ and $(V', \maF')$ are special homotopy equivalent,
then they are leafwise homotopic equivalent.}
\end{lemma}

\begin{proof}
{The equivariance of ${\tilde H}$ and ${\tilde H}'$ with respect to $\alpha$ and $\beta$, and the trivial action on the $[0,1]$ factor, allows to induce leafwise maps $H:V\times [0,1]\to V$ and $H': V'\times [0,1] \to V'$ by setting,}
$$
H([\tm, \theta];t) := [H(\tm, t, \theta)] \text{ and } H'([\tm', \theta'];t) := [H'(\tm', t, \theta')].
$$
{In the same way the maps ${\tilde f}$ and ${\tilde g}$ induce leafwise maps $f$ and $g$ which are leafwise homotopy equivalences through the homotopies $H$ and $H'$.}
\end{proof}

\begin{lemma}
If   $f:(V,\maF)\to (V',\maF')$ is a special homotopy equivalence induced by $\tf(\tm,\theta)= (h(\tm,\theta),k(\theta))$
as in the previous definition, then $\alpha:\Gamma\to \Gamma'$ is an isomorphism and
$k:T\to T'$ is an equivariant homeomorphism.
\end{lemma}

\begin{proof}
{Let ${\tilde f}$ and ${\tilde g}$ be equivariant leafwise smooth maps which give a
special  homotopy equivalence as in the above definition. We denote by $k$ and $k'$ the continuous equivariant maps induced by ${\tilde f}$ and ${\tilde g}$ on $T$ and $T'$ respectively. So, }
$$
k: T \rightarrow T' \text{ and } k': T' \rightarrow T.
$$
{Since our homotopies send leaves to leaves, the composite maps $k'\circ k$ and $k\circ k'$ are identity maps. Moreover, if $\alpha$ and $\beta$ are the group homomorphisms corresponding to the equivariance property of ${\tilde f}$ and ${\tilde g}$ respectively, then the homotopy ${\tilde H}$ satisfies}
$$
{\tilde H} ((\tm, t, \theta)\gamma) = {\tilde H} (\tm, t, \theta) (\beta\circ \alpha)(\gamma), \quad \forall t\in [0,1].
$$
{Therefore, applying this relation to $t=0$, we get $\beta\circ \alpha = id_{\Gamma}$. The same argument gives the relation $\alpha \circ \beta= id_{\Gamma '}$.}
\end{proof}

\begin{remark}
As already remaked, easy examples show that the foliations $(V,\maF)$ and $(V', \maF')$ can be leafwise homotopy equivalent with non isomorphic groups $\Gamma$ and $\Gamma'$ and non homeomorphic spaces $T$ and $T'$. 
\end{remark}



%
%

\subsection{{$\rho^\nu (V,\F)$ is  metric independent}} 

We fix a continuous leafwise smooth Riemannian metric $g$ on
 $(V,F)$.  $g$ is lifted to a $\Gamma$-equivariant leafwise metric ${\tilde g}$ on $\tM\times T$, see \cite{MS}. So ${\tilde g}=({\tilde g}(\theta))_{\theta\in T}$, where
 ${\tilde g}(\theta)$ is a metric on $\tM\times \{\theta\}$
 and we assume that this structure is transversely continuous and equivariant with respect to the action of $\Gamma$. 
 In what follows we shall refer to  the bundle of exterior powers of the cotangent bundle as the Grassmann bundle.
 Consider the $\Gamma$-equivariant  vector bundle $\widehat E$ over $\tM\times T$, obtained by
 pulling back from $V$ the longitudinal Grassmann bundle $E$ of the foliation $(V,\maF)$.
Assume for the sake of simplicity of signs that the dimension of $M$ is $4\ell - 1$ that is in the notations of Section \ref{Section.Index}, $m=2\ell$.
Consider the associated $\Gamma$-equivariant family of signature operators
 $(\tilde{D}^{{\rm sign}}_\theta)_{\theta\in T}$ associated with ${\tilde g}$, as defined in Section \ref{Section.Index}. We denote by $D^{{\rm sign}}$ the longitudinal signature operator
 on $(V,\F)$ associated with the leafwise metric $g$
 acting on  leafwise $2\ell -1$ forms.


Recall that $\nu$ is a $\Gamma$-invariant Borel measure on $T$. We have defined 
in Subsection \ref{sub:fol-eta-rho}
a foliated rho-invariant $\rho^\nu (D^{{\rm sign}}; V,\F)$.
We want to investigate the behavior of $\rho^\nu (D^{{\rm sign}}; V,\F)$ under a 
change of metric and under a leafwise diffeomorphism. First, we deal with the invariance of $\rho^\nu$ with respect to a change of metric. Up to constant, we can replace $\rho^\nu (D^{{\rm sign}}; V,\F)$, as it is usual, see \cite{APS2} \cite{Cheeger-Gromov-JDG}, by the $\rho$ invariant of the foliation $(V,\maF)$ defined as:
$$
\rho^\nu(V,\maF;g) := \frac{1}{{\sqrt{\pi}}} \int_0^\infty \left[\tau^\nu (
{\tilde *} {\tilde d}
e^{-t{\tilde \Delta}} - \tau^\nu_{\maF} (   * d  e^{-t\Delta}\right] \frac{dt}{{\sqrt{t}}},
$$
{where ${\tilde \Delta}$ and $\Delta$ are the Laplace operators on leafwise $2\ell -1$ forms, associated with the  metrics $\tilde g$ and $g$.

%
%

\begin{proposition}
Let  $\Gamma$, $\tM$, $T$, $\nu$ and $(V,\maF)$ be as above. 
Let $(g_u)_{u\in [0,1]}$ be a continuous leafwise smooth one-parameter family of continuous leafwise smooth metrics on $(V, \maF)$. Then
\begin{equation}\label{Cheeger-Gromov}
\rho^\nu (V, \F; g_0)=\rho^\nu (V,\F; g_1)
\end{equation}
\end{proposition}

\begin{proof}
The proof of this proposition in the case where $T$ is reduced to a point was first given by J. Cheeger and M. Gromov in \cite{Cheeger-Gromov-JDG}. The Cheeger-Gromov proof extends  to the general case of measured foliations and in particular to the case of foliated bundles and we proceed to explain
the easy modifications needed for foliated bundles. 
Let $D_u$, for $0\leq u\leq 1$, be  the leafwise operator on $2\ell - 1$ leafwise differential forms of $(V, \maF)$, given by $D_u = *_u \circ d$, where $d$ is the leafwise de Rham operator and $*_u$ is the leafwise Hodge operator associated with the metric $g_u$. It is easy to see
 that $u\mapsto \tau_\F^\nu (D_u e^{-t \Delta_u})$ is smooth. Since $V$ is compact, the elliptic estimates along the leaves are uniform and we have for instance 
$$
\maR (e^{-r\Delta_0}) \subset {\rm Dom} (\Delta_u) , \quad \forall r > 0 \text{ and } u\in [0,1]. 
$$
Here $\maR$ denotes the range of an operator and ${\rm Dom}$  the domain. Therefore, we can follow the steps of the proof in \cite{Cheeger-Gromov-JDG} and deduce the fundamental relation
$$
\frac{d}{du}|_{u=0} \tau_\maF^\nu (D_u e^{-t \Delta_u}) = \tau_\maF^\nu (\frac{d *}{du}(0) d  e^{-t \Delta_0}) + 2t \frac{d}{dt} \tau_\F^\nu (\frac{d *}{du}(0) d e^{-t \Delta_0})).
$$
Using integration by parts, we then deduce
$$
{\sqrt \pi} \frac{d}{du}|_{u=0} \int_\ep^A \tau_\F^\nu (D_u e^{-t\Delta_u}) \frac{dt}{{\sqrt t}} = 2 {\sqrt A} \tau^\nu_\F ( \frac{d *}{du}(0) d   e^{-A \Delta_0}) -  2 {\sqrt \ep} \tau^\nu_\F ( \frac{d *}{du}(0) d e^{-\ep \Delta_0}).
$$
Using the normality of the trace $\tau^\nu_\maF$ and the spectral decomposition in the type II$_\infty$ von Neumann algebra $W^*_\nu (V, \F; E)$, we deduce that
$$
\lim_{A\to +\infty} 2{\sqrt A} \tau^\nu_\F (\frac{d *}{du}(0) d  e^{-A \Delta_0}) = 0.
$$
Now, the same estimates are as well valid in the type II$_\infty$ von Neumann algebra $W^*_\nu (G;E)$ with the normal trace $\tau^\nu$. Hence, we are reduced to comparing the limits as $\ep\to 0$ of the difference
$$
2 {\sqrt \ep} \tau^\nu_\F ( \frac{d *}{du}(0) d e^{-\ep \Delta_0}) - 2 {\sqrt \ep} \tau^\nu ( \frac{d {\tilde *}}{du}(0) d   e^{-\ep {\tilde \Delta}_0}).
$$
Replacing the heat operators by corresponding parametrices which are localized near the units $V$,
in the two groupoids involved, see for instance \cite{Co-LNM},  the limit of the two terms in the above difference  is proved to be the same
by classical arguments, which finishes the proof.
\end{proof}

\medskip
{\bf{According to the previous Proposition we can now denote by $\rho^\nu (V,\maF)$
the signature  rho invariant associated to any metric as before. All the leafwise maps considered in the rest of the paper are assumed to respect the orientations.}}

\medskip
If we are now given a leafwise smooth homeomorphism  $f:V\longrightarrow V'$,
then we can transport the leafwise metric $g$ from $V$ to $f_*g$ on  $V'$ 
and form the corresponding signature operator $D^{{\rm sign}}\,'$ along the leaves of $(V',\maF')$ and also the $\Gamma$-equivariant signature operator $\tD^{{\rm sign}}\,'= (\tD^{{\rm sign}}\,'_{\theta'})_{\theta'\in T'}$ corresponding to the lifted $\Gamma$-invariant metric.
Finally, the $\Gamma$-invariant measure $\nu$ on $T$, yields a  holonomy invariant 
transverse measure $\Lambda (\nu)$ on the foliation $(V,\maF)$. The leafwise diffeomorphism $f$ sends transversals to transversals and allows to transport the measure $\Lambda(\nu)$ into a holonomy invariant transverse measure $f_*\Lambda (\nu)$ on $(V',\maF')$. Such a measure 
yields by restriction to a fiber  a $\Gamma'$-invariant measure $\nu'$ on $T'$
so that  $f_*\Lambda (\nu) = \Lambda (\nu')$. More precisely, a fiber $V'_{m'_0}$ of the fibration $V'\to M'$ is a transversal to the foliation $\maF'$ and hence the holonomy invariant transverse measure  $f_*\Lambda (\nu)$ restricts to a measure on $V'_{m'_0}$. On the other hand, by fixing $\tm'_0$ with $[\tm'_0]=m'_0$ we get an identification of $V'_{m'_0}$ with the space $T'$. It is an easy exercise to check that the corresponding mesure on $T'$ through this identification is $\Gamma'$-invariant and that the associated holonomy invariant transverse measure on the foliation $(V', \maF')$ is precisely $f_*\Lambda (\nu)$.

\begin{proposition} With the above notations, we have the following equalities {for the eta invariants
associated with the two signature operators $D^{{\rm sign}}$ and $D^{{\rm sign}}\,'$}:
{$$
\eta^\nu_{\rm{down}} (D^{{\rm sign}}) = \eta^{\nu'}_{\rm{down}} (D^{{\rm sign}}\,') \;\text{ and }\; \eta^\nu_{\rm{up}} (\tD^{{\rm sign}}) = \eta^{\nu'}_{\rm{up}} (\tD^{{\rm sign}}\,').
$$}
\end{proposition}

\begin{proof}
  Let us prove, for example, the second equality (the first one will be obtained in a similar way). 
Let $W$ be the regular von Neumann algebra associated to $(V,\maF)$, the vertical 
Grassmann bundle $\widehat{E}$ and $g$. Let $\tau$ be
the trace defined by $g$ and $\nu$ and let $W'$ and $\tau'$ be the corresponding objects, associated
to $(V',\maF')$, $f_*g$ and the transported measure $\nu'$ under the leafwise diffeomorphism $f$.  The leafwise diffeomorphism $f$ lifts to a leafwise diffeomorphism ${\tilde f}$ between the monodromy groupoids $G$ and $G'$. More precisely, for any $x\in V$ $f$ lifts to a diffeomorphism ${\tilde f}_x: G_x \to G'_{f(x)}$ which induces, by the pull-back of forms, a unitary $U_x$ between the spaces of $L^2$-forms. Recall that the metric on $(V', \maF')$ is  $f_*g$. The signature operator on $G'_{f(x)}$ associated with the metric $f_*g$ is easily identified with the push-forward operator under ${\tilde f}$, that is the conjugation of the signature operator on $G_x$ by the unitary $U_x$. Hence the functional calculus of ${\tilde D}^{{\rm sign}}\,'_{f(x)}$ is also the conjugation of the functional calculus of ${\tilde D}^{{\rm sign}}_x$ by $U_x$. So, in particular, for any $x\in V$ we have
$$
\tD^{{\rm sign}}_{f(x)}\, ' \exp (-t (\tD^{{\rm sign}}_{f(x)}\,')^2) = U_x \tD^{{\rm sign}}_{x}  \exp (-t (\tD^{{\rm sign}}_{x})^2)  U_x^{-1}.
$$
Now, by definition of the trace $\tau'$ associated with the image measure $\nu'$, one easily shows that
$$
\tau' (U_x \tD^{{\rm sign}}_{x}  \exp (-t (\tD^{{\rm sign}}_{x})^2)  U_x^{-1} ) = \tau (\tD^{{\rm sign}}_{x}  \exp (-t (\tD^{{\rm sign}}_{x})^2)).
$$
Therefore, the $f_*\Lambda (\nu)$ measured eta invariant of the $G'$-invariant family $({\tilde D}^{{\rm sign}}\,'_{x'})_{x'\in V'}$ as defined by Peric in \cite{Peric} coincides with the $\Lambda(\nu)$ measured eta invariant of the $G$-invariant family $({\tilde D}^{{\rm sign}}_x)_{x\in V}$. On the other hand and as we already observed, these measured eta invariants coincide with ours for the $\Gamma'$-invariant and $\Gamma$-invariant families of signature operators on $\tM'\times T'$ and $\tM \times T$ respectively. Hence the proof is complete.


\end{proof}

\begin{corollary} Let $(V,\maF, \nu)$ and $(V',\maF',\nu')$ be two foliated bundles
as above and assume that there exists a leafwise smooth homeomorphism 
between $(V,\maF)$ and $(V',\maF')$ such that $f_*\nu=\nu'$. Then 
$$\rho^\nu (V,\maF)= \rho^{\nu'}(V',\maF')\,.$$
\end{corollary}

\begin{proof}
We use the two previous propositions. The first one allows to compute $\rho^\nu (V,\maF)$ using any metric $g$. Then we apply the same proposition to $\rho^{f_*\nu}(V',\maF')$ and compute it using the image metric $f_*g$. Finally, the second proposition allows to finish the proof.
\end{proof}

\section{ Loops, determinants and Bott periodicity}\label{loops+bott}

 As before, let $\A_m$ be the maximal $C^*$-algebra of the groupoid $T\rtimes \Gamma$; let $\E_m$
be the $\A_m$-Hilbert module considered in the previous sections. Thus $\E_m$ is
obtained by completion of the $\A_c$-Module $C^\infty_c (\widetilde{M}\times T, \widehat{E})$.
Let $\D_m$ be the regular unbounded $\A_m$-linear operator induced by a $\Gamma$-equivariant
family of Dirac operators.
Let $$\mathcal{I} \K_{\A_m}(\E_m):= \{A\in \maB_{\A_m} (\maE_m) \text{ such that } A-\Id \in \K_{\A_m}(\E_m)
\text{ and } A \text { is invertible}\}$$
Let $\Omega ( \mathcal{I} \K_{\A_m}(\E_m)) $ be the space of homotopy classes of loops in 
$\mathcal{I} \K_{\A_m}(\E_m)$ which contain the identity operator. Then, using  the inverse of the Bott isomorphism  $\beta^{-1}: 
  \Omega (\mathcal{I}\K_{\A_m}(\E_m))\to K_0 (\K _{\A_m}(\E_m))  $, the isomorphism 
$(\chi^{-1}_m)_* : K_0 (\K _{\A_m}(\E_m)) \to K_0 (\maB^E_m)$
induced by  $\chi_m : \maB^E_m \to K_{\A_m}(\E_m)$,
and the inverse of the Morita isomorphism $\maM_m : K_0 (\A_m)\to K_0 (\maB^E_m)$
of Proposition \ref{prop:morita-compatible},  we obtain an isomorphism 
$$\Omega (\mathcal{I}\K_{\A_m}(\E_m))\stackrel{\beta^{-1}}{\longrightarrow}
K_0 (\K _{\A_m}(\E_m)) \stackrel{(\chi^{-1}_m)_*}{\longrightarrow} K_0 (\maB^E_m)
 \stackrel{ \maM_m^{-1}}{\longrightarrow} 
 K_0 (\A_m)\,$$
We denote by  $\Theta:\Omega (\mathcal{I}\K_{\A_m}(\E_m))\rightarrow K_0 (\A_m)$  the composition of these isomorphisms.
Recall the representations 
$$\pi^{{\rm reg}} :  \maB^E_m\to W^*_\nu(G;E)\,;\quad 
\pi^{{\rm av}}:  \maB^E_m\to W^*_\nu(V,\maF;E).$$

Given a morphism $\alpha$ between two $C^*$-algebras, we denote, with obvious abuse
of notation,  by $\Omega\alpha$
 the induced map on homotopy classes of loops.
We thus obtain maps $\Omega \pi^{{\rm reg}}$, $\Omega \chi_m^{-1}$, $\Omega\pi^{{\rm av}}$;
we define 
$$\sigma^{{\rm reg}}: \Omega (\mathcal{I}\K_{\A_m}(\E_m))\rightarrow
 \Omega (\mathcal{I}\K (W^*_\nu (G;E)))\,;
\quad 
\sigma^{{\rm av}}: \Omega (\mathcal{I}\K_{\A_m}(\E_m))\rightarrow \Omega (\mathcal{I}\K (W^*_\nu (V,\maF;E)))
$$
with 
$$ \sigma^{{\rm reg}}:= \Omega \pi^{{\rm reg}}\circ\Omega \chi_m^{-1}\,,\quad
\sigma^{{\rm av}}:= \Omega \pi^{{\rm av}}\circ\Omega \chi_m^{-1}\,.$$

\noindent
Recall, finally, that if $\ell$ is a loop in $\mathcal{I} L^1 (W^*_\nu (V,\maF;E))$, or
more generally in $\mathcal{I}\K(W^*_\nu (V,\maF;E))$,
then $\ell$ has a well defined determinant $w^\nu_{\maF}(\ell)\in \CC$. Similarly,
if $\ell$  is a loop in $\mathcal{I} L^1 (W^*_\nu (G;E))$, or
more generally in $\mathcal{I}\K(W^*_\nu (G;E))$,
then $\ell$ has a well defined determinant $w^\nu (\ell)\in \CC$.

\begin{proposition}\label{prop:diagram}
The following diagram commutes:

\begin{picture}(400,70)

\put(70,55){$\Omega  (\mathcal{I}\K_{\A_m}(\E_m)) $}
\put(80,30){$\downarrow$}
\put(70,5){$\Omega (\mathcal{I}\K (W^*_\nu (V,\maF;E)))$}
\put(85,30){$\sigma^{{\rm av}}$}

\put(200,63){$ \Theta $}
\put(180,59){\vector(1,0){60}}
\put(180,8){\vector(1,0){60}}
\put(200,12){$w^\nu_\maF $}

\put(250,55){$K_0 (\A_m)$}
\put(250,30){$\downarrow$}
\put(250,5){$\CC$.}
\put(260,30){$ \tau^\nu_{{\rm av},*}$}
\end{picture}

Similarly, the following diagram commutes:

\begin{picture}(400,70)

\put(70,55){$\Omega  (\mathcal{I}\K_{\A_m}(\E_m)) $}
\put(80,30){$\downarrow$}
\put(70,5){$\Omega (\mathcal{I}\K (W^*_\nu (G;E)))$}
\put(85,30){$\sigma^{{\rm reg}}$}

\put(200,63){$ \Theta $}
\put(180,59){\vector(1,0){60}}
\put(180,8){\vector(1,0){60}}
\put(200,12){$w^\nu $}

\put(250,55){$K_0 (\A_m)$}
\put(250,30){$\downarrow$}
\put(250,5){$\CC$.}
\put(260,30){$ \tau^\nu_{{\rm reg},*}$}
\end{picture}

\end{proposition}

\begin{proof}
Recall that  for a $C^*$-algebra $A$ the Bott isomorphism $\beta: 
K_0 (A)\rightarrow K_1 (SA) $ is given by the map $[p]\to [(\exp(2\pi i t p)]$;
as there will be several $C^*$-algebras involved, we denote this map by $\beta_A$.
We
observe that
$$ \beta_{\maB_m^E} \circ (\chi^{-1}_m)_* = \Omega (\chi^{-1}_m)_* \circ \beta_{\K_{\A_m}(\E_m)}\,.$$
Therefore,
\begin{eqnarray*}
\Omega \pi^{{\rm av}} \circ \beta_{\maB^E_m} \circ (\chi^{-1}_m)_* \circ \beta^{-1}_{\K_{\A_m}(\E_m)} &=&
\Omega \pi^{{\rm av}} \circ \Omega (\chi^{-1}_m)_* \circ ( \beta_{\K_{\A_m}(\E_m)} \circ \beta_{\K_{\A_m}(\E_m)}^{-1})\\
&=&\Omega \pi^{{\rm av}} \circ \Omega (\chi^{-1}_m)_*\\
&=& \sigma^{{\rm av}}
\end{eqnarray*}
On the other hand, by definition of $\Omega \pi^{{\rm av}}$,
$$\Omega \pi^{{\rm av}}\circ  \beta_{\maB^E_m}= \beta_{\K (W^*_\nu (V,\maF;E))} \circ \pi^{{\rm av}}_*;$$
therefore 
\begin{eqnarray*}
w^\nu_{\maF} \circ  \Omega \pi^{{\rm av}} \circ \beta_{\maB^E_m} &=& w^\nu_{\maF} \circ 
\beta_{\K (W^*_\nu (V,\maF;E))} \circ \pi^{{\rm av}}_*\\
&=& \tau^\nu_\maF \circ \pi^{{\rm av}}_*\\
&=&   \tau^\nu_{{\rm av},*}
\end{eqnarray*}
where $\tau^\nu_{{\rm av}}$ is the trace on $\maB^E_m$ as defined in 
Subsection \ref{subsec:traces} and with the equality $w^\nu_{\maF} \circ 
\beta_{\K (W^*_\nu (V,\maF;E))}=\tau^\nu_\maF$ proved by direct computation. To finish the proof
we simply apply Proposition \ref{prop:morita-compatible}.
\end{proof}

\begin{definition}\label{def:determinants}
We shall denote by $w^\nu_{{\rm av}}: \Omega  ( \mathcal{I}\K_{\A_m}(\E_m)) \to \CC
$ and $w^\nu_{{\rm reg}}: \Omega  ( \mathcal{I}\K_{\A_m}(\E_m)) \to \CC$ the compositions
$w^\nu\circ \sigma^{{\rm av}}$ and $w^\nu\circ \sigma^{{\rm reg}}$ respectively.
\end{definition}

We can summarize the previous Proposition by the following two equations
\begin{equation}\label{eq:commute}
w^\nu_{{\rm av}}= \Theta\circ \tau^\nu_{{\rm av},*}\,,\quad \quad 
w^\nu_{{\rm reg}}= \Theta\circ \tau^\nu_{{\rm reg},*}
\end{equation}

\begin{remark}\label{paths-vs-loops}
Definition \ref{def:determinants} can be extended to a path in $ \mathcal{I}\K_{\A_m}(\E_m)$
provided the two extreme points are mapped 
by $\pi^{{\rm reg}}\circ \chi_m^{-1}$ and $\pi^{{\rm av}}\circ \chi_m^{-1}$ into $\tau^\nu$ trace class
and $\tau^\nu_{\maF}$ trace class perturbations of the identity respectively.
\end{remark}

%

%

\section{On the homotopy invariance of rho on foliated bundles}\label{sec:keswani}

Before plunging into foliated bundles and the foliated homotopy invariance of the signature
rho invariant defined in Section \ref{sec:foliated-rho}, we digress briefly 
and treat a general orientable measured foliation $(V,\maF)$. We denote by $\Lambda$ the holonomy
invariant transverse measure. We fix a longitudinal riemannian metric
on $(V,\maF)$ and we denote by $D^{{\rm sign}}$ the associated longitudinal
signature operator. Let $G$ be the monodromy groupoid associated to $(V,\maF)$.
Then, as already remarked,  Peric  has defined in \cite{Peric} a foliated eta invariant 
$\eta^\Lambda(\tD^{{\rm sign}})$, with $\tD^{{\rm sign}}$ the lift of $D^{{\rm sign}}$ to the monodromy covers, a $G$-equivariant operator on $G$.
The work of Peric employs the holonomy groupoid, but is is not difficult to see that
his arguments apply to the monodromy groupoid as well.
Ramachandran, on the other hand, has defined in \cite{Ra} an eta invariant 
$\eta^\Lambda (D^{{\rm sign}})$
using the measurable 
groupoid defined by the foliation, as we have already observed. 
We infer that the definition of foliated rho invariant is basically present in the literature. It suffices to define $\rho^\Lambda (D^{{\rm sign}}):= \eta^\Lambda(\tD^{{\rm sign}})-\eta^\Lambda(D^{{\rm sign}})$.
Assume now that $G^x_x$ is torsion-free for any $x\in V$, then Connes has defined in \cite{Co}
a Baum-Connes map $K_* (BG)\rightarrow K_* (C^*_{{\rm reg}}(V,\maF))$ which factors through a maximal Baum-Connes map with values in the $K$-theory of the maximal $C^*$-algebra $C^*_{{\rm max}}(V,\maF)$.
 Here $BG$ is the classifying space of the monodromy groupoid, see \cite{Co}, page 126. If $(V,\mathcal{F})$
 is equal to the foliated bundle $V=\tM\times_\Gamma T$, then $BG$ is given by 
 the homotopy quotient $E\Gamma\times_\Gamma T$, with  $E\Gamma$ equal to  the universal space for $\Gamma$ principal bundles. The Baum-Connes conjecture  states that the Baum-Connes map is an isomorphism. We shall make a stronger assumption here, namely that the {\it maximal} Baum-Connes map is an isomorphism. This is a non trivial assumption and even if it is known to be satisfied for instance for amenable actions, there are simple examples where it fails to be true. The general conjecture
one would then like to make goes as follows.

\medskip
\noindent
Let $(V',\maF')$ be another foliation, endowed with a holonomy
invariant transverse measure $\Lambda '$ and let $f:(V, \maF)\to (V',\maF')$ be a leafwise 
measure preserving 
homotopy equivalence. 

\medskip
{\bf Conjecture:}
{\it If $G^x_x$ is torsion-free for any $x\in V$ and  $K_* (BG)\rightarrow K_* (C^*_{{\rm max}}(V,\maF))$
is an isomorphism, then $\rho^\Lambda (D^{{\rm sign}})=\rho^{\Lambda '} (D^{{\rm sign}} \,')$}

\medskip
We shall now specialize to foliated bundles.
Let $\Gamma$, $T$ and $\tM$ be as in the previous sections. Let $V:= \tM\times_\Gamma T$
and let $(V,\maF)$ 
be the associated foliated  bundle. We assume the existence of a $\Gamma$-invariant
measure on $T$; let $\Lambda (\nu)$ be the associated holonomy invariant transverse measure on $(V, \maF)$.
Let $D=(D_L)_{L\in V/\maF}$ be a longitudinal Dirac-type operator. Let $\tD=(\tD_\theta)_{\theta\in T}$ be
the associated $\Gamma$-equivariant family of Dirac operators.
As already remarked the eta invariant of Peric (with  the monodromy groupoid replacing the holonomy groupoid in Peric' definition),
$\eta^{\Lambda(\nu)} (\tD)$, is equal to our $\eta^\nu_{{\rm up}} (\tD)$. Similarly,  the eta invariant of Ramachandran, $\eta^{\Lambda (\nu)} (D)$, is equal to
our $\eta^\nu_{{\rm down}} (D)$. Thus the rho invariant 
$\rho^{\Lambda(\nu)} (D)$ defined above,  is indeed
equal to our rho invariant  $\rho^\nu (D;V,\maF)$.
Assume now that 
 $\tM'$ 
 is  the $\Gamma'$ universal covering of a compact manifold $M'$
 and  let $T'$ be  
 a compact space endowed with a continuous action of $\Gamma'$ by homeomorphisms. We consider $\tM'\times T'$ and the foliated  bundle
 $V':= \tM' \times_{\Gamma'} T'$. Let  $(V',\maF')$ be  the associated foliated space.
 We assume the existence of a $\Gamma'$-invariant measure $\nu'$  on $T'$ and we let $\Lambda(\nu')$
 the associated transverse measure on $(V',\maF')$.
Given a measure preserving foliated homotopy equivalence $f: V\to V'$,  
we can apply  the general conjecture stated in the previous
section to the invariants $\rho^{\Lambda(\nu)} (D)$, $\rho^{\Lambda(\nu')} (D')$ with $D$ and $D'$
denoting now the signature operators. We obtain in this way a conjecture about the
homotopy invariance of the signature rho invariant $\rho^\nu (D;V,\maF)$ defined and studied in this paper;
we shall deal with the general conjecture on foliated spaces in a different paper.
 In the rest of this Section, we shall tackle the homotopy invariance 
 of rho for the special  
homotopy equivalences descending from equivariant homotopies
$\tf : \tM\times T \rightarrow \tM'\times T'$ as described in the previous section.

\subsection{The Baum-Connes map for the discrete groupoid $T\rtimes \Gamma$}

In order to tackle the homotopy invariance of our $\rho^\nu (D^{{\rm sign}}; V,\maF)$
we first need to describe in the most geometric way the Baum-Connes map
relevant to foliated bundles. 
This subsection is thus devoted to recall the definition of the Baum-Connes map with coefficients
in the $\Gamma-C^*$-algebra $C(T)$ and, more importantly, to give a very geometric
description of it.
There are indeed several definitions
available in the literature, with proofs of their compatibility sometime missing.
The differences are all concentrated in the domain and, consequently,
 in the definition of the application; the target is always the same, namely $K_* (C(T)\rtimes_r \Gamma)$
(which is nothing but  $K_* ( \maA_r)$ in our notation).
Notice that if $T$ is a point, we also have two different possibilities for the classical
Baum-Connes map, depending on whether we consider, on the left hand side,
the Baum-Douglas definition
of K-homology or, instead,  Kasparov's definition; although 
the compatibility of the two pictures has been assumed for many years, a complete
proof only appeared  recently, see the paper \cite{BHS}.
Going back to our more general situation,
we begin with  the Baum-Connes-Higson definition \cite{BCH}, which is given is terms
of Kasparov KK-theory and the intersection product:
\begin{equation}\label{BCH}
\mu_{BCH}: K^\Gamma_j (\underline{E}\Gamma;C(T))\rightarrow K_j (C(T)\rtimes_r \Gamma)
\end{equation}
The group on the left is, by definition, 
$$\lim_{X\subset\underline{E}\Gamma} KK^j_\Gamma (C_0 (X),C(T))$$ 
with the direct limit taken over the directed system of all $\Gamma$-compact
subset of $\underline{E}\Gamma$.
Similarly, there is a maximal Baum-Connes-Higson map:
\begin{equation}\label{BCH-max}
\mu_{BCH}: K^\Gamma_j (\underline{E}\Gamma;C(T))\rightarrow K_j (C(T)\rtimes_m \Gamma)
\end{equation}
Next, we have the original definition of Baum and Connes \cite{BC-enseignement}, with the left hand side
defined in terms of Gysin maps:
\begin{equation}\label{BC}
\mu_{BC}: K^j (T,\Gamma)\rightarrow K_j (C(T)\rtimes_r \Gamma)
\end{equation}
We are not aware of a published proof of the compatibility of these two maps.\\
There is a third
description of the Baum-Connes map with coefficients in $C(T)$: consider as set of cycles
the (isomorphism classes of) pairs $(X,E\to X\times T)$ where $X$ is a spin$_c$  proper $\Gamma$-manifold and $E$
is a $\Gamma$-equivariant vector bundle on $X\times T$ ; define the usual Baum-Douglas 
equivalence relation on these cycles, bordism, direct sum and bundle modification;
we obtain a group that we denote by $K_j^{{\rm geo}} (T\rtimes \Gamma)$ with $j=\dim M\;
{\rm mod}\;  2$. 
The Baum-Connes map in this case is denoted
\begin{equation}\label{B}
\mu_{\rtimes }: K_j^{{\rm geo}} (T\rtimes \Gamma) \rightarrow K_j (C(T)\rtimes_r \Gamma)
\end{equation}
and is very simply described as the map that associates to $[X,E\to X\times T]$
the index class of the $\Gamma$-equivariant family $(D_\theta)_{\theta\in T}$,
with $D_\theta$ the spin$_c$ Dirac operator on $X$ twisted by $E\big |_{X\times\{\theta\}}$.
Also in this case we have a maximal version of the map:
\begin{equation}\label{B-max}
\mu_{\rtimes }: K_j^{{\rm geo}} (T\rtimes \Gamma) \rightarrow K_j (C(T)\rtimes_m \Gamma)
\end{equation}
Thanks to the Ph.D. thesis of Jeff Raven \cite{Raven} we know that the two groups
$K^\Gamma_j (\underline{E}\Gamma;C(T))$ and $K_j^{{\rm geo}} (T\rtimes \Gamma) $
are isomorphic and the two pairs of maps 
 \eqref{BCH}, \eqref{B} and  \eqref{BCH-max}, \eqref{B-max} are compatible; the proof of Raven's isomorphism is far from being trivial. 
 Notice that, as in \cite{keswani3}, we can consider orientable manifolds instead of spin$_c$ manifolds; thus
 the set of cycles for this version of Raven's group is given by pairs $(X, E\to X\times T)$ 
 with $X$ an orientable proper  riemannian $\Gamma$-manifold
 and $E$ a $\Gamma$-equivariant vector bundle on $X\times T$ endowed with an equivariant   Clifford-module  structure with respect to
${\rm Cl}(T^*X)$. Introduce on these cycles the equivalence relation given by bordism,
 direct sum and bundle modification as in \cite{keswani3} (Subsection 2.2, pages 59 and 60).
 The resulting group will be isomorphic to  $K_j^{{\rm geo}} (T\rtimes \Gamma) $ and the resulting Baum-Connes map will
 be compatible.
  In the rest of this work we  look at the stability properties
of our foliated rho-invariant for the signature operator under a bijectivity hypothesis on the map \eqref{B-max}. However, in order to exhibit examples we do need to use the compatibility between 
 \eqref{BCH-max} and \eqref{B-max}; indeed almost
all examples where the Baum-Connes assumption is satisfied are proved using
the Baum-Connes-Higson description.

\subsection{Homotopy invariance of $\rho^\nu (D^{{\rm sign}};V, \maF)$
for special homotopy equivalences}

We can state the main result of this Section as follows:

\begin{theorem}\label{homotopy-invariance}
Let $V:=\tM\times_\Gamma T$ and  
 $V':=\tM' \times_{\Gamma'} T'$ be two
foliated flat bundles, with $\Gamma$ and $\Gamma'$  discrete {\bf torsion-free} groups \footnote{The assumption on $\Gamma$ and $\Gamma'$ can be replaced by the weaker assumption that the isotropy groups are torsion-free, as can be checked in the proof.}.
Assume that there exists a {\bf special} leafwise 
homotopy equivalence $f : (V,\maF)\to (V',\maF')$ and let 
$k:T\to T'$ be the  induced equivariant
homeomorphism . Let $\nu'$ be a $\Gamma'$-invariant measure on $T'$; let 
$\nu:= k^* \nu'$ be the corresponding $\Gamma$-invariant measure on $T$.
Assume that the Baum-Connes map \eqref{B} for the maximal $C^*$-algebra 
$$\mu_{\rtimes }: K_j^{{\rm geo}} (T\rtimes \Gamma)\longrightarrow  K_* (C(T)\rtimes_{{\rm max}} \Gamma)$$
is {\bf bijective}. Then
\begin{equation}
\rho^\nu (V,\mathcal{F})\,=\,\rho^{\nu'} (V',\mathcal{F}')\,.
\end{equation}
\end{theorem}

\medskip
\noindent
{\bf Sketch of the proof.} We follow the method of Keswani, see \cite{Kes1}, \cite{Kes2} and \cite{keswani3}.
We simply denote the relevant signature operators by $D'=(D'_{L'})_{L'\in V'/\maF'}$, $\tD'=(\tD'_\theta)_{\theta\in T} $, 
$\D'_m$ and $D=(D_{L})_{L\in V/\maF}$, $\tD=(\tD_\theta)_{\theta\in T}$, $\D_m$. We shall first assume that $T=T^\prime$
and $\Gamma=\Gamma'$. Consider,
with obvious notation, 
the trivial $\Gamma$-equivariant  fibration $(\tM'  \sqcup -\tM)\times T\to T$
as well as the foliated space $(X,\maF^\sqcup)$, with $X:=V'\sqcup (-V)$
and $\maF^\sqcup$ induced by $\maF$ and $\maF'$. The longitudinal Grassmann bundles
on $V'$ and $-V$ define a longitudinally smooth  bundle $H$ over the foliated space
$X$.  
Let $\widehat{H}$ be the equivariant vector bundle on $(\tM'  \sqcup -\tM)\times T\to T$
obtained by pulling back the bundle $H$. All the constructions explained in the previous
sections extend to 
$(\tM' \sqcup -\tM)\times T\to T$ and $\widehat{H}$  as well as to 
$(X,\maF^\sqcup)$ and $H$. More precisely, we treat $(\tM' \sqcup -\tM)\times \{\theta\}$ as the leaf of the product foliation even if it is not connected and we consider the induced lamination $\maF^\sqcup$. So the leaves are not connected for us. Clearly, we can define the $C^*$-algebra $\maB_m^H$ as the completion of the convolution algebra	of compactly supported continuous sections over the corresponding monodromy groupoid $G^\sqcup$, with respect to the direct sum of the regular representations in $L^2(\tM', \tE ')\oplus L^2(\tM, \tE)$. Note that $G^\sqcup$ can be identified with the space
$$
G^\sqcup = [(\tM' \sqcup -\tM) \times  (\tM' \sqcup -\tM) \times T] / \Gamma.
$$
The reader should note that $\maB_m^H$ is different from the $C^*$-algebra of the monodromy groupoid of the disjoint union of the two foliations $(V',\maF')$ and $(V,\maF)$, and that $\maB_m^H$ is Morita equivalent to the $C^*$-algebra $\maA_m$. Indeed, we then have a well defined 
 $\A_m$-Hilbert module $\H_m$ (this is nothing but $\E^\prime_m\oplus \E_m$)
 as well as an isomorphism 
 $\chi_m : \maB_m^H\to \K_{\A_m}(\H_m)$ constructed in the same way as in the previous sections.
 Now, there are again representations
 $$\pi^{{\rm reg}} = (\pi_\theta^{{\rm reg}})_{\theta\in T}: \maB_m^H\to W^*_\nu (G^\sqcup;H)\,,\quad \pi^{{\rm av}}= (\pi_\theta^{{\rm av}})_{\theta\in T}: \maB_m^H\to W^*_\nu 
 (X, \maF^\sqcup;H).$$ Here, the von Neumann algebras $W^*_\nu (G^\sqcup;H)$ and $W^*_\nu 
 (X, \maF^\sqcup;H)$ are defined using $\nu$-essentially bounded families over $T$ as in the previous sections, except that the operators act on the direct sums of the Hilbert spaces. Said differently, we are again simply allowing disconnceted leaves. Finally, the previous constructions of traces and determinants on foliations, work as well for these two von Neumann algebras. 
 So,  extending obviously the constructions of Section \ref{loops+bott}, using the composition operation of Hilbert modules, we can consider  determinants 
$$w^\sqcup_{{\rm reg}}: \Omega( \mathcal{I}\K_{\A_m}(\H_m))\to \CC\,,\quad
w^\sqcup_{{\rm av}}: \Omega( \mathcal{I}\K_{\A_m}(\H_m))\to \CC$$
Following the notation
of Subsection \ref{subsect:winding}, consider the path in $\mathcal{I}\K_{\A_m}(\H_m)$
$$\mathcal{W}_\epsilon:= \left( \psi_t (\D^\prime_m)\oplus (\psi_t (\D_m))^{-1}
\right)^{t=1/\epsilon}_{t=\epsilon}\,.$$
Consider 
 $w^\sqcup_{{\rm reg}}(  \mathcal{W}_\epsilon)$ and $w^\sqcup_{{\rm av}}(\mathcal{W}_\epsilon )$
 (one can easily show that the determinants of these paths are indeed well defined, see 
 Remark \ref{paths-vs-loops}). The proof proceeds along the following steps:
 
\begin{itemize}
\item we connect $\psi_\epsilon(\D^\prime_m)\oplus (\psi_\epsilon (\D_m))^{-1}$ to the identity
using the small time path $ST_\epsilon$. This step is based on the injectivity
of the Baum-Connes map and on the homotopy invariance of the signature index class;
\item we connect  $\psi_{1/\epsilon}(\D^\prime_m)\oplus (\psi_{1/\epsilon} (\D_m))^{-1}$ to the identity via the large time path $LT_{1/\epsilon}$. This step is based on the surjectivity of the Baum-Connes map, on the foliated homotopy invariance
of the space of leafwise harmonic forms
and  on the homotopy invarance of the signature index class;
\item we obtain
in this way a loop $\ell$ in $\mathcal{I}\K_{\A_m}(\H_m)$, i.e. an element of $\Omega(\mathcal{I}\K_{\A_m}(\H_m))$;
\item we prove that 
$w^\sqcup_{{\rm reg}}(LT_{1/\epsilon})$ and $w^\sqcup_{{\rm av}}(LT_{1/\epsilon})$ are 
well defined and that 
\begin{equation}\label{large-to-zero}
w^\sqcup_{{\rm reg}}(LT_{1/\epsilon})\to 0 \quad \text{and}\quad  w^\sqcup_{{\rm av}}(LT_{1/\epsilon})\to 0\quad\text{as}\quad \epsilon\downarrow 0
\end{equation}
\item we prove that 
$w^\sqcup_{{\rm reg}}(ST_{\epsilon})$ and $w^\sqcup_{{\rm av}}(ST_{\epsilon})$
are well defined and 
\begin{equation}\label{small-to-zero}
(w^\sqcup_{{\rm reg}}(ST_{\epsilon})- w^\sqcup_{{\rm av}}(ST_{\epsilon}))\to 0 \quad\text{as} \quad\epsilon\downarrow 0
\end{equation}
\end{itemize}
Now consider the map $\Theta: \Omega (\mathcal{I}\K_{\A_m}(\H_m))\to K_0(\A_m)$. By the surjectivity
of the Baum-Connes map one proves, using $\Theta$, the following fundamental equality:
\begin{equation}\label{equality-deter.}
w^\sqcup_{{\rm reg}} (\ell)-w^\sqcup_{{\rm av}}(\ell)=0
\end{equation}
which means that
$$(w^\sqcup _{{\rm reg}}(  \mathcal{W}_\epsilon)- w^\sqcup_{{\rm av}}(\mathcal{W}_\epsilon ))+
(w^\sqcup_{{\rm reg}}(LT_{1/\epsilon})- w^\sqcup_{{\rm av}}(LT_{1/\epsilon}) + (w^\sqcup_{{\rm reg}}(ST_{\epsilon})- w^\sqcup_{{\rm av}}(ST_{\epsilon}))=0$$
Taking the limit as $\epsilon\downarrow 0$, using \eqref{large-to-zero}, \eqref{small-to-zero} and recalling that 
$$\lim_{\epsilon\downarrow 0}(w^\sqcup_{{\rm reg}}(  \mathcal{W}_\epsilon)- w^\sqcup_{{\rm av}}(\mathcal{W}_\epsilon ))=\rho^{\nu'} (D^{{\rm sign}}; V^\prime,\mathcal{F}^\prime)-\rho^\nu(D^{{\rm sign}};V,\mathcal{F})$$
we end the proof in the particular case $T=T'$ and $\Gamma=\Gamma'$.
In the general case we know that, since we have assumed the special homotopy equivalence, $T$ and $T'$ are homeomorphic
and that the two groups are isomorphic. Therefore, the above proof can be adapted easily.

\section{Proof of the homotopy invariance for special homotopy equivalences: details}\label{sec:proof}

We shall now provide more details for the proof of Theorem \ref{homotopy-invariance}; most of our work
in the previous sections will go into the proof. We shall work under the additional assumption that $T=T'$ and
$\Gamma=\Gamma'$.

\subsection{Consequences of surjectivity I: equality of determinants}

The following Proposition will play a crucial role in our analysis. Recall that we have defined
traces $\tau^\nu_{{\rm reg},*}: K_0 (\A_m)\to\CC$ and $\tau^\nu_{{\rm av},*}: K_0 (\A_m)\to\CC$;
where in our notation $\A_m := C(T)\rtimes_m\Gamma$.

\begin{proposition}\label{prop:equality-traces}
Assume the Baum-Connes map $\mu_\rtimes:
K_0^{{\rm geo}} (T\rtimes \Gamma) \rightarrow K_0(C(T)\rtimes_m\Gamma)$
surjective;  then $$\tau^\nu_{{\rm reg},*}=\tau^\nu_{{\rm av},*}$$
\end{proposition}

\begin{proof}
Each K-theory class $\alpha\in K_0(C(T)\rtimes_m\Gamma)$ is, by the surjectivity of $\mu_\rtimes$, 
the index class associated to  a $\Gamma$-equivariant 
family of Dirac-type operators on manifolds without
boundary. Using 
formula \eqref{c-index-equality} (which is a consequence of the analogue of Atiyah's index theorem
on coverings and Calderon's formula),
we end the proof.
\end{proof}

\begin{proposition}
If the Baum-Connes map $\mu_\rtimes: K_0^{{\rm geo}} (T\rtimes \Gamma) \rightarrow K_0(C(T)\rtimes_m \Gamma)$ is surjective, then
$w^\nu_{{\rm av}}$  and $w^\nu_{{\rm reg}}$ coincide 
on 
$\Omega  ( \mathcal{I}\K_{\A_m}(\E_m))  $.
\end{proposition}

\begin{proof}
Recall that $w^\nu_{{\rm av}}: \Omega  ( \mathcal{I}\K_{\A_m}(\E_m)) \to \CC
$ and $w^\nu_{{\rm reg}}: \Omega  ( \mathcal{I}\K_{\A_m}(\E_m)) \to \CC$ 
are defined by passing to the respective Von Neumann algebras
and then taking the de La Harpe - Skandalis determinant there (see Definition \ref{def:determinants}):
in formulae
$$ w^\nu_{{\rm av}}: =w^\nu\circ \sigma^{{\rm av}}\,,\quad w^\nu_{{\rm reg}}:= w^\nu\circ \sigma^{{\rm reg}}$$
Using the commutative diagram
of Proposition \ref{prop:diagram}, as summarized in formula
\eqref{eq:commute}, 
and the  equality of traces on $K_0$ given by Proposition \ref{prop:equality-traces},
we immediately conclude the proof.
\end{proof}

\begin{corollary}
Let $V=\tM\times_\Gamma T$ and $V'=\tM'\times_\Gamma T$ be
two homotopy equivalent foliated bundles as in the previous subsection,
i.e. through a special homotopy equivalence.
Let $\H_m=\E_m^\prime \oplus \E_m$ be the $\A_m$-Hilbert module associated to the disjoint
union of $\tM'\times T$ and $-(\tM\times T)$. Let $\ell$ be a loop in $\Omega
(\mathcal{I}\K_{\A_m}(\H_m))$.\\
If the Baum-Connes map $\mu_\rtimes$ is surjective, then
\begin{equation}\label{equality-det}
w^\sqcup_{{\rm av}} (\ell) = w^\sqcup_{{\rm reg}} (\ell)
\end{equation}
\end{corollary}
If we consider, in particular, the loop $\ell\in \Omega
(\mathcal{I}\K_{\A_m}(\H_m))$
defined in the sketch of the proof
of Theorem \ref{homotopy-invariance}, then we 
 have justified formula \eqref{equality-deter.}.

\subsection{Consequences of surjectivity II: the large time path}

Let $V=\tM\times_\Gamma T$ and $V'=\tM'\times_\Gamma T$ be
two homotopy equivalent foliated bundles as in the previous subsection,
i.e. through a special homotopy equivalence with $\Gamma=\Gamma'$
and $T=T'$. 
We consider the Cayley transforms 
of the regular operators $\D_m: \E_m\to\E_m$  and $\D'_m: \E'_m \to \E'_m$: 
$$\mathcal{U}:= (\D_m - i \Id)(\D_m +i\Id)^{-1}\,,\quad \mathcal{U}':= (\D'_m - i \Id)(\D'_m +i\Id)^{-1}$$
Let $\tf:\tM\times T\to \tM'\times T$ be a fiberwise smooth equivariant map inducing the special
homotopy equivalence between $(V,\maF)$ and $(V',\maF')$; let $g$ and $\tg$ choices 
for the homotopy inverses of $f$ and $\tf$, with $\tg:  \tM'\times T \to \tM\times T$ inducing $g$.
Following \cite{Kes1} (Section 3) one can construct a path  
of unitaries in $\H_m=\E_m^\prime \oplus \E_m$, $\mathcal{V}(t), t\in [0,2]$, connecting 
$\mathcal{U}' \oplus \mathcal{U}^{-1}=\mathcal{V}(0)$ to the identity $\Id_{\H_m}=\mathcal{V}(2)$. 
The path $\mathcal{V}(t), t\in [0,2]$ (which is denoted ${\bf W}(t)$ in  \cite{Kes1}) is
obtained by defining a perturbation  $\sigma(t)$ of the grading operator defining the signature
operator; the definition of $\sigma(t)$, which is really due
to Higson and Roe,  employs 
the pull back operator defined by the homotopy equivalence $\tilde g$ (precomposed and composed
respectively with
an extension to $\E_m$ and $\E'_m$  of the smoothing operators $(\phi (\tD_\theta))_{\theta\in T}$,
$(\phi (\tD'_\theta))_{\theta\in T}$, $\phi$ being 
a rapidly decreasing smooth function with compactly supported Fourier transform). We omit the actual definition
of $\mathcal{V} (t)$ since it is somewhat lengthy and refer instead to \cite{Kes1}, pages 968-969.

Recall that our goal is to construct a path connecting
$ \psi_{1/\epsilon} (\D'_m))\oplus (\psi_{1/\epsilon} (\D_m)))^{-1}$,
(where $ \psi_\alpha(x)=-\exp(i\pi\frac{2}{\sqrt{\pi}}\int_0^{\alpha x} e^{-u^2}du$),
to the identity on $\H_m$.

To this end, notice  that the Cayley transform of the operator $\D_m$ can be expressed
as $-\exp(i\pi \chi (\D_m))$, with $\pi\chi(x)=2\arctan (x)$.
Recall, from \cite{Kes1}, that a chopping function is an odd continuous function  $\mu:\RR\to \CC$
such that $|\mu(x)|\leq 1$ and $\lim_{x\to \pm\infty} \mu(x)=\pm 1$; 
both $\chi(x):= \frac{2}{\pi}\arctan(x)$ and $\phi(x):=\frac{2}{\sqrt{\pi}}\int_0^{x} e^{-u^2}du$ are chopping functions.
Two chopping functions
$\mu_1$ and $\mu_2$ 
can be homotoped one to the other via the straight line homotopy $k_s=(1-s)\mu_1 + s \mu_2$.
Thus
$\mathcal{U'}\oplus \mathcal{U}^{-1}$, which is equal to
$-\exp(i\pi \chi (\D'_m))\oplus -\exp(-i\pi \chi (\D_m))$, can be joined to 
$$-\exp(i\pi \phi (\D'_m))\oplus -\exp(-i\pi \phi (\D_m))\,,\quad \phi(x)=\frac{2}{\sqrt{\pi}}\int_0^x e^{-u^2}du$$
via the path $\K(s):= -\exp(i\pi k_s (\D'_m))\oplus -\exp(-i\pi k_s (\D_m))$. We denote by $LT$
the concatenation of $\K(s)$ and  $\mathcal{V}(t)$; $LT$ is a path joining 
$-\exp(i\pi \phi (\D'_m))\oplus -\exp(-i\pi \phi (\D_m))$, with $ \phi(x)=\frac{2}{\sqrt{\pi}}\int_0^x e^{-u^2}du$,
to the identity.

\begin{definition}
Let $\epsilon>0$ be fixed.
The large time path $LT_{1/\epsilon}$ is the path obtained from the above construction but
with the operators $\D_m$ and $\D'_m$ replaced by $\frac{1}{\epsilon}\D_m$ and $\frac{1}{\epsilon}\D'_m$
respectively.
The large time path connects 
$$ \psi_{1/\epsilon} (\D'_m))\oplus (\psi_{1/\epsilon} (\D_m)))^{-1}\,,\quad \text{ with }\quad \psi_{1/\epsilon}(x)=-\exp\left( i\pi(\frac{2}{\sqrt{\pi}}\int_0^{x/\epsilon} e^{-u^2}du) \right)\,,$$
to the identity.
\end{definition}

For later use, we notice that 
\begin{equation}\label{phi-psi}
\psi_{1/\epsilon}=-\exp( i\pi \phi_{1/\epsilon}) \,,\quad \text{ with } 
\quad \phi_{1/\epsilon}(x)=\frac{2}{\sqrt{\pi}}\int_0^{x/\epsilon} e^{-u^2}du\,.
\end{equation}

For each fixed $\epsilon>0$
 $LT_{1/\epsilon}$ is a path in 
$\mathcal{I} \K_{\A_m}(\H_m)$ (we recall that this is the group
consisting of the operators 
$A\in \maB_{\A_m} (\H_m)$  such that $ A-\Id \in \K_{\A_m}(\H_m)$
 and  $A$  is invertible).
In order to show this property we first recall that at the end of Subsection \ref{subsect:dirac}, Sobolev modules 
$\E_m^{(\ell)}$ were
introduced and the compactness of the inclusion $\E_m^{(\ell)} \hookrightarrow \E_m^{(k)}$, $\ell>k$ was stated.
Observe then that if $\chi$
is any chopping function with the property that $\chi ' \sim 1/x^2$ as $|x|\to \infty$, then,
using the compactness of the inclusion of the  Sobolev  module $\maE^{(1)}_m$
into $\maE_m$, one proves easily that $-\exp(i\pi\chi (\D_m)\in \mathcal{I} \K_{\A_m}(\H_m)$.
Notice now that both $\frac{2}{\pi}\arctan(x)$ and $\frac{2}{\sqrt{\pi}}\int_0^{x} e^{-u^2}du$ 
satisfy this condition; thus $LT_{1/\epsilon}\in \mathcal{I} \K_{\A_m}$.

\subsection{The determinants of the large time path}

Recall the isomorphism $\chi_m : \maB^H_m \to K_{\A_m}(\H_m)$,
and the representations 
$$\pi^{{\rm reg}} :  \maB^H_m\to W^*_\nu(G^\sqcup;H)\,;\quad 
\pi^{{\rm av}}:  \maB^E_m\to W^*_\nu(X,\maF^\sqcup;H)\,,\;\;\;\text{ with } X=V\sqcup (-V')$$
Proceeding  as in Section \ref{loops+bott}, we 
can use $\chi_m^{-1}$ and $\pi^{{\rm reg}}$ in order to define a path $\sigma^{{\rm reg}}(LT_{1/\epsilon})$ in $\mathcal{I}\K (W^*_\nu(G^\sqcup;H))$.The end-points  of this path 
are $\tau^\nu$ trace class perturbations of the identity; thus, see Remark \ref{paths-vs-loops}, the determinant $w^\nu (\sigma^{{\rm reg}}(LT_{1/\epsilon}))$ is well defined and we can set
$$w^\nu_{{\rm reg}} (LT_{1/\epsilon}):= w^\nu (\sigma^{{\rm reg}}(LT_{1/\epsilon}))\,.$$
Similarly, 
$$w^\nu_{{\rm av}} (LT_{1/\epsilon}):= w^\nu_{\maF^\sqcup} (\sigma^{{\rm av}}(LT_{1/\epsilon}))$$
is well defined (and we recall that $\maF^\sqcup$ is the foliation induced on $X=V'\sqcup (- V)$
by the foliations $\maF$ and $\maF'$ on $V$ and $V'$).

\begin{proposition}\label{prop:limit-large}
As $\epsilon\downarrow 0$ we have
\begin{equation}\label{limit-large}
w^\nu_{{\rm reg}} (LT_{1/\epsilon})\longrightarrow 0\,,\quad w^\nu_{{\rm av}} (LT_{1/\epsilon})
\longrightarrow 0
\end{equation}
\end{proposition}

\begin{proof}
Fix $\epsilon>0$ and recall that $LT_{1/\epsilon}$ is the composition of two paths: the path
$\mathcal{V}_{1/\epsilon}$,  connecting 
$$-\exp(i\pi \chi (\frac{1}{\epsilon}\D'_m))\oplus -\exp(-i\pi \chi (\frac{1}{\epsilon}\D_m))
\text{ (with }
\chi(x)=\frac{2}{\pi}\arctan(x)\text{)}\quad \text{ to }\quad \Id_{\H_m}\,,$$
and the straight line path $\mathcal{K}_{1/\epsilon}$,
connecting $ \psi_{1/\epsilon} (\D'_m))\oplus (\psi_{1/\epsilon} (\D_m)))^{-1}$
to $-\exp(i\pi \chi (\frac{1}{\epsilon}\D'_m))\oplus -\exp(-i\pi \chi (\frac{1}{\epsilon}\D_m))$.
Consider $\sigma^{{\rm reg}}(LT_{1/\epsilon})$ in $\mathcal{I}\K (W^*_\nu(G^\sqcup;H))$;
for the signature family $\tP$ associated to $\tM'\sqcup (-\tM)\times T \to T$ 
denote by $\tilde{\Pi}:= (\tilde{\Pi}_\theta)_{\theta\in T}$ the element 
in $W^*_\nu(G^\sqcup;H)$ defined by the family of orthogonal projections onto
the null space. Then, proceeding as in Keswani \cite{Kes1}, one can show that 
$\sigma^{{\rm reg}}(LT_{1/\epsilon})$ converges strongly  to the path 
\begin{equation}
\label{nu-infty}
\tilde{\mathcal{V}}_\infty (t)=\cases -\tilde{\Pi}+\tilde{\Pi}^\perp\,,\quad t\in [-1,3/2]\\
-e(t)\tilde{\Pi}+\tilde{\Pi}^\perp\,,\quad t\in [3/2,2]\endcases
\end{equation}
with
$$e(t)=-\begin{pmatrix} \exp(2\pi i t) & 0
\cr 0  & \exp(-2\pi i t) \cr
\end{pmatrix} 
$$
More precisely: $\sigma^{{\rm reg}} (\mathcal{K}_{1/\epsilon})$ converges
strongly  to the constant path $\tilde{\Pi}+\tilde{\Pi}^\perp$, whereas 
$\sigma^{{\rm reg}} (\mathcal{V}_{1/\epsilon})$ (is homotopic, with fixed end-points, to 
a path that) converges strongly  to $\tilde{\mathcal{V}}_\infty (t)$.
 Similarly, if we denote by $\Pi\in W^*_\nu(X, \maF^\sqcup;H)$, $X=V'\sqcup (-V)$,
 the projection onto the null space of the longitudinal signature operator on $X$,
 then $\sigma^{{\rm av}}(LT_{1/\epsilon})$ converges strongly to the path 
\begin{equation}
\label{nu-infty-bis}\mathcal{V}_\infty (t)=\cases -\Pi+\Pi^\perp\,,\quad t\in [-1,3/2]\\
-e(t)\Pi+\Pi^\perp\,,\quad t\in [3/2,2]\endcases
\end{equation}
We can now end the proof \footnote{Notice that the proof given by Keswani for coverings
is not totally correct; the argument given here corrects the mistakes there.}. Recall
the function $\phi_{1/\epsilon}(x):=\frac{2}{\sqrt{\pi}}\int_0^{x/\epsilon} e^{-u^2}du$, see
formula \eqref{phi-psi};
consider the function $\alpha(x)$ equal to zero for $x=0$, equal to 1 for $x>0$
and equal to $-1$ for $x<0$; let $\varpi_{1/\epsilon}(t)= (1-t)\phi_{1/\epsilon} + t \alpha $ be
the straight-line path joining $\phi_{1/\epsilon}$ to $\alpha$; consider the path 
$$X_{1/\epsilon} (t):= -\exp(i\pi\varpi_{1/\epsilon}(t)(\tD')\oplus 
-\exp(-i\pi\varpi_{1/\epsilon}(t)(\tD))\,.
$$
We notice for later use that as $\epsilon\to 0$, $\phi_{1/\epsilon}$ converges pointwise to $\alpha$;
using once again  the spectral theorem for unbounded operators this means
that, in the strong topology,
\begin{equation}\label{strong}
\phi_{1/\epsilon}(\tP)\longrightarrow \alpha(\tP)\quad\text{as}\quad \epsilon\downarrow 0
\end{equation}
where we recall that $\tP$ denotes the
signature family on $(\tM'\sqcup (-\tM))\times T\to T$.
We go back to the path $X_{1/\epsilon} (t)$, which  is  a path in $W^*_\nu(G^\sqcup;H)$ 
joining 
$\sigma^{{\rm reg}} (\psi_{1/\epsilon} (\D'_m))\oplus (\psi_{1/\epsilon} (\D_m)))^{-1})$,
i.e. $ \psi_{1/\epsilon} (\tD'))\oplus (\psi_{1/\epsilon} (\tD)))^{-1}$, to the constant
path $-\tilde{\Pi}+\tilde{\Pi}^\perp$. 
Consider the loop $\gamma_{1/\epsilon}$ in 
$W^*_\nu(G^\sqcup;H)$ obtained by the concatenation of 
$X_{1/\epsilon} (t)$, $\tilde{\mathcal{V}}_\infty (t)$ and the reverse of 
$\sigma^{{\rm reg}}(LT_{1/\epsilon})$; by the above results the loop
$\gamma_{1/\epsilon}$ is strongly null homotopic, thus its determinant is equal to zero.
Summarizing:
$$w^\nu(\sigma^{{\rm reg}}(LT_{1/\epsilon}))=w^\nu (\tilde{\mathcal{V}}_\infty)+
w^\nu (X_{1/\epsilon} )$$
which can be rewritten as
$$w^\nu_{{\rm reg}} (LT_{1/\epsilon})=w^\nu (\tilde{\mathcal{V}}_\infty)+
w^\nu (X_{1/\epsilon} )$$
Computing
$$\tilde{\mathcal{V}}_\infty (t)^{-1} \frac{d\tilde{\mathcal{V}}_\infty (t)}{dt}=\cases 0\,,\quad t\in [-1,3/2]\\
(2\pi i)\begin{pmatrix} \Id& 0
\cr 0  & -\Id\cr
\end{pmatrix} \tilde{\Pi}\,,\quad t\in [3/2,2] \endcases$$
and recalling that the von Neumann dimension of the null space of the signature
operator is a foliated homotopy invariant, see \cite{HL-betti}, we deduce that 
$w^\nu (\tilde{\mathcal{V}}_\infty (t))=0$. 
Thus the first part of the Proposition will follow from the following result:
$$w^\nu (X_{1/\epsilon}) \longrightarrow 0 \quad \text{as }\quad \epsilon\downarrow 0\,. $$
However, this is clear from \eqref{strong} and the normality of the trace, given that,
by direct computation,
$$ w^\nu (X_{1/\epsilon})=\frac{1}{2\pi i} \tau^\nu \left( \phi_{1/\epsilon} (\tP)-\alpha(\tP)
\right)$$
Essentially the same argument, using the strong convergence of  
$\sigma^{{\rm av}}(LT_{1/\epsilon})$ to $\mathcal{V}_\infty$ (see \eqref{nu-infty-bis}),
shows that
$w^\nu_{{\rm av}} (LT_{1/\epsilon})\longrightarrow 0$.

\end{proof}

\subsection{Consequences of injectivity: the small time path}

So far, we have connected  the $t=1/\epsilon$  endpoint of the path
$$\mathcal{W}_\epsilon:= \left( \psi_t (\D^\prime_m)\oplus (\psi_t (\D_m))^{-1}
\right)^{t=1/\epsilon}_{t=\epsilon}\,\quad\text{in}\quad   \mathcal{I}\K_{\A_m}(\H_m)$$
to the identity using the large time path $LT_{1/\epsilon}$. We also showed that 
$$\lim_{\epsilon\to 0} (w^\nu_{{\rm reg}} (LT_{1/\epsilon}) -w^\nu_{{\rm av}} (LT_{1/\epsilon})) = 0\,.$$
We now wish to close up the concatenation of $\mathcal{W}_\epsilon$ and $LT_{1/\epsilon}$
to a loop based at the identity. This step will be achieved through the small time path
$ST_\epsilon$, a path in $\mathcal{I}\K_{\A_m}(\H_m)$ connecting the $t=\epsilon$
end point of $\mathcal{W}_\epsilon$ to the identity.
We shall want to ensure once again that
\begin{equation}\label{limit-small}
\lim_{\epsilon\to 0} (w^\nu_{{\rm reg}} (ST_{\epsilon}) -w^\nu_{{\rm av}} (ST_{\epsilon})) = 0\,.\end{equation}
The existence of {\it a } path connecting 
$\left( \psi_\epsilon (\D^\prime_m)\oplus (\psi_\epsilon (\D_m))^{-1}
\right)$ to the identity is in fact not difficult and follows from the proof of the Hilsum-Skandalis
theorem; what is more delicate is the construction of a path satisfying the crucial
property \eqref{limit-small}. It is here that the injectivity of the Baum-Connes map
is used.

\medskip
Let $V=\tM\times_\Gamma T$ and $V'=\tM' \times_\Gamma T$ be two homotopy equivalent
foliated bundles as in the previous subsections, with $\tM$ and $\tM'$ orientable.
We fix  leafwise $\Gamma$-equivariant metrics on $\tM\times T$ and
$\tM'\times T$. We denote by $\tD=(\tD_\theta)$,
$D=(D_L)_{L\in V/\mathcal{F}}$ and $\D_m$ respectively the $\Gamma$-equivariant
signature family, the longitudinal signature operator on  $(V,\mathcal{F})$ and the
$\mathcal{A}_m$-linear signature operator on the $\mathcal{A}_m$-Hilbert module
$\maE_m$; we fix similar notations for $V'=\tM' \times_\Gamma T$
and we let as usual $\maH_m=\maE'_m\oplus \maE_m$. We denote by 
$\Lambda$ and $\Lambda'$ the vertical Grassmann bundles on $\tM\times T$ and
$\tM'\times T$ respectively.
Consider  the index classes $\Ind (\D_m)$, $\Ind (\D'_m)$, two elements
in $K_1 (\maA_m)$. By the foliated homotopy invariance of the signature index class
we know that $\Ind (\D_m)=\Ind (\D'_m)$. On the other hand, using the very definition 
of the Baum-Connes map $\mu_\rtimes$, we have
$\Ind (\D_m)=\mu_\rtimes [\tM,\Lambda\to \tM\times T]$ and $\Ind (\D'_m)=
\mu_\rtimes [\tM',\Lambda'\to \tM'\times T ] $,
so that, by the assumed injectivity of $\mu_\rtimes$ we infer that
\begin{equation}\label{equality-in-komology}
[\tM,\Lambda\to \tM\times T]=
 [\tM',\Lambda'\to \tM'\times T ] \quad\text{in}\quad K^{{\rm geo}}_1 (T\rtimes \Gamma)\,.
\end{equation}
This is the information we want to use. Before stating the main result of this subsection
we give a convenient definition.

\begin{definition}\label{def:controlled-chopping}
We shall say that a chopping function $\chi$ is controlled if 
\begin{itemize}
\item the derivative of $\chi$ is a Schwartz function;
\item the Fourier transform of $\chi$ is supported in $[-1,1]$
\footnote{Notice that it is impossible to have,
 as required  in \cite{Kes1}, that $\hat{\chi}$ is smooth and compactly supported
(since, otherwise, $\chi$ itself , which is the Fourier transform of  $\hat{\chi}$,
 would be rapidly decreasing and thus 
{\bf not} a chopping function)};
\item the functions $\chi^2-1$ and $\chi (\chi^2-1)$ are Schwartz
and 
their Fourier transform is supported in $[-1,1]$.
\end{itemize}
\end{definition}
For the existence of such a function, see \cite{Moscovici-Wu-GAFA}.
\begin{theorem}\label{theo:equivalent-cycles}
If $[\tM,\Lambda\to \tM\times T]=
 [\tM',\Lambda'\to \tM'\times T ] $ in  $K^{{\rm geo}}_1 (T\rtimes \Gamma)$ then there exist
a $\Gamma$-proper manifold $Y$, a longitudinally smooth
$\Gamma$-equivariant vector bundle
$\widehat{L}\to Y\times T$ and 
a continuous $s$-path of $\Gamma$-equivariant families on $Y$
$$\tB_s:=(\tB_{s,\theta})_{\theta\in T}\quad s\in (0,1)$$
such that
\begin{enumerate}
\item for each $s\in (0,1)$ and $\theta\in T$,
  $(\tB_{s,\theta})$ is a first order elliptic differential operator on $Y$ acting 
on the sections of $\widehat{L}|_{\tY\times \{\theta\}}$ 
\item the $\maA_m$-Hilbert bundle $\maL_m$ obtained by completing $C^\infty_c (Y\times T, \widehat{L})$
contains $\maE'_m \oplus \maE_m$ as an orthocomplemented  submodule; thus there is an
orthogonal decomposition $\maL_m= (\maE'_m \oplus \maE_m)\oplus (\maE'_m \oplus \maE_m)^\perp$
\item for any controlled chopping function $\chi$ the path 
$-\exp(i\pi \chi(\maB_s))$ is norm continuous in the space of bounded operators in $\maL_m$
(here, for $s\in (0,1)$,
$\maB_s$ denotes the regular $\maA_m$-linear operator
defined by the family $(\tB_{s,\theta})_{\theta\in T}$);
\item we have, in norm topology,
\begin{eqnarray*}
\lim_{s\to 1} (-\exp(i\pi \chi(\maB_s)))&=& \Id_{\maL_m}\\
 \lim_{s\to 0} (-\exp(i\pi \chi(\maB_s)))&=&
\left(-\exp(i\pi\chi(\D'_m))\oplus-\exp(-i\pi\chi(\D_m) \right) \oplus \Id_\perp
\end{eqnarray*} 
with  $\Id_\perp$ denoting the identity on $(\maE'_m \oplus \maE_m)^\perp$.
\item  $-\exp(i\pi \chi(\maB_s))\in \mathcal{I}\maK_{\A_m}$ 
\end{enumerate}
\end{theorem}

\begin{proof}
If $[\tM,\Lambda\to \tM\times T]=
 [\tM',\Lambda'\to \tM'\times T ] $ in  $K^{{\rm geo}}_1 (T\rtimes \Gamma)$, then we know that
 we can pass from $(\tM,\Lambda\to \tM\times T)$ to 
 $(\tM',\Lambda'\to \tM'\times T)$ through a finite number of equivalences. The most
 delicate one is   bordism, so we assume directly that we have a
  manifold
 $X$ endowed with a proper action of $\Gamma$, a $\Gamma$-equivariant bundle
 $\widehat{H}$ on $X\times T$, a proper $\Gamma$-manifold  with boundary $Z'$ 
 and an equivariant vector bundle $\widehat{F}'$ on $Z'\times T$ such that 
the boundary of $Z'$ is equal to $X$ and $\widehat{F}'$ restricted to $\partial Z'\times T$ is
equal to  $\widehat{H}$.
Consider the manifold with cylindrical ends, $Z$, obtained attaching to $Z'$ a cylinder
$[0,\infty)\times X$; consider the cylinder $W=X\times \RR$; these are proper $\Gamma$
manifolds if we extend the action to be trivial in the cylindrical direction;
extend bundles to the cylindrical parts in the obvious way.
The $\Gamma$ manifold $Y$ appearing in the statement of the Theorem is the disjoint
union of $Z$,  $-Z$ and $W$, as in \cite{keswani3}. The bundle 
$\widehat{L}$ is given in terms of
$\widehat{H}$ and its extension to the cylindrical parts. 
The equivariant
families $\tB_{s}$, $s\in (0,1)$, appearing in the statement of the Theorem 
 are {\it explicitly defined} (in \cite{keswani3} see:
 the last displayed formula page 70; the last
displayed formula page 72;
the second displayed formula page 76 and the first displayed formula page 77).
We shall see an example in a moment.
The common feature of these operators is that they are Dirac-type on all
of $Y$ but look like an harmonic oscillator along the cylindrical ends.
Since we have extended the action in a trivial way to the $\RR$-direction
of the cylindrical ends we can decompose the Hilbert module defined on the cylinder
$(X\times \RR)\times T$ as $\maE_m (X)\otimes_{\CC} L^2(\RR)$; using the spectral decomposition
of the harmonic oscillator we see, as in \cite{keswani3},
 that there is an orthogonal decomposition of Hilbert modules $\maE_m (X\times \RR)=
 (\maE_m (X\times \RR))' \oplus^\perp  (\maE_m (X\times \RR))''$ with
  $(\maE_m (X\times \RR))' $ equal to the tensor product of $\maE_m (X)$
  with the 1-dimensional space generated by the kernel of the harmonic oscillator
    and $(\maE_m (X\times \RR))''$ equal to the tensor product of
   $ \maE_m (X)$ with the orthogonal space to this kernel in $L^2 (\RR)$.
    In particular,
 $(\maE_m (X\times \RR))'  \simeq \maE_m (X)$, so that
  the Hilbert module $\maL_m$
  obtained completing $C^{\infty,0}_c (Y\times T, \widehat{H})$ does contain $\maE_m (X)$
  as an orthocomplemented submodule.
  Regarding the statements involving the continuity and limiting properties of
 $-\exp(i\pi \maB_s)$: we shall treat only the first of the four steps  proving 
 Theorem 5.1.10 in \cite{keswani3}.
Thus $Y$ is the cylinder $X\times \RR$ and 
$$\maB_t= \begin{pmatrix} 0&\D_X\cr
\D_X&0 \cr \end{pmatrix} +\frac{1}{t}   \begin{pmatrix} x&\partial_x\cr
-\partial_x& -x \cr \end{pmatrix} \quad\text{with}\quad t\in (0,1].
$$
The operator $\maB_t$ restricted to $ (\maE_m (X\times \RR))' $ is
precisely $\begin{pmatrix} 0&\D_X\cr
\D_X&0 \cr \end{pmatrix}$; let us consider $\maB_t$ restricted to the
orthocomplement $(\maE_m (X\times \RR))''$ and denote it $\maC_t$,
so that $$\maB_t= \begin{pmatrix} 0&\D_X\cr
\D_X&0 \cr \end{pmatrix}\oplus \maC_t\,.$$
We can prove the norm-resolvent
  continuity of $\maC_t$ (this notion extends to the $C^*$-algebraic framework)
  exactly as in \cite{keswani3}; we also obtain that $f(\maC_t)$ goes to 0
  in norm  as $t\to 0$ for any rapidly decreasing function $f$.
  Using the fact that $\chi^2-\Id$ is indeed rapidly decreasing
  we see that $\chi^2(\maC_t)-\Id$ goes to zero in norm as $t\to 0$.
  A similar statement holds for $\chi(\maC_t)(\chi^2(\maC_t)-\Id)$.
 Then, writing as in \cite{Kes1}
  $$-\exp(i\pi z)= h(\pi^2 (1-z^2))+(i\pi z)g (\pi^2 (1-z^2))$$
  with $h$ and $g$ entire,
  we prove that $-\exp(i\pi\chi(\maC_t))$ converges in norm to the identity
  on $(\maE_m (X\times \RR))''$, so that
  $-\exp(i\pi\chi(\maB_t))$ converges to (two copies of) $-\exp(i\pi\D_X)\oplus \Id_\perp$
  as $t\to 0$.
 Of course, it is not true in this case that $-\exp(i\pi\chi(\maB_t))$ converges to
 the identity as $t\to 1$ but the idea is that there will be further paths
 of operators in $\mathcal{I}\maK_{\A_m}$ with the property that their concatenation
 will produce the desired  path, joining   $-\exp(i\pi\D_X)\oplus \Id_\perp$ to the identity
 up to stabilization. For the bordism relation these paths are obtained by adapting to our context,
 as we have done above, the remaining three paths appearing in the treatment of the bordism 
 relation in \cite{keswani3}; see in particular the Subsections 5.1.2, 5.1.3, 5.1.4 there.
 Finally, let us comment about cycles that are equivalent through a bundle
modification. We are thus considering, in general, 
$$(X,E\to X\times T)\sim (X',E'\to X'\times T)\equiv (\widehat{X},\widehat{E}\to \widehat{X}\times T)$$
where, as explained for example in \cite{keswani3}, $\widehat{X}$ is a sphere bundle 
$S^{2n}\to \widehat{X}\xrightarrow{\pi} X$ and $\widehat{E}$ is the tensor product of
 $(\pi\times \Id_T)^* (E)$ and a certain bundle $V$ built out of the Grassmann bundle of $\widehat{X}$; $V$ is defined
 originally on $\widehat{X}$ and  then extended trivially on all of $\widehat{X}\times T$. 
 Consider the two $T$-families of Dirac-type operators defined by the equivariant Clifford 
 modules $E\to X\times T$ and $E'\to X'\times T$ respectively and denote them briefly
 by $P=(P_\theta)_{\theta\in T}$ and  $P'=(P'_\theta)_{\theta\in T}$ (for this argument 
 we thus forget about the tilde). Let $\E_m$ and $\E'_m$ be the two Hilbert modules
 associated to these data and let $\P$ and $\P '$ be the regular operators
 defined by the two families above.
   Then we want to show that there exist 
\begin{itemize}   
 \item  an orthogonal decomposition of Hilbert modules
   $\E'_m=\E_m\oplus \E_m^\perp $ 
   \item  a continuous $s$-path of $\Gamma$-equivariant first order differential operators $R_s:=(R_{s,\theta})_{\theta\in T}$,
   $s\in [0,2)$,
 on $\widehat{X}$ with  $R_0=P'$ and with regular extensions $\mathcal{R}_{s}$, $s\in [0,2)$;
 \item for any controlled chopping function $\chi$ the path $-\exp (i\pi \chi (\mathcal{R}_{s}))$ is norm continuous
 in the space of bounded operators  in $\E'_m$
 \item  $(-\exp (i\pi \chi (\mathcal{R}_{s})))\longrightarrow (-\exp (i\pi \chi (\mathcal{P})))\oplus \Id_\perp$ as $s\to 2$
 \end{itemize}
 The existence of the  $s$-path $R_s:=(R_{s,\theta})_{\theta\in T}$,
   $s\in [0,2)$, is proved following the arguments in \cite{keswani3}, Subsection 5.2: thus
   we write $P'=P^0 + P^1 +Z^0$ where for each $\theta\in T$, $P^1_\theta$ is a vertical operator 
   on the fiber bundle $S^{2n}\to \widehat{X}\xrightarrow{\pi} X$,  $P^0_\theta$ is a  horizontal operator
   defined in terms of  $P_\theta $ and $Z^0_\theta$ is a 0-th order operator. Define $R_s$, for $s\in [0,1]$ as 
   $R_s:= P^0 + P^1 +(1-s) Z^0$ so that $R_0=P'$ as required.
   Next observe, as in \cite{keswani3}, that
   for each $\theta\in T$ the
   vertical operator $P^1_\theta$ has a one-dimensional kernel, when restricted to each sphere of the sphere
   bundle $S^{2n}\to \widehat{X}\xrightarrow{\pi} X$; using the orthogonal projection onto the null space
   of these
   operators on spheres we obtain an orthogonal decomposition $\E'_m=\mathcal{U}\oplus \mathcal{U}^\perp$
   with $\mathcal{U}$ isomorphic to $\E_m$. 
   We can now define $R_s$ for $s\in [1,2)$; consider $R_1$ and its extension to $\E'_m$ which
   is diagonal with respect to the orthogonal decomposition. The restriction of $\mathcal{R}_1$ to 
   $\mathcal{U}$ is, by definition,  $\P^0$, given that $\P^1$ is zero on $\mathcal{U}$; 
   using the isomorphism between   $\mathcal{U}$ and  $\E_m$,
   $\P^0$ can be  connected to $\P$, since they differ by the extension of  a 0-th order operator $Z_1$
  (it will suffice to consider $P^0+(s-1) Z_1$, $s\in [1,2]$).
   For the restriction of $\mathcal{R}_1=\P^1+\P^0$ to 
   $\mathcal{U}^\perp$ we consider instead the open  path $\P^0 + \frac{1}{2-s}\P^1$, $s\in [1,2)$; summarizing, we have defined
   a continuous $s$-path of regular operators  $\mathcal{R}_{s}$, $s\in [0,2)$. 
   Using the fact that
   $(\P^1)^2$ is strictly positive on   $\mathcal{U}^\perp$ one can prove the stated continuity properties, as well
   as the crucial fact that 
   $(-\exp (i\pi \chi (\mathcal{R}_{s})))\longrightarrow (-\exp (i\pi \chi (\mathcal{P})))\oplus \Id_\perp$ as $s\to 2$.\\
   Putting together the above two constructions, the one for the bordism relation and the one for the bundle modification
   relation, one can end the proof of the first four items in the statement of the Theorem.
 We finally tackle the property that   $-\exp(i\pi\chi(\maB_s))\in \mathcal{I}\maK_{\A_m}$.
From the fact that $\chi$ is controlled, 
it suffices to show that $f(\maB_s)$ is in $\K_{\maA_m}$ if
 $f$ is rapidly decreasing; let us see this property for the case of the cylinder considered
 above.
 With respect to the above decomposition, 
 $$f(\maB_s)=f(\begin{pmatrix} 0&\D_X\cr\D_X&0 \cr \end{pmatrix})\oplus f(\maC_t)\,.$$
 and it suffices to see that $ f(\maC_t)$ is compact.
 Write $ f(\maC_t)= (f(\maC_t) (\maC_t^2)^N  )\circ (\maC_t^2)^{-N}$, where
 we recall that $\maC_t^2$ is positive; since $f$
 is rapidly decreasing the first operator is bounded; thus we are left with the task
 of proving that $(\maC_t^2)^{-N}$ is compact. Recall that $\maC^2_t$ is the restriction
 to $ (\maE_m (X\times \RR))''$
 of $(\D^2\otimes \Id_{2\times 2} + t^{-2} X^2)$, with $X=\begin{pmatrix} x&\partial_x\cr
-\partial_x& -x \cr \end{pmatrix}$.
Write $(\maC_t^2)^{-N}$ in terms of the heat kernel, using  the inverse Mellin transform:
$$(\maC_t^2)^{-N}=\frac{1}{(N-1)!}\int_0^\infty \exp(-t \maC_t^2)t^{N-1}dt\,.$$
Observe that the heat kernel of  $(\D^2\otimes \Id_{2\times 2} + t^{-2} X^2)$ decouples.
Using again the invertibility of $\maC^2_t$, the properties of the heat kernel of $\D^2$
and, more importantly,  of the heat kernel of the harmonic oscillator, it is not difficult 
to end the proof.

\end{proof}
Let $\chi_\epsilon(x):= \chi(\epsilon x)$. Then, up to a harmless stabilization,
 the above theorem allows us to connect 
 $\left(-\exp(i\pi\chi_\epsilon (\D'_m))\oplus-\exp(-i\pi\chi_\epsilon (\D_m) \right)$
 to the identity; we denote by $\gamma^\epsilon_1 \in \mathcal{I}\maK_{\A_m}$,
 $\gamma^\epsilon_1\equiv (\gamma^\epsilon_1(s))_{s\in [0,1]}$
 the resulting path.
 Recall, however, that our goal is rather  to connect
 $\left(-\exp(i\pi\phi_\epsilon (\D'_m))\oplus-\exp(-i\pi\phi_\epsilon (\D_m) \right)$
 to the identity, with $\phi_\epsilon (x)=\frac{2}{\sqrt{\pi}}\int_0^{\epsilon x} e^{-u^2} du$.
Take the linear homotopy 
between the two chopping functions $\chi$ and $\phi$ and consider 
$$\maM(t):= t \left( \chi(\epsilon \D'_m) \oplus -\chi(\epsilon \D_m)\right) + (1-t)
 \left( \phi(\epsilon \D'_m)
 \oplus -\phi(\epsilon \D_m) \right)
$$
Consider the path
 $$\gamma^\epsilon_2 (t)=-\exp(i\pi \maM(t))\,.$$
 
\begin{definition}
The small time path $ST_\epsilon$ is the path obtained by the concatenation
of $\gamma^\epsilon_1$ and $\gamma^\epsilon_2$.
$ST_\epsilon$ is a path   in $\mathcal{I}\K_{\A_m}$
and connects  $\psi_\epsilon(\D'_m)\oplus (\psi_\epsilon(\D_m))^{-1}\equiv 
\left(-\exp(i\pi\phi_\epsilon (\D'_m))\oplus-\exp(-i\pi\phi_\epsilon (\D_m) \right)$ to the identity.
\end{definition}

\subsection{{The determinants of the small time path}}

Let $(X,\maF)$, $X=Z\times_\Gamma T$, be a foliated bundle as in the proof of 
Theorem \ref{theo:equivalent-cycles}; let $L$ be a continuous longitudinally
smooth vector bundle on $X$ as in Theorem \ref{theo:equivalent-cycles} . 
Let $\L_m$ be the associated Hilbert $\maA_m$-module.
Let $\maB^L_m$ be the maximal $C^*$-algebra associated to the groupoid 
$G^Z:=(Z\times Z\times T)/\Gamma$.
Recall the isomorphism $\chi_m : \maB^L_m \to K_{\A_m}(\L_m)$,
and the representations 
$$\pi^{{\rm reg}} :  \maB^L_m\to W^*_\nu(G^Z;L)\,;\quad 
\pi^{{\rm av}}:  \maB^L_m\to W^*_\nu(X,\maF;L)\,.$$
Proceeding  as in Section \ref{loops+bott}, we 
can use $\chi_m^{-1}$ and $\pi^{{\rm reg}}$ in order to define a path $\sigma^{{\rm reg}}(ST_{\epsilon})$ in $\mathcal{I}\K (W^*_\nu(G^Z;L))$. The end-points  of this path 
are $\tau^\nu$ trace class perturbations of the identity; thus, see Remark \ref{paths-vs-loops}, the determinant $w^\nu (\sigma^{{\rm reg}}(ST_\epsilon))$ is well defined and we can set
$$w^\nu_{{\rm reg}} (ST_{\epsilon}):= w^\nu (\sigma^{{\rm reg}}(ST_{\epsilon}))\,.$$
Similarly, 
$$w^\nu_{{\rm av}} (ST_{\epsilon}):= w^\nu_{\maF} (\sigma^{{\rm av}}(ST_{\epsilon}))$$
is well defined.
The goal of this subsection is to indicate a proof of the following 

\begin{theorem}\label{theo:limit-small-path}
As $\epsilon\downarrow 0$ we have
\begin{equation}\label{limit-small-path}
w^\nu_{{\rm reg}} (ST_{\epsilon})- w^\nu_{{\rm av}} (ST_{\epsilon})\longrightarrow 0\,.
\end{equation}
\end{theorem}
\begin{proof}
To simplify the notation we shall  assume that the injectivity
radius of $(\tM_\theta,\tg_\theta)$ is greater or equal to 1 for each $\theta\in T$; we also assume that
for each $\theta\in T$ the distance between $\tm$ and $\tm\,\gamma$ is greater than 1
for each $\tm\in\tM_\theta$ and for each $\gamma\in \Gamma(\theta)$, $\gamma\not= e$.
We begin by a few preliminary remarks. Recall that $ST_\epsilon$ is the concatenation
of two paths: $\gamma^\epsilon_1$ and $\gamma^\epsilon_2$.
Using the fundamental Proposition   \ref{F-Comp} we observe that 
$$\sigma^{{\rm reg}} (\gamma^\epsilon_1 (t)) \equiv \sigma^{{\rm reg}} 
(-\exp(i\pi\chi(\maB_t))
= -\exp(i\pi\chi(\tB_t))$$
and
$$\sigma^{{\rm av}} (\gamma^\epsilon_1 (t)) \equiv \sigma^{{\rm av}} 
(-\exp(i\pi\chi(\maB_t))
= -\exp(i\pi\chi(B_t))$$
with $B_t = ((B_t)_L)_{L\in X/\maF}$ the longitudinal differential operator
induced by the $\Gamma$-equivariant family $\tB_t$.
(Once again, here and before the statement of Theorem  \ref{theo:limit-small-path}
we are using a slight extension 
of the results proved in Section \ref{sec:modules},
allowing for manifolds with cylindrical ends and operators that are modeled like harmonic oscillators
along the ends).
Similarly, up to a harmless stabilization by  $\Id_\perp$ (that will in any case disappear
after taking determinants), we can write 
$$\sigma^{{\rm reg}} (\gamma^\epsilon_2 (t)) = -\exp 
i\pi \left( t\chi(\epsilon \tP) +  (1-t)\phi(\epsilon \tP) \right) \,,\quad
\sigma^{{\rm av}} (\gamma^\epsilon_2 (t)) = -\exp 
i\pi \left( t\chi(\epsilon P) +  (1-t)\phi(\epsilon P) \right) 
$$
where  $\tP$ and $P$ are the signature operators on $(\tM'\sqcup (-\tM))\times T\to T$
and on $(X:=V' \sqcup (-V),\maF^\sqcup)$ respectively (this is the notation we had introduced
in the subsection on the large time path). 
 One can prove that for $j=1,2$  the paths
$\sigma^{{\rm reg}} (\gamma^\epsilon_j)$ and $\sigma^{{\rm av}} (\gamma^\epsilon_j)$
are all made of {\it trace class} perturbations of the identity;
moreover,  the determinants
of these two paths are well defined {\it individually} and without the regularizing
procedure explained in Proposition \ref{prop:approximation}. We shall justify this claim in a moment. This property granted, we can  
  break the proof of \eqref{limit-small-path}
into two distinct statements:
\begin{equation}\label{limit-small-path-1}
w^\nu_{{\rm reg}} (  \gamma^\epsilon_1)- w^\nu_{{\rm av}} (\gamma^\epsilon_1)\longrightarrow 0\,.
\end{equation}
\begin{equation}\label{limit-small-path-2}
w^\nu_{{\rm reg}} (  \gamma^\epsilon_2)- w^\nu_{{\rm av}} (\gamma^\epsilon_2)\longrightarrow 0\,.
\end{equation}
We now tackle \eqref{limit-small-path-2} which is slightly easier since it involves 
exclusively operators on manifolds without boundary.
\\
First we observe 
 that to each operator $\tP_\theta$ and $P_L$ we can 
apply the results of \cite{taylor-pso}, \cite{roe-partitioning}. In particular, using the properties of
$\chi$, which is of controlled type, and $\phi$ we have:
\begin{enumerate}
\item $\chi(\tP_\theta)$ and 
$\chi(P_L)$,  are 
given by  0-th order pseudodifferential operators with Schwartz kernel
localized in an uniform  $R$-neighbourhood of the diagonal (remember that
the Fourier transform of $\chi$ is compactly supported); we shall assume without loss of generality
that $R=1$;
\item $\phi(\tP_\theta)$ and 
$\phi(P_L)$  are each one the sum of a   0-th order pseudodifferential operators with Schwartz kernel
localized in an uniform  $R=1$-neighbourhood of the diagonal and of an integral
operator with  smooth kernel;
\item if $\tilde{\chi}$ denotes the linear chopping function equal to $sign(x)$ for $|x|>1$ and equal to $x$ for $|x|\leq 1$ then $(\chi(\tP_\theta)-\tilde{\chi}(\tP_\theta))_{\theta\in T}$ and
$(\phi(\tP_\theta)-\tilde{\chi}(\tP_\theta))_{\theta\in T}$ are $\tau^\nu$ trace class elements
given by longitudinally smooth kernels (indeed, the differences $\chi-\tilde{\chi}$ 
and $\phi-\tilde{\chi}$ are rapidly decreasing);
\item similarly, $(\chi(P_L)-\tilde{\chi}(P_L))_{L\in X/\maF^\sqcup}$ and $(\phi(P_L)-\tilde{\chi}(P_L))_{L\in X/\maF^\sqcup}$
are $\tau^\nu_{\maF^{\sqcup}}$ trace class elements given by longitudinally smooth kernels;
\item consequently,  
$(\chi(\tP_\theta)-\phi(\tP_\theta))_{\theta\in T}$ and 
$(\chi(P_L)-\phi(P_L))_{L\in X/\maF^\sqcup}$ are both trace class elements
given by longitudinally smooth kernels; indeed it suffices to write
$(\chi(\tP_\theta)-\phi(\tP_\theta))_{\theta\in T}= ((\chi(\tP_\theta)-\tilde{\chi}(\tP_\theta))_{\theta\in T}+
(\tilde{\chi}(\tP_\theta)-\phi(\tP_\theta))_{\theta\in T}$
\end{enumerate}
Notice that these properties imply easily the claim we have made about
the  determinants
of  $\sigma^{{\rm reg}}(\gamma^\epsilon_2)$ and $\sigma^{{\rm av}}(\gamma^\epsilon_2)$.
We go back to our goal, i.e. proving \eqref{limit-small-path-2}.
We  observe that since $\gamma^\epsilon_2$ is defined in terms of 
a linear homotopy, we have, by direct computation,
$$w^\nu_{{\rm reg}}(\gamma^\epsilon_2)=-\frac{1}{2} \tau^\nu 
\left( \chi(\epsilon \tP)-\phi(\epsilon \tP) \right)\quad 
w^\nu_{{\rm av}}(\gamma^\epsilon_2)=-\frac{1}{2} \tau^\nu_{\maF^\sqcup}
\left( \chi(\epsilon P)-\phi(\epsilon P) \right)$$
Write
$$ \tau^\nu 
\left( \chi(\epsilon \tP)-\phi(\epsilon \tP) \right)= \tau^\nu 
\left( (\chi(\epsilon \tP)-(\phi(\epsilon \tP)_\epsilon) - ( \phi(\epsilon \tP)- (\phi(\epsilon \tP)_\epsilon) \right)$$
with $(\phi(\epsilon \tP))_\epsilon$ a compression of $\phi(\epsilon\tP)$ to a $\Gamma$-equivariant
$\epsilon$-neighbourhood of  $\{(\tm,\tm,\theta), \tm\in\tM,
\theta\in T\}$ in $\tM\times\tM\times T$. Both $(\chi(\epsilon \tP)-(\phi(\epsilon \tP))_\epsilon)$
and $ ( \phi(\epsilon \tP)- (\phi(\epsilon \tP))_\epsilon) $ are individually $\tau^\nu$ trace class:
indeed the first term is the $\epsilon$-compression of a longitudinally smooth kernel (since $\chi(\epsilon \tP)$ is already  $\epsilon$-local) and it is therefore $\tau^\nu$  trace class;
the second term can be written as the sum $ (\phi(\epsilon \tP)-\chi(\epsilon\tP))+ (\chi(\epsilon \tP) -(\phi(\epsilon \tP))_\epsilon)$ and both terms are trace class; thus 
$$ \tau^\nu 
\left( \chi(\epsilon \tP)-\phi(\epsilon \tP) \right)= \tau^\nu 
\left( (\chi(\epsilon \tP)-(\phi(\epsilon \tP))_\epsilon \right) - \tau^\nu \left( \phi(\epsilon \tP)- (\phi(\epsilon \tP))_\epsilon \right)$$
A similar expression can be written for $\tau^\nu_{\maF^\sqcup}
\left( \chi(\epsilon P)-\phi(\epsilon P) \right)$.
Consider now the difference 
$w^\nu_{{\rm reg}}(\gamma^\epsilon_2)-w^\nu_{{\rm av}}(\gamma^\epsilon_2)$ which
is the sum   \begin{equation}
\label{sum}
\left( \tau^\nu (
\chi(\epsilon \tP)-\phi(\epsilon \tP))_\epsilon) - \tau^\nu_{\maF^\sqcup} 
(\chi(\epsilon P)-\phi(\epsilon P))_\epsilon) \right) + 
\left( \tau^\nu (
\phi(\epsilon \tP)-\phi(\epsilon \tP))_\epsilon) - \tau^\nu_{\maF^\sqcup} 
(\phi(\epsilon P)-\phi(\epsilon P))_\epsilon) \right)\,.
\end{equation}
As already remarked  the two differences 
$\chi(\epsilon \tP)- (\phi(\epsilon \tP))_\epsilon$ and $\chi(\epsilon P)- (\phi(\epsilon P))_\epsilon$
are given by  longitudinally smooth kernel which are supported in an $\epsilon$-neighbourhood of the diagonal.
Proceeding as in the proof of Proposition \ref{prop:index-equal} we shall now prove that
$ \tau^\nu (
\chi(\epsilon \tP)-\phi(\epsilon \tP)_\epsilon) - \tau^\nu_{\maF^\sqcup} 
(\chi(\epsilon P)-\phi(\epsilon P)_\epsilon) $
is in fact equal to zero for $\epsilon$ small enough.
Indeed, consider the $\Gamma$-equivariant family $\chi(\epsilon \tP)$;
we know that $\chi(\epsilon \tP)\in \Psi^0_c (G,E)$.
Similarly, consider 
$\phi(\epsilon \tP)_\epsilon\in \Psi^0_c (G,E)$.
We know that $\chi(\epsilon \tP)- (\phi(\epsilon \tP))_\epsilon\in \Psi^{-\infty}_c (G,E)$
and that this operator extends to an element $\maP_{\chi,\psi}^\epsilon\in
\maK_{\maA_m}(\maH_m)$. Observe now that
$$(\maP_{\chi,\psi}^\epsilon)\otimes_{\pi^{{\rm reg}}_\theta}\Id= 
\chi(\epsilon \tP_\theta)-(\phi(\epsilon \tP_\theta))_\epsilon\,,\quad 
(\maP_{\chi,\psi}^\epsilon)\otimes_{\pi^{{\rm av}}_\theta} \Id= 
\chi(\epsilon P_L)-(\phi(\epsilon P_L))_\epsilon\,\quad\text{with}\quad L=L_\theta$$
Using 
Theorem \ref{theo:functional-calculus} we thus can write
$$ \tau^\nu (
\chi(\epsilon \tP)-\phi(\epsilon \tP))_\epsilon) - \tau^\nu_{\maF^\sqcup} 
(\chi(\epsilon P)-\phi(\epsilon P))_\epsilon)=
\tau^\nu_{{\rm reg}} ( \maP_{\chi,\psi}^\epsilon) - \tau^\nu_{{\rm av}} 
(\maP_{\chi,\psi}^\epsilon)$$
where we have omitted the isomorphism $\chi_m^{-1}:\maK_{\maA_m}(\maH_m)\to \maB^H_m$.
Taking $\epsilon$ small enough and proceeding precisely as in the proof of 
Proposition \ref{prop:index-equal} we see that the right hand side is equal to
zero for $\epsilon$ small enough (it is in this last step
that we use the fact that  $\maP_{\chi,\psi}^\epsilon$  is given by an $\epsilon$-localized 
smoothing kernel).
Finally, the terms in the second summand of \eqref{sum} are individually
zero since they are trace class elements given by longitudinally smooth kernels which restrict
to zero on the diagonal. 
Summarizing: $w^\nu_{{\rm reg}}(\gamma^\epsilon_2)-w^\nu_{{\rm av}}(\gamma^\epsilon_2)=0$
for $\epsilon$ small enough.

We are left with the task
of proving that $\gamma^\epsilon_1$ has  well defined determinants and that
\begin{equation}\label{limit-gamma-1}
\lim_{\epsilon\to 0} w^\nu_{{\rm reg}}(\gamma^\epsilon_1)-w^\nu_{{\rm av}}(\gamma^\epsilon_1)=0\,.
\end{equation}
To this end we begin by writing explicitly the left hand side:
\begin{equation}\label{make-sense-up}
w^\nu_{{\rm reg}}(\gamma^\epsilon_1)=\frac{1}{2\pi i}\int_0^1 \tau^\nu \left (
( -\exp(-i\pi\chi(\epsilon\tB_t)))\frac{d}{dt} (-\exp(i\pi\chi(\epsilon\tB_t))) \right) dt
\end{equation}
\begin{equation}\label{make-sense-down}
w^\nu_{{\rm av}}(\gamma^\epsilon_1)=\frac{1}{2\pi i}\int_0^1 \tau^\nu_{\maF} \left (
( -\exp(-i\pi\chi(\epsilon B_t)))\frac{d}{dt} (-\exp(i\pi\chi(\epsilon B_t))) \right) dt
\end{equation}
provided the right hand sides make sense. To see why the last statement is true,
we begin by making a general comment  on the traces we are using.
Remember that the two paths of operators $\tB_s$ and $B_s$, $s\in (0,1)$,
are defined on foliated bundles
that might have as leaves manifolds with cylindrical ends.
We define the two relevant von Neumann algebras in the obvious way and we
define the two traces $\tau^\nu$ and $\tau^\nu_{\maF}$ as we did in Subsection
\ref{subsec:traces}. Needless to say, an arbitrary  smoothing operator will not be
trace class on such a foliation, since its Schwartz kernel might not be integrable
in the cylindrical direction. (This is the typical situation for the heat kernel associated
to a Dirac operator which restrict to
a $\RR^+$-invariant operator $\frac{d}{dt} + D_\pa$ along the cylindrical ends.)
We now write
$$\exp(i\pi z)= h(\pi^2 (1-z^2))+ (i\pi z) g(\pi^2 (1-z^2))$$
with $h$ and $g$ entire. Recall that $\chi$ is of controlled type; we shall now see that
this implies that  $1-\chi^2(\tB_t)$
is $\tau^\nu$ trace class and $1-\chi^2(B_t)$ is  $\tau^\nu_\maF$ trace class; moreover
these operators
 are given by longitudinally smooth kernels that are supported within
a uniform $(R=1)$-neighbourhood of the diagonal.
These statements
are clear when $(\tM,\Lambda\to \tM\times T)\sim  (\tM',\Lambda' \to \tM' \times T)$ through
a bundle modification or a direct sum of vector bundles
(indeed, from our discussion of the bundle modification relation in the proof of
Theorem \ref{theo:equivalent-cycles}, it is clear that in this case we remain within the category of foliations
of compact manifolds without boundary and it suffices to apply
\cite{roe-foliation} for the latter property
and \cite{FackKosaki} for the first).
If $(\tM,\Lambda\to \tM\times T)\sim  (\tM',\Lambda' \to \tM' \times T)$ through a bordism, then we use the fact that $\tB_{\theta,t}$ and $(B_t)_L$
are again of bounded propagation speed (this is needed  in order to make claims about
their Schwartz kernel) and restrict to  harmonic oscillators along the cylinders
of the relevant manifolds with cylindrical ends (this is needed in order to make claims
about the trace class property). For the trace class property
we also make use of the results in  \cite{FackKosaki}, proceeding as in \cite{keswani3} but using
singular numbers instead on eigenvalues.\\
Using $\exp(i\pi z)= h(\pi^2 (1-z^2))+ (i\pi z) g(\pi^2 (1-z^2))$
we can then conclude, as in \cite{Kes1} Lemma 4.1.7,
 that
 $$\sigma^{{\rm reg}}(\gamma^\epsilon_1 (t))\equiv -\exp(i\pi\chi(\epsilon\tB_t))\quad \text{and}\quad  \sigma^{{\rm av}}(\gamma^\epsilon_1(t)) \equiv -\exp(i\pi\chi(\epsilon B_t))\,,\quad t\in [0,1]$$ are
 piecewise continuosly differentiable in the $L^1$ norm and that  they both have a well defined
 determinant, as we had claimed (notice that in the proof of Lemma 4.1.7 in \cite{Kes1}
 only the controlled property of $\chi$ is used). \\
 Having justified \eqref{make-sense-up} and \eqref{make-sense-down}, we next  make the following\\
{\bf Claim:} {\it there exists  polynomials $p_1, p_2$ such that, uniformly in }$s\in [0,1]$,
\begin{equation}\label{st-estimate2}
||\chi(\tB_s) -  \chi (\epsilon\tB_s)||_1 < p_1(\frac{1}{\epsilon})\,,\quad || \chi(B_s) -  \chi (\epsilon B_s)||_1 < p_2(\frac{1}{\epsilon})
\end{equation}
Assume the Claim; then using the inequality
$$|| AB ||_1 \leq || A ||_1 || B ||_\infty\,,\quad A\in L^1(\maM,\tau)\cap \maM\,,\quad B\in \maM$$
which is valid in any Von Neumann algebra $\maM$ endowed with a faithful normal trace $\tau$,
one can show, proceeding exactly as in Lemma 4.2.8 of
\cite{Kes1},  that there exist polynomials $q_1$ and $q_2$ such that, uniformly in $s\in [0,1]$,
\begin{equation}\label{st-estimate1}
|| \chi^2 (\epsilon\tB_s) - \Id||_1 < q_1(\frac{1}{\epsilon})\,,\quad || \chi^2 (\epsilon B_s) - \Id||_1 < q_2(\frac{1}{\epsilon})
\end{equation}
We first end the proof of \eqref{limit-gamma-1} using \eqref{st-estimate1}.\\
For any entire function $f(z)=\sum_{n=0}^\infty a_n z^n$
we define $[f(z)]_N:=  \sum_{n=0}^N a_n z^n$.
Consider the entire function $h$ in the decomposition $\exp(i\pi z)= h(\pi^2 (1-z^2))+ (i\pi z) g(\pi^2 (1-z^2))$.
Proceeding as in Lemma 4.2.6 in \cite{Kes1} we show using  the first
inequality in \eqref{st-estimate1}
that for each $\alpha>0$ there exists an $\epsilon>0$ and an integer $N_\epsilon$
such that
\begin{itemize}
\item $|| h(\pi^2(\Id-\chi^2(\epsilon\tB_s))- [h( \pi^2(\Id-\chi^2(\epsilon\tB_s) )]_{N_\epsilon}||_1< \alpha$
\item $[h( \pi^2(\Id-\chi^2(\epsilon\tB_s))]_{N_\epsilon}$ is of propagation less than 1
\end{itemize}
Remark here that $N_\epsilon$ is in fact fixed by $\epsilon$ 
and, with our conventions, can be set to be equal to the integral part of $1/\epsilon$. Thus the left hand side
of the above inequality can be thought of as a positive function of $\epsilon$, converging to
0 when $\epsilon\downarrow 0$. 
A similar statement can be made for the derivative of  $h(\epsilon\tB_s)$ with respect to $s$. Applying the same
reasoning to the second summand in the decomposition  $\exp(i\pi z)= h(\pi^2 (1-z^2))+ (i\pi z) 
g(\pi^2 (1-z^2))$ we conclude as in \cite{Kes1}
Lemma 4.2.10, that for each $\alpha>0$ there exists an $\epsilon>0$ and an integer $N_\epsilon$
such that
\begin{eqnarray}\label{st-estimate3}
& | \int_0^1 \tau^\nu \left (
( -\exp(-i\pi\chi(\epsilon\tB_t)))\frac{d}{dt} (-\exp(i\pi\chi(\epsilon\tB_t))) \right) dt -\\
&\int_0^1 \tau^\nu \left (
([ -\exp(-i\pi\chi(\epsilon\tB_t))]_{N_\epsilon})\frac{d}{dt} ([-\exp(i\pi\chi(\epsilon\tB_t))]_{N_\epsilon}) \right) dt |< \alpha
\end{eqnarray}
Similarly, using the second inequality in the Claim and the second inequality in
\eqref{st-estimate1}, we can prove that
for each $\alpha>0$ there exists a $\delta>0$ and an integer $N_{\delta}$
such that
\begin{eqnarray}\label{st-estimate4}
& | \int_0^1 \tau^\nu_{\maF} \left (
( -\exp(-i\pi\chi(\delta B_t)))\frac{d}{dt} (-\exp(i\pi\chi(\delta B_t))) \right) dt -\\
&\int_0^1 \tau^\nu_{\maF} \left (
([ -\exp(-i\pi\chi(\delta B_t))]_{N_\delta})\frac{d}{dt} ([-\exp(i\pi\chi(\delta B_t))]_{N_\delta}) \right) dt |< \alpha
\end{eqnarray}
Since the left hand sides of the inequalities \eqref{st-estimate3},  \eqref{st-estimate4} can be thought of as
positive functions of $\epsilon$ and $\delta$ converging to 0 as  $\epsilon\downarrow 0$ and
$\delta\downarrow 0$, it is clear that we can ensure the existence of a common value, say $\eta$ and $N_\eta$,
for which both inequalities are satisfied.
Consider again the difference $|w^\nu_{{\rm reg}} (\gamma_1^\epsilon) - w^\nu_{{\rm av}} (\gamma_1^\epsilon) |$ that we rewrite as
$|A_\epsilon+B_\epsilon+C_\epsilon|$ with
$$A_\epsilon:= w^\nu_{{\rm reg}} \gamma_1^\epsilon -
\int_0^1 \tau^\nu \left (
([ -\exp(-i\pi\chi(\epsilon\tB_t))]_{N_\epsilon})\frac{d}{dt} ([-\exp(i\pi\chi(\epsilon\tB_t))]_{N_\epsilon}) \right) dt $$
\begin{eqnarray*}B_\epsilon&:= &  \int_0^1 \tau^\nu \left (
([ -\exp(-i\pi\chi(\epsilon\tB_t))]_{N_\epsilon})\frac{d}{dt} ([-\exp(i\pi\chi(\epsilon\tB_t))]_{N_\epsilon}) \right) dt \\
& - &\int_0^1 \tau^\nu_{\maF} \left (
([ -\exp(-i\pi\chi(\epsilon B_t))]_{N_\epsilon})\frac{d}{dt} ([-\exp(i\pi\chi(\epsilon B_t))]_{N_\epsilon}) \right) dt
\end{eqnarray*}
$$ C_\epsilon:= \int_0^1 \tau^\nu_{\maF} \left (
([ -\exp(-i\pi\chi(\epsilon B_t))]_{N_\epsilon})\frac{d}{dt} ([-\exp(i\pi\chi(\epsilon B_t))]_{N_\epsilon}) \right) dt -  w^\nu_{{\rm av}} \gamma_1^\epsilon
$$
Obviously $|A_\epsilon+B_\epsilon+C_\epsilon|\leq |A_\epsilon |+|B_\epsilon|+|C_\epsilon|$.
We know that for each $\alpha>0$ there exists a common $\epsilon$ such that $|A_\epsilon|<\alpha$
and $|C_\epsilon|<\alpha$; on the other hand, using the
fact that $ [-\exp(i\pi\chi(\epsilon B_t))]_{N_\epsilon}$ is of  propagation equal to 1, we can prove,
proceeding as in Proposition \ref{prop:index-equal},
that there exists $\epsilon$ such that $B_\epsilon=0$.
Thus we have proved \eqref{limit-gamma-1} modulo the Claim.\\
We shall prove the Claim for the particular case of the cylinder; let us prove,
for example, the first inequality. Consider
$$\tB_t= \begin{pmatrix} 0&\tD\cr
\tD&0 \cr \end{pmatrix} +\frac{1}{t}   \begin{pmatrix} x&\partial_x\cr
-\partial_x& -x \cr \end{pmatrix} \quad\text{with}\quad t\in (0,1].
$$
Observe that the left hand side of the first inequality in the Claim is nothing
but the last term in inequality (4.3) in \cite{keswani3}.
Proceed now exactly as in the part of the proof of Lemma 4.7 in \cite{keswani3}
that begins with the inequality (4.3). It is not difficult to realize that the proof given there, i.e. the proof of the first inequality in the Claim, can be easily
adapted to our Von Neumann context using
singular numbers and the results of Fack and Kosaki. More precisely, the operator $\tB_t^2$ can be diagonalized with respect to the eigenfunctions of the operator $X^2$, with
$$X=   \begin{pmatrix} x&\partial_x\cr
-\partial_x& -x \cr \end{pmatrix}.$$ The functional calculus of $\tB_t^2$ is then reuced to the functional calculus of the operator $$
\tD' + \frac{1}{t} \lambda_k, \quad\text{with}\quad
\tD'=\begin{pmatrix} \tD^2 & 0 \cr
0 & \tD^2 \cr \end{pmatrix}
$$
and where  $\lambda_k$ is an eigenvalue of $X^2$ as in \cite{keswani3}  . Now the $L^1$-norm $\|\chi(\tB_s) -  \chi (\epsilon\tB_s)\|_1$ is given by the sum over $k$ of $L^1$-norms in corresponding von Neumann algebras of the operator  $(\chi -  \chi_\epsilon) (\tD' + \lambda_k)$. By \cite{FackKosaki}, this $L^1$-norm is expressed in terms of the singular numbers $\mu^\nu_s(\tD' + \lambda_k)= \mu^\nu_s (\tD') + \lambda_k$. This reduces the estimate to the similar estimate of the singular numbers of $\tD'$ exactly as in \cite{keswani3}. This latter being a leafwise elliptic second order differential operator, we can use the estimate $\mu_s(\tD') \sim s^{2/p}$ where $p$ is the dimension of the leaves, see for instance \cite{BenameurFack}. Hence the proof of the first inequality of the claim is completed following the steps of \cite{keswani3}. 
 The proof of the second inequality
in the Claim is similar. Thus we have proved the Claim and thus \eqref{limit-gamma-1}
in the case of cylinders. For manifolds with cylindrical ends we split the relevant
statements into  purely cylindrical ones and  statements on  compact foliated bundles,
as in \cite{keswani3}. We end here our explanation of the
 proof of  \eqref{limit-gamma-1}.  The proof of Theorem \ref{homotopy-invariance}
 is now complete.

\end{proof}

{\small
\bibliographystyle{plain}
\bibliography{ierofb-arxiv}

\begin{thebibliography}{10}

\bibitem{Atiyah-covering}
M.~F. Atiyah.
\newblock Elliptic operators, discrete groups and von {N}eumann algebras.
\newblock In {\em Colloque ``Analyse et Topologie'' en l'Honneur de Henri
  Cartan (Orsay, 1974)}, pages 43--72. Ast\'erisque, No. 32--33. Soc. Math.
  France, Paris, 1976.

\bibitem{APS1}
M.~F. Atiyah, V.~K. Patodi, and I.~M. Singer.
\newblock Spectral asymmetry and {R}iemannian geometry. {I}.
\newblock {\em Math. Proc. Cambridge Philos. Soc.}, 77:43--69, 1975.

\bibitem{APS2}
M.~F. Atiyah, V.~K. Patodi, and I.~M. Singer.
\newblock Spectral asymmetry and {R}iemannian geometry. {II}.
\newblock {\em Math. Proc. Cambridge Philos. Soc.}, 78(3):405--432, 1975.

\bibitem{BC-enseignement}
Paul Baum and Alain Connes.
\newblock Geometric {$K$}-theory for {L}ie groups and foliations.
\newblock {\em Enseign. Math. (2)}, 46(1-2):3--42, 2000.

\bibitem{BCH}
Paul Baum, Alain Connes, and Nigel Higson.
\newblock Classifying space for proper actions and {$K$}-theory of group {$C\sp
  \ast$}-algebras.
\newblock In {\em $C\sp \ast$-algebras: 1943--1993 (San Antonio, TX, 1993)},
  volume 167 of {\em Contemp. Math.}, pages 240--291. Amer. Math. Soc.,
  Providence, RI, 1994.

\bibitem{BHS}
Paul Baum, Nigel Higson, and Thomas Schick.
\newblock On the equivalence of geometric and analytic {$K$}-homology.
\newblock {\em Pure Appl. Math. Q.}, 3(1):1--24, 2007.

\bibitem{Moulay-triangulation}
Moulay-Tahar Benameur.
\newblock Triangulations and the stability theorem for foliations.
\newblock {\em Pacific J. Math.}, 179(2):221--239, 1997.

\bibitem{BenameurFack}
Moulay-Tahar Benameur and Thierry Fack.
\newblock Type {II} non-commutative geometry. {I}. {D}ixmier trace in von
  {N}eumann algebras.
\newblock {\em Adv. Math.}, 199(1):29--87, 2006.

\bibitem{BH3}
Moulay-Tahar Benameur and James Heitsch.
\newblock The higher harmonic signature for foliations.
\newblock preprint, submitted.

\bibitem{Benameur-Nistor-JFA}
Moulay-Tahar Benameur and Victor Nistor.
\newblock Homology of algebras of families of pseudodifferential operators.
\newblock {\em J. Funct. Anal.}, 205(1):1--36, 2003.

\bibitem{BenameurOyono}
Moulay-Tahar Benameur and Herv{\'e} Oyono-Oyono.
\newblock Index theory for quasi-crystals. {I}. {C}omputation of the gap-label
  group.
\newblock {\em J. Funct. Anal.}, 252(1):137--170, 2007.

\bibitem{Bismut-inventiones}
Jean-Michel Bismut.
\newblock The {A}tiyah-{S}inger index theorem for families of {D}irac
  operators: two heat equation proofs.
\newblock {\em Invent. Math.}, 83(1):91--151, 1985.

\bibitem{BF2}
Jean-Michel Bismut and Daniel~S. Freed.
\newblock The analysis of elliptic families. {II}. {D}irac operators, eta
  invariants, and the holonomy theorem.
\newblock {\em Comm. Math. Phys.}, 107(1):103--163, 1986.

\bibitem{Chang-W}
Stanley Chang and Shmuel Weinberger.
\newblock {On Invariants of Hirzebruch and Cheeger-Gromov}.
\newblock {\em Geom. Topol.}, 7:311--319, 2003.

\bibitem{Cheeger-Gromov-JDG}
Jeff Cheeger and Mikhael Gromov.
\newblock Bounds on the von {N}eumann dimension of {$L\sp 2$}-cohomology and
  the {G}auss-{B}onnet theorem for open manifolds.
\newblock {\em J. Differential Geom.}, 21(1):1--34, 1985.

\bibitem{ConnesSkandalis}
A.~Connes and G.~Skandalis.
\newblock The longitudinal index theorem for foliations.
\newblock {\em Publ. Res. Inst. Math. Sci.}, 20(6):1139--1183, 1984.

\bibitem{Co-LNM}
Alain Connes.
\newblock Sur la th\'eorie non commutative de l'int\'egration.
\newblock In {\em Alg\`ebres d'op\'erateurs (S\'em., Les Plans-sur-Bex, 1978)},
  volume 725 of {\em Lecture Notes in Math.}, pages 19--143. Springer, Berlin,
  1979.

\bibitem{Co}
Alain Connes.
\newblock {\em Noncommutative geometry}.
\newblock Academic Press Inc., San Diego, CA, 1994.

\bibitem{dlH-Ska}
P.~de~la Harpe and G.~Skandalis.
\newblock D\'eterminant associ\'e \`a une trace sur une alg\'ebre de {B}anach.
\newblock {\em Ann. Inst. Fourier (Grenoble)}, 34(1):241--260, 1984.

\bibitem{Dixmier2}
Jacques Dixmier.
\newblock {\em Les alg\`ebres d'op\'erateurs dans l'espace hilbertien
  (alg\`ebres de von {N}eumann)}.
\newblock Les Grands Classiques Gauthier-Villars. [Gauthier-Villars Great
  Classics]. \'Editions Jacques Gabay, Paris, 1996.
\newblock Reprint of the second (1969) edition.

\bibitem{Dixmier1}
Jacques Dixmier.
\newblock {\em Les {$C\sp *$}-alg\`ebres et leurs repr\'esentations}.
\newblock Les Grands Classiques Gauthier-Villars. [Gauthier-Villars Great
  Classics]. \'Editions Jacques Gabay, Paris, 1996.
\newblock Reprint of the second (1969) edition.

\bibitem{FackKosaki}
Thierry Fack and Hideki Kosaki.
\newblock Generalized {$s$}-numbers of {$\tau$}-measurable operators.
\newblock {\em Pacific J. Math.}, 123(2):269--300, 1986.

\bibitem{Go-Lo}
Alexander Gorokhovsky and John Lott.
\newblock Local index theory over \'etale groupoids.
\newblock {\em J. Reine Angew. Math.}, 560:151--198, 2003.

\bibitem{Hei-Laz}
James~L. Heitsch and Connor Lazarov.
\newblock A {L}efschetz theorem for foliated manifolds.
\newblock {\em Topology}, 29(2):127--162, 1990.

\bibitem{HL-betti}
James~L. Heitsch and Connor Lazarov.
\newblock Homotopy invariance of foliation {B}etti numbers.
\newblock {\em Invent. Math.}, 104(2):321--347, 1991.

\bibitem{Hilsum-Skandalis-stabilite}
Michel Hilsum and Georges Skandalis.
\newblock Morphismes {$K$}-orient\'es d'espaces de feuilles et fonctorialit\'e
  en th\'eorie de {K}asparov (d'apr\`es une conjecture d'{A}. {C}onnes).
\newblock {\em Ann. Sci. \'Ecole Norm. Sup. (4)}, 20(3):325--390, 1987.

\bibitem{HiSka}
Michel Hilsum and Georges Skandalis.
\newblock Invariance par homotopie de la signature \`a coefficients dans un
  fibr\'e presque plat.
\newblock {\em J. Reine Angew. Math.}, 423:73--99, 1992.

\bibitem{keswani3}
Navin Keswani.
\newblock Geometric {$K$}-homology and controlled paths.
\newblock {\em New York J. Math.}, 5:53--81 (electronic), 1999.

\bibitem{Kes1}
Navin Keswani.
\newblock Relative eta-invariants and {$C\sp \ast$}-algebra {$K$}-theory.
\newblock {\em Topology}, 39(5):957--983, 2000.

\bibitem{Kes2}
Navin Keswani.
\newblock Von {N}eumann eta-invariants and {$C\sp *$}-algebra {$K$}-theory.
\newblock {\em J. London Math. Soc. (2)}, 62(3):771--783, 2000.

\bibitem{LaMoNi-Documenta}
Robert Lauter, Bertrand Monthubert, and Victor Nistor.
\newblock Pseudodifferential analysis on continuous family groupoids.
\newblock {\em Doc. Math.}, 5:625--655 (electronic), 2000.

\bibitem{LPETALE}
Eric Leichtnam and Paolo Piazza.
\newblock \'{E}tale groupoids, eta invariants and index theory.
\newblock {\em J. Reine Angew. Math.}, 587:169--233, 2005.

\bibitem{Lenz-et-al}
Daniel Lenz, Norbert Peyerimhoff, and Ivan Veseli{\'c}.
\newblock Groupoids, von {N}eumann algebras and the integrated density of
  states.
\newblock {\em Math. Phys. Anal. Geom.}, 10(1):1--41, 2007.

\bibitem{Mathai-JFA}
Varghese Mathai.
\newblock Spectral flow, eta invariants, and von {N}eumann algebras.
\newblock {\em J. Funct. Anal.}, 109(2):442--456, 1992.

\bibitem{Melrose}
Richard~B. Melrose.
\newblock {\em The {A}tiyah-{P}atodi-{S}inger index theorem}, volume~4 of {\em
  Research Notes in Mathematics}.
\newblock A K Peters Ltd., Wellesley, MA, 1993.

\bibitem{MS}
Calvin~C. Moore and Claude~L. Schochet.
\newblock {\em Global analysis on foliated spaces}, volume~9 of {\em
  Mathematical Sciences Research Institute Publications}.
\newblock Cambridge University Press, New York, second edition, 2006.

\bibitem{Moriyoshi-Natsume}
H.~Moriyoshi and T.~Natsume.
\newblock The godbillon-vey cyclic cocycle and longitudinal dirac operators.
\newblock {\em Pacific J. Math.}, 172:483--539, 1996.

\bibitem{Moscovici-Wu-GAFA}
H.~Moscovici and F.-B. Wu.
\newblock Localization of topological {P}ontryagin classes via finite
  propagation speed.
\newblock {\em Geom. Funct. Anal.}, 4(1):52--92, 1994.

\bibitem{Neumann}
Walter~D. Neumann.
\newblock Signature related invariants of manifolds. {I}. {M}onodromy and
  {$\gamma $}-invariants.
\newblock {\em Topology}, 18(2):147--172, 1979.

\bibitem{NWX}
Victor Nistor, Alan Weinstein, and Ping Xu.
\newblock Pseudodifferential operators on differential groupoids.
\newblock {\em Pacific J. Math.}, 189(1):117--152, 1999.

\bibitem{Peric}
Goran Peri{\'c}.
\newblock Eta invariants of {D}irac operators on foliated manifolds.
\newblock {\em Trans. Amer. Math. Soc.}, 334(2):761--782, 1992.

\bibitem{phillips-homotopy}
John Phillips.
\newblock The holonomic imperative and the homotopy groupoid of a foliated
  manifold.
\newblock {\em Rocky Mountain J. Math.}, 17(1):151--165, 1987.

\bibitem{Pia-Sch1}
Paolo Piazza and Thomas Schick.
\newblock Bordism, rho-invariants and the {B}aum-{C}onnes conjecture.
\newblock {\em J. Noncommut. Geom.}, 1(1):27--111, 2007.

\bibitem{Pia-Sch2}
Paolo Piazza and Thomas Schick.
\newblock Groups with torsion, bordism and rho invariants.
\newblock {\em Pacific J. Math.}, 232(2):355--378, 2007.

\bibitem{Ra}
Mohan Ramachandran.
\newblock von {N}eumann index theorems for manifolds with boundary.
\newblock {\em J. Differential Geom.}, 38(2):315--349, 1993.

\bibitem{Raven}
Jeff Raven.
\newblock An equivariant bivariant chern character.
\newblock Ph. Thesis 2004, Pennsylvania State University.

\bibitem{Renault}
Jean Renault.
\newblock {\em A groupoid approach to {$C\sp{\ast} $}-algebras}, volume 793 of
  {\em Lecture Notes in Mathematics}.
\newblock Springer, Berlin, 1980.

\bibitem{roe-foliation}
John Roe.
\newblock Finite propagation speed and {C}onnes' foliation algebra.
\newblock {\em Math. Proc. Cambridge Philos. Soc.}, 102(3):459--466, 1987.

\bibitem{roe-partitioning}
John Roe.
\newblock Partitioning noncompact manifolds and the dual {T}oeplitz problem.
\newblock In {\em Operator algebras and applications, Vol.\ 1}, volume 135 of
  {\em London Math. Soc. Lecture Note Ser.}, pages 187--228. Cambridge Univ.
  Press, Cambridge, 1988.

\bibitem{roe-book}
John Roe.
\newblock {\em Elliptic operators, topology and asymptotic methods}, volume 395
  of {\em Pitman Research Notes in Mathematics Series}.
\newblock Longman, Harlow, second edition, 1998.

\bibitem{shubin}
M.~A. Shubin.
\newblock {\em Pseudodifferential operators and spectral theory}.
\newblock Springer Series in Soviet Mathematics. Springer-Verlag, Berlin, 1987.
\newblock Translated from the Russian by Stig I. Andersson.

\bibitem{taylor-pso}
Michael~E. Taylor.
\newblock {\em Pseudodifferential operators}, volume~34 of {\em Princeton
  Mathematical Series}.
\newblock Princeton University Press, Princeton, N.J., 1981.

\bibitem{Vassout-these}
Stephan Vassout.
\newblock Feuilletages et r\'esidu non commutatif longitudinal.
\newblock Ph.D Thesis, 2001, Universit\' e Paris 6.

\bibitem{Vassout-jfa}
St{\'e}phane Vassout.
\newblock Unbounded pseudodifferential calculus on {L}ie groupoids.
\newblock {\em J. Funct. Anal.}, 236(1):161--200, 2006.

\bibitem{Wei}
Shmuel Weinberger.
\newblock Homotopy invariance of {$\eta$}-invariants.
\newblock {\em Proc. Nat. Acad. Sci. U.S.A.}, 85(15):5362--5363, 1988.

\end{thebibliography}

\end{document}